\documentclass[11pt,reqno,twoside]{article}
\usepackage{amsmath,amsfonts,amsthm,amssymb}
\usepackage{graphics}
\usepackage[all]{xy}
\usepackage{color}
\theoremstyle{definition}
\newtheorem{theorem}{Theorem}[section]
\newtheorem{proposition}[theorem]{Proposition}

\newtheorem{lemma}[theorem]{Lemma}
\newtheorem{remark}[theorem]{Remark}

\newtheorem{definition}[theorem]{Definition}
\newtheorem{example}[theorem]{Example}

\begin{document}


\title{\normalsize\bf POLYNOMIAL OF AN ORIENTED SURFACE-LINK DIAGRAM VIA QUANTUM $A_2$ INVARIANT}

\author{\small 
YEWON JOUNG
\smallskip\\
{\small\it 
Department of Mathematics, Pusan National University, 
}\\ 
{\small\it 
Busan 46241, Korea
}\\
{\small\it 
yewon112@pusan.ac.kr
}
\smallskip\\
{\small SEIICHI KAMADA}
\smallskip\\
{\small\it 
Department of Mathematics, Osaka City University,
}\\ 
{\small\it Osaka 558-8585, Japan}\\
{\small\it skamada@sci.osaka-cu.ac.jp}
\smallskip\\
{\small AKIO KAWAUCHI}
\smallskip\\
{\small\it 
Osaka City University Advanced Mathematical Institute, Osaka City University
}\\ 
{\small\it Osaka 558-8585, Japan}\\
{\small\it kawauchi@sci.osaka-cu.ac.jp}
\smallskip\\
and
\smallskip\\
{\small SANG YOUL LEE}
\smallskip\\
{\small\it 
Department of Mathematics, Pusan National University,
}\\ 
{\small\it Busan 46241, Korea}\\
{\small\it sangyoul@pusan.ac.kr}}

\renewcommand\leftmark{\centerline{\footnotesize 
Y. Joung, S. Kamada, A. Kawauchi \& S. Y. Lee}}
\renewcommand\rightmark{\centerline{\footnotesize 
Polynomial of an oriented surface-link diagram via quantum $A_2$ invariant}}

\maketitle

\begin{abstract}
It is known that every surface-link can be presented by a marked graph
diagram, and such a diagram presentation is unique up to moves called
Yoshikawa moves. G. Kuperberg introduced a regular isotopy invariant,
called the quantum $A_2$ invariant, for tangled trivalent graph
diagrams. In this paper, a polynomial for a marked graph diagram is
defined by use of the quantum $A_2$ invariant and it is studied
how the polynomial changes under Yoshikawa moves. The notion of a
ribbon marked graph is introduced to show that this polynomial
is useful for an invariant of a ribbon 2-knot. 
\end{abstract}





\section{Introduction}
\label{intro}

A {\it marked graph diagram} (or {\it ch-diagram}) is a link diagram possibly with some $4$-valent vertices equipped with markers; \xy (-4,4);(4,-4) **@{-}, 
(4,4);(-4,-4) **@{-},  
(3,-0.2);(-3,-0.2) **@{-},
(3,0);(-3,0) **@{-}, 
(3,0.2);(-3,0.2) **@{-}, 
\endxy.
An {\it oriented marked graph diagram} is a marked graph diagram in which every edge has an orientation such that each marked vertex looks like
\xy (-4,4);(4,-4) **@{-}, 
(4,4);(-4,-4) **@{-}, 
(3,3.2)*{\llcorner}, 
(-3,-3.4)*{\urcorner}, 
(-2.5,2)*{\ulcorner},
(2.5,-2.4)*{\lrcorner}, 
(3,-0.2);(-3,-0.2) **@{-},
(3,0);(-3,0) **@{-}, 
(3,0.2);(-3,0.2) **@{-}, 
\endxy. 
It is known that a surface-link is presented by a marked graph diagram 
(cf. \cite{Lo, Yo}),
and such a presentation diagram is unique up to Yoshikawa moves 
(cf. \cite{KK, Sw}). 
See Section \ref{sect-omgr-osl} for details. By using marked graph diagrams, some properties and invariants of surface-links were studied in \cite{As,JKaL,JKL,KKL,KJL1,KJL2,Le1,Le2,Le3,Le4,So,Yo}.

A {\it tangled trivalent graph diagram} is an oriented link diagram possibly with some trivalent vertices whose incident edges are oriented all inward or all outward. In \cite{Kup}, G. Kuperberg introduced a regular isotopy  invariant $\langle\cdot\rangle_{A_2}$, called the $A_2$ bracket (polynomial), for  tangled trivalent graph diagrams, which is derived from the Reshetikhin-Turaev quantum invariant (cf. \cite{RT}) corresponding to the simple Lie algebra $A_2$. 

In \cite{Le3}, the fourth author introduced a method of constructing invariant for a surface-link by means of a marked graph diagram and a state-sum model associated to a 
classical link invariant as its state evaluation. In this paper, 
we define a polynomial 
in $\mathbb Z[a^{-1}, a, x,y]$ for an oriented marked graph diagram by using the $A_2$ bracket $\langle\cdot\rangle_{A_2}$  in the line of \cite{Le3} and study how the 
polynomial changes under Yoshikawa moves.  In the process of this argument, 
the notion of a ribbon marked graph is introduced to show 
that this polynomial is useful for an invariant of a ribbon 2-knot. 

This paper is organized as follows. In Section \ref{sect-omgr-osl}, we review marked graphs and their presenting surface-links. In Section \ref{sect-qa2-inv}, we recall the quantum $A_2$ invariant $\langle\cdot\rangle_{A_2}$ for link diagrams and tangled trivalent graph diagrams. In Section \ref{sect-poly-osl}, we define a Laurent polynomial 
$\ll D \gg (a,x,y) \in \mathbb Z[a^{-1},a,x,y]$ for an oriented marked graph diagram 
$D$. In Section~\ref{sect-on-inv-osl}, we study how the polynomial  $\ll \cdot \gg$ 
changes under Yoshikawa moves $\Gamma_6$,  $\Gamma'_6$, $\Gamma_7$ and $\Gamma_8$.  
In Section~\ref{sect-special}, we discuss specializations of the invariant by 
considering some quotients of the ring $\mathbb Z[a^{-1},a,x,y]$.  
In Section~\ref{sect:ribbon}, the notion of a ribbon marked graph is introduced to 
derive an invariant of ribbon $2$-knots from the polynomial.  
 In Sections~\ref{sect-pf-prop-e3t} and \ref{sect-pf-prop-e4t}, we prove key lemmas 
used in Section~\ref{sect-on-inv-osl}. 


\section{Marked graphs and surface-links}\label{sect-omgr-osl}

In this section, we review marked graphs and their presenting surface-links. A {\it marked graph} is a spatial graph $G$ in $\mathbb R^3$ which satisfies the following:
\begin{itemize}
  \item $G$ is a finite regular graph with $4$-valent vertices, say $v_1, v_2, . . . , v_n$.
  \item Each $v_i$ is a rigid vertex; that is, we fix a rectangular neighborhood $N_i$ homeomorphic to $\{(x, y)|-1 \leq x, y \leq 1\},$ where $v_i$ corresponds to the origin and the edges incident to $v_i$ are represented by $x^2 = y^2$.
  \item Each $v_i$ has a {\it marker}, which is the interval on $N_i$ given by  $\{(x, 0)|-1 \leq x \leq 1\}$.
\end{itemize}

An {\it orientation} of a marked graph $G$ is a choice of an orientation for each edge of $G$ in such a way that every vertex in $G$ looks like 
\xy (-5,5);(5,-5) **@{-}, 
(5,5);(-5,-5) **@{-}, 
(3,3.2)*{\llcorner}, 
(-3,-3.4)*{\urcorner}, 
(-2.5,2)*{\ulcorner},
(2.5,-2.4)*{\lrcorner}, 
(3,-0.2);(-3,-0.2) **@{-},
(3,0);(-3,0) **@{-}, 
(3,0.2);(-3,0.2) **@{-}, 
\endxy.
A marked graph $G$ is said to be 
{\it orientable} if it admits an orientation. Otherwise, it is said to be {\it non-orientable}. By an {\it oriented marked graph} we mean an orientable marked graph with a fixed orientation. Two oriented marked graphs are said to be {\it equivalent} if they are ambient isotopic in $\mathbb R^3$ with keeping the rectangular neighborhoods, markers and the orientation. As usual, a marked graph can be described by a diagram in $\mathbb R^2$, which is a link diagram with some $4$-valent vertices equipped with markers (see Figure~\ref{fig-nori-mg}). 

Two marked graph diagrams present equivalent marked graphs if and only if they are related by a finite sequence of Yoshikawa moves $\Gamma_1, \Gamma'_1, \Gamma_2, \Gamma_3, \Gamma_4, \Gamma'_4$ and $\Gamma_5$ depicted in Figure~\ref{fig-moves-type-II-o}.

\begin{figure}[ht]
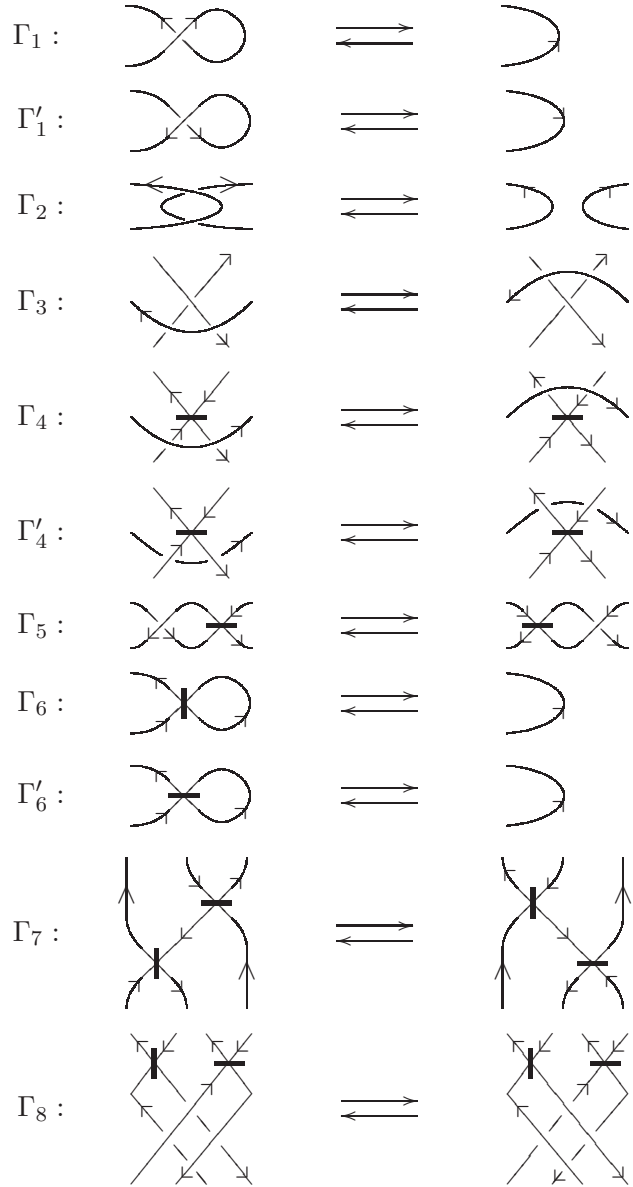

\centerline{\xy 
(12,2);(16,6) **@{-}, 
(12,6);(13.5,4.5) **@{-},
(14.5,3.5);(16,2) **@{-}, 
(16,6);(22,6) **\crv{(18,8)&(20,8)},
(16,2);(22,2) **\crv{(18,0)&(20,0)}, (22,6);(22,2) **\crv{(23.5,4)},
(7,8);(12,6) **\crv{(10,8)}, (7,0);(12,2) **\crv{(10,0)}, (12.4,5) *{\ulcorner}, (15.6,5) *{\urcorner},
(35,5);(45,5) **@{-} ?>*\dir{>}, (35,3);(45,3) **@{-} ?<*\dir{<},
(63.8,2) *{\urcorner},
(57,8);(57,0) **\crv{(67,7)&(67,1)}, (-5,4)*{\Gamma_1 :}, (73,4)*{},
\endxy}
\vskip.3cm
\centerline{ 
\xy (12,2);(16,6) **@{-}, 
(12,6);(13.5,4.5) **@{-},
(14.5,3.5);(16,2) **@{-}, 
(16,6);(22,6) **\crv{(18,8)&(20,8)},
(16,2);(22,2) **\crv{(18,0)&(20,0)}, (22,6);(22,2) **\crv{(23.5,4)},
(7,8);(12,6) **\crv{(10,8)}, (7,0);(12,2) **\crv{(10,0)}, (12.4,2.5) *{\llcorner}, (15.6,2.5) *{\lrcorner},
(35,5);(45,5) **@{-} ?>*\dir{>}, (35,3);(45,3) **@{-} ?<*\dir{<},(63.85,5.2) *{\lrcorner},
(57,8);(57,0) **\crv{(67,7)&(67,1)}, (-5,4)*{\Gamma'_1 :}, (73,4)*{},
\endxy}
\vskip.3cm
\centerline{ \xy (7,7);(7,1)  **\crv{(23,6)&(23,2)}, (16,6.3);(23,7)
**\crv{(19,6.9)}, (16,1.7);(23,1) **\crv{(19,1.1)},
(14,5.7);(14,2.3) **\crv{(8,4)}, (10,6.9) *{<}, (20,6.9) *{>},
(35,5);(45,5) **@{-} ?>*\dir{>}, (35,3);(45,3) **@{-} ?<*\dir{<},
(57,7);(57,1) **\crv{(65,6)&(65,2)}, (73,7);(73,1)
**\crv{(65,6)&(65,2)}, (60,5.3) *{\ulcorner}, (70,5.3) *{\urcorner}, (-5,4)*{\Gamma_2 :},
\endxy}
\vskip.3cm
\centerline{ 
\xy (7,6);(23,6) **\crv{(15,-2)}, 
(10,0);(11.5,1.8) **@{-}, 
(17.5,3);(14.5,6.6) **@{-},
(14.5,6.6);(10,12) **@{-}, 
(20,12);(15.5,6.6) **@{-},
(14.5,5.5);(12.5,3) **@{-},
(18.5,1.8);(20,0) **@{-},
(19.5,11) *{\urcorner}, 
(19,1.1) *{\lrcorner},  
(9,3.5) *{\ulcorner},
(35,7);(45,7) **@{-} ?>*\dir{>}, 
(35,5);(45,5) **@{-} ?<*\dir{<},
(57,6);(73,6) **\crv{(65,14)}, 
(70,12);(68.5,10.2) **@{-}, 
(67.5,9);(65.5,6.5) **@{-}, 
(64.6,5.5);(60,0) **@{-}, 
(62.5,9);(64.4,6.6) **@{-}, 
(64.4,6.6);(70,0) **@{-}, 
(61.5,10.2);(60,12) **@{-},
(69.5,11) *{\urcorner}, 
(69,1.1) *{\lrcorner},  
(58,7) *{\llcorner},
(-5,6)*{\Gamma_3:},
\endxy}
\vskip.3cm
 \centerline{ \xy 
 (7,6);(23,6)  **\crv{(15,-2)}, 
 (10,0);(11.5,1.8) **@{-},
(12.5,3);(20,12) **@{-}, 
(10,12);(17.5,3) **@{-}, 
(18.5,1.8);(20,0) **@{-}, 
(13,6);(17,6) **@{-}, (13,6.1);(17,6.1) **@{-}, (13,5.9);(17,5.9)
**@{-}, (13,6.2);(17,6.2) **@{-}, (13,5.8);(17,5.8) **@{-},
  (13,3.1) *{\urcorner}, (17.5,8.9) *{\llcorner}, (19,1.1) *{\lrcorner},  (21,3.5) *{\urcorner},(13,8) *{\ulcorner},
(35,7);(45,7) **@{-} ?>*\dir{>}, 
(35,5);(45,5) **@{-} ?<*\dir{<},
(57,6);(73,6)  **\crv{(65,14)}, 
(70,12);(68.5,10.2) **@{-},
(67.5,9);(60,0) **@{-}, 
(70,0);(62.5,9) **@{-}, 
(61.5,10.2);(60,12) **@{-}, 
(63,6);(67,6) **@{-}, (63,6.1);(67,6.1) **@{-}, (63,5.9);(67,5.9)
**@{-}, (63,6.2);(67,6.2) **@{-}, (63,5.8);(67,5.8) **@{-},
(62,2) *{\urcorner}, (67,8.3) *{\llcorner}, (67.5,3) *{\lrcorner},  (70.9,8) *{\lrcorner},(61.3,10) *{\ulcorner}, 
(-5,6)*{\Gamma_4:},
\endxy}
\vskip.3cm
 \centerline{ \xy 
  (13,2.2);(17,2.2)  **\crv{(15,1.7)}, 
  (7,6);(11,3)  **\crv{(10,3.5)}, 
  (23,6);(19,3)  **\crv{(20,3.5)}, 
 (10,0);(20,12) **@{-}, 
(10,12);(20,0) **@{-}, 
(13,6);(17,6) **@{-}, (13,6.1);(17,6.1) **@{-}, (13,5.9);(17,5.9)
**@{-}, (13,6.2);(17,6.2) **@{-}, (13,5.8);(17,5.8) **@{-}, 
(13,3.1) *{\urcorner}, (17.5,8.9) *{\llcorner}, (19,1.1) *{\lrcorner},  (21,3.5) *{\urcorner},(13,8) *{\ulcorner},
(35,7);(45,7) **@{-} ?>*\dir{>}, 
(35,5);(45,5) **@{-} ?<*\dir{<},
   (63,9.8);(67,9.8)  **\crv{(65,10.3)}, 
  (57,6);(61,9)  **\crv{(60,8.5)}, 
  (73,6);(69,9)  **\crv{(70,8.5)}, 
(70,12);(60,0) **@{-}, 
(70,0);(60,12) **@{-}, 
(63,6);(67,6) **@{-}, (63,6.1);(67,6.1) **@{-}, (63,5.9);(67,5.9)
**@{-},(63,6.2);(67,6.2) **@{-}, (63,5.8);(67,5.8) **@{-}, 
(62,2) *{\urcorner}, (67,8.3) *{\llcorner}, (67.5,3) *{\lrcorner},  (70.9,8) *{\lrcorner},(61.3,10) *{\ulcorner}, 
(-5,6)*{\Gamma'_4:},
\endxy}
\vskip.3cm
\centerline{ \xy (9,2);(13,6) **@{-}, (9,6);(10.5,4.5) **@{-},
(11.5,3.5);(13,2) **@{-}, (17,2);(21,6) **@{-}, (17,6);(21,2)
**@{-}, (13,6);(17,6) **\crv{(15,8)}, (13,2);(17,2) **\crv{(15,0)},
(7,7);(9,6) **\crv{(8,7)}, (7,1);(9,2) **\crv{(8,1)}, (23,7);(21,6)
**\crv{(22,7)}, (23,1);(21,2) **\crv{(22,1)}, 
(17,4);(21,4) **@{-}, (17,4.1);(21,4.1) **@{-}, (17,3.9);(21,3.9)
**@{-}, (17,4.2);(21,4.2) **@{-}, (17,3.8);(21,3.8) **@{-},
(10,3) *{\llcorner},  (12,3) *{\lrcorner}, (21,6) *{\llcorner},(21,2.2) *{\lrcorner},
(35,5);(45,5) **@{-} ?>*\dir{>}, (35,3);(45,3) **@{-} ?<*\dir{<},
(59,2);(63,6) **@{-}, (59,6);(63,2) **@{-}, (67,2);(71,6) **@{-},
(67,6);(68.5,4.5) **@{-}, (69.5,3.5);(71,2) **@{-}, (63,6);(67,6)
**\crv{(65,8)}, (63,2);(67,2) **\crv{(65,0)}, (57,7);(59,6)
**\crv{(58,7)}, (57,1);(59,2) **\crv{(58,1)}, (73,7);(71,6)
**\crv{(72,7)}, (73,1);(71,2) **\crv{(72,1)}, 
(63,4);(59,4) **@{-}, (63,4.1);(59,4.1) **@{-}, (63,3.9);(59,3.9)
**@{-}, (63,4.2);(59,4.2) **@{-}, (63,3.8);(59,3.8) **@{-},
(59.5,2.5) *{\llcorner},  (59,6) *{\lrcorner}, (71,6) *{\llcorner},(71,2.2) *{\lrcorner},
 (-5,4)*{\Gamma_5:},
\endxy}
\vskip.3cm
\centerline{ 
\xy (12,6);(16,2) **@{-}, (12,2);(16,6) **@{-},
(16,6);(22,6) **\crv{(18,8)&(20,8)}, (16,2);(22,2)
**\crv{(18,0)&(20,0)}, (22,6);(22,2) **\crv{(23.5,4)}, (7,8);(12,6)
**\crv{(10,8)}, (7,0);(12,2) **\crv{(10,0)}, (11,0.4) *{\urcorner}, (11,6) *{\ulcorner},(21.5,1)*{\urcorner},
(35,5);(45,5) **@{-} ?>*\dir{>}, (35,3);(45,3) **@{-} ?<*\dir{<},
(57,8);(57,0) **\crv{(67,7)&(67,1)}, (-5,4)*{\Gamma_6 :}, (73,4)*{},
(14,6);(14,2) **@{-}, (14.1,6);(14.1,2) **@{-}, (13.9,6);(13.9,2)
**@{-}, (14.2,6);(14.2,2) **@{-}, (13.8,6);(13.8,2) **@{-}, 
(63.8,2) *{\urcorner},
\endxy}
\vskip.3cm
\centerline{ \xy (12,6);(16,2) **@{-}, (12,2);(16,6) **@{-},
(16,6);(22,6) **\crv{(18,8)&(20,8)}, (16,2);(22,2)
**\crv{(18,0)&(20,0)}, (22,6);(22,2) **\crv{(23.5,4)}, (7,8);(12,6)
**\crv{(10,8)}, (7,0);(12,2) **\crv{(10,0)}, (11,0.4) *{\urcorner}, (11,6) *{\ulcorner},(21.5,1)*{\urcorner},
(35,5);(45,5) **@{-} ?>*\dir{>}, (35,3);(45,3) **@{-} ?<*\dir{<},
(57,8);(57,0) **\crv{(67,7)&(67,1)}, (-5,4)*{\Gamma'_6 :},
(73,4)*{}, (12,4);(16,4) **@{-}, (12,4.1);(16,4.1) **@{-},
(12,4.2);(16,4.2) **@{-}, (12,3.9);(16,3.9) **@{-},
(12,3.8);(16,3.8) **@{-}, (63.8,2) *{\urcorner},
\endxy}
\vskip.3cm
\centerline{ \xy (9,4);(17,12) **@{-}, (9,8);(13,4) **@{-},
(17,12);(21,16) **@{-}, (17,16);(21,12) **@{-}, (7,0);(9,4)
**\crv{(7,2)}, (7,12);(9,8) **\crv{(7,10)}, (15,0);(13,4)
**\crv{(15,2)}, (17,16);(15,20) **\crv{(15,18)}, (21,16);(23,20)
**\crv{(23,18)}, (21,12);(23,8) **\crv{(23,10)}, (7,12);(7,20)
**@{-}, (23,8);(23,0) **@{-},
(11,4);(11,8) **@{-}, 
(10.9,4);(10.9,8) **@{-}, 
(11.1,4);(11.1,8) **@{-}, 
(10.8,4);(10.8,8) **@{-}, 
(11.2,4);(11.2,8) **@{-},
(17,14);(21,14) **@{-}, 
(17,14.1);(21,14.1) **@{-},
(17,13.9);(21,13.9) **@{-}, 
(17,14.2);(21,14.2) **@{-},
(17,13.8);(21,13.8) **@{-},
(7,15) *{\wedge},(23,5) *{\wedge},(15,10) *{\llcorner},(8,2.3) *{\urcorner}, (21.5,16) *{\urcorner},(16,17) *{\lrcorner},(13.7,3) *{\lrcorner},
(35,11);(45,11) **@{-} ?>*\dir{>}, (35,9);(45,9) **@{-} ?<*\dir{<},
(71,4);(63,12) **@{-}, (71,8);(67,4) **@{-}, (63,12);(59,16) **@{-},
(63,16);(59,12) **@{-}, (73,0);(71,4) **\crv{(73,2)}, (73,12);(71,8)
**\crv{(73,10)}, (65,0);(67,4) **\crv{(65,2)}, (63,16);(65,20)
**\crv{(65,18)}, (59,16);(57,20) **\crv{(57,18)}, (59,12);(57,8)
**\crv{(57,10)}, (73,12);(73,20) **@{-}, (57,8);(57,0) **@{-},
(61,12);(61,16) **@{-}, 
(61.1,12);(61.1,16) **@{-},
(60.9,12);(60.9,16) **@{-}, 
(61.2,12);(61.2,16) **@{-},
(60.8,12);(60.8,16) **@{-},
(57,5) *{\wedge},(73,15) *{\wedge},(65,10) *{\lrcorner},(58,17) *{\ulcorner}, (71.5,3) *{\ulcorner},(66.3,3) *{\llcorner},(63.7,16.8) *{\llcorner},
(67,6);(71,6) **@{-}, 
(67,6.1);(71,6.1) **@{-}, 
(67,5.9);(71,5.9) **@{-}, 
(67,6.2);(71,6.2) **@{-},
(67,5.8);(71,5.8) **@{-},  
(-5,10)*{\Gamma_7:}, 
 \endxy}
\vskip.3cm
\centerline{ 
\xy (7,20);(14.2,11) **@{-}, (15.8,9);(17.4,7) **@{-},
(19,5);(23,0) **@{-}, (13,20);(7,12) **@{-}, (7,12);(11.2,7) **@{-},
(12.7,5.2);(14.4,3.2) **@{-}, (15.7,1.6);(17,0) **@{-},
(17,20);(23,12) **@{-}, (13,0);(23,12) **@{-}, (7,0);(23,20) **@{-},
(10,18);(10,14) **@{-}, (10.1,18);(10.1,14) **@{-},
(9.9,18);(9.9,14) **@{-}, (10.2,18);(10.2,14) **@{-},
(9.8,18);(9.8,14) **@{-}, (18,16);(22,16) **@{-},
(18,16.1);(22,16.1) **@{-}, (18,15.9);(22,15.9) **@{-},
(18,16.2);(22,16.2) **@{-}, (18,15.8);(22,15.8) **@{-},
 (8.5,17.7) *{\ulcorner}, (18.5,17.7) *{\ulcorner}, (9.1,9) *{\ulcorner}, (12,18.4) *{\llcorner}, (14.5,1.6) *{\llcorner}, (21.8,18.4) *{\llcorner}, (17,12) *{\urcorner}, (21.7,1.6) *{\lrcorner},
(35,11);(45,11) **@{-} ?>*\dir{>}, (35,9);(45,9) **@{-} ?<*\dir{<},
(73,20);(65.8,11) **@{-}, (64.2,9);(62.6,7) **@{-}, (61,5);(57,0)
**@{-}, (67,20);(73,12) **@{-}, (73,12);(68.8,7) **@{-},
(67.3,5.2);(65.6,3.2) **@{-}, (64.3,1.6);(63,0) **@{-},
(63,20);(57,12) **@{-}, (67,0);(57,12) **@{-}, (73,0);(57,20)
**@{-},
 (58.5,17.7) *{\ulcorner}, (68.5,17.7) *{\ulcorner}, (59.1,9) *{\ulcorner}, (62,18.4) *{\llcorner}, (64,1.2) *{\llcorner}, (71.8,18.4) *{\llcorner}, (67,12) *{\urcorner}, (71.7,1.6) *{\lrcorner},
(60,18);(60,14) **@{-}, (60.1,18);(60.1,14) **@{-},
(59.9,18);(59.9,14) **@{-}, (60.2,18);(60.2,14) **@{-},
(59.8,18);(59.8,14) **@{-}, (68,16);(72,16) **@{-},
(68,16.1);(72,16.1) **@{-}, (68,15.9);(72,15.9) **@{-},
(68,16.2);(72,16.2) **@{-}, (68,15.8);(72,15.8) **@{-},
(-5,10)*{\Gamma_{8}:}, 
\endxy}
\caption{Yoshikawa moves}
\label{fig-moves-type-II-o}
\end{figure}

By a {\it surface-link} we mean a closed 2-manifold smoothly (or piecewise linearly and locally flatly) embedded in the $4$-space $\mathbb R^4$. Two surface-links are said to be {\it equivalent} if they are ambient isotopic. 

For a given marked graph diagram $D$, let $L_-(D)$ and $L_+(D)$ be classical link diagrams obtained from $D$ by replacing each marked vertex \xy (-4,4);(4,-4) **@{-}, 
(4,4);(-4,-4) **@{-},  
(3,-0.2);(-3,-0.2) **@{-},
(3,0);(-3,0) **@{-}, 
(3,0.2);(-3,0.2) **@{-}, 
\endxy with \xy (-4,4);(-4,-4) **\crv{(1,0)},  
(4,4);(4,-4) **\crv{(-1,0)}, 
\endxy and \xy (-4,4);(4,4) **\crv{(0,-1)}, 
(4,-4);(-4,-4) **\crv{(0,1)},   
\endxy, respectively (see Figure~\ref{fig-nori-mg}). We call $L_-(D)$ and $L_+(D)$ the {\it negative resolution} and the {\it positive resolution} of D, respectively.
A marked graph diagram $D$ is said to be {\it admissible} if both resolutions $L_-(D)$ and $L_+(D)$ are diagrams of trivial links.  
A marked graph is called  {\it admissible}  if its diagram is admissible. 

\begin{figure}[ht]
\begin{center}
\resizebox{0.65\textwidth}{!}{%
  \includegraphics{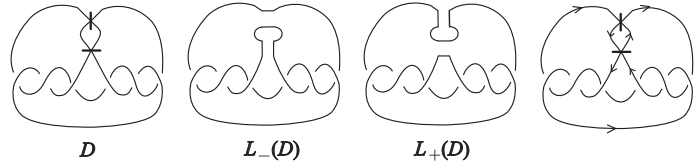}}
\caption{A marked graph diagram and its resolutions}\label{fig-nori-mg}
\end{center}
\end{figure}

For $t \in \mathbb R,$ we denote by $\mathbb R^3_t$ the hyperplane of $\mathbb R^4$ whose fourth coordinate
is equal to $t \in \mathbb R$, i.e., $\mathbb R^3_t := \{(x_1, x_2, x_3, x_4) \in
\mathbb R^4~|~ x_4 = t \}$.  Let $p:\mathbb R^4 \to \mathbb R$ be the projection given by $p(x_1, x_2, x_3, x_4)=x_4$. 
Any surface-link $\mathcal L$ can be deformed into a surface-link $\mathcal L'$, called a {\it hyperbolic splitting} of $\mathcal L$,
by an ambient isotopy of $\mathbb R^4$ in such a way that
the projection $p: \mathcal L' \to \mathbb R$ satisfies that
all critical points are non-degenerate,
all the index 0 critical points (minimal points) are in $\mathbb R^3_{-1}$,
all the index 1 critical points (saddle points) are in $\mathbb R^3_0$, and
all the index 2 critical points (maximal points) are in $\mathbb R^3_1$ (cf. 
 \cite{Ka2,Kaw,KSS,Lo}). 

Let $\mathcal L$ be a surface-link and let ${\mathcal L'}$ be a hyperbolic splitting of $\mathcal L.$ The cross-section $\mathcal L'\cap \mathbb R^3_0$ at $t=0$ is a spatial $4$-valent regular graph in $\mathbb R^3_0$. We give a marker at each $4$-valent vertex (saddle point) that indicates how the saddle point opens up above as illustrated in Figure~\ref{sleesan2:fig1}. 

\begin{figure}[ht]
\begin{center}
\resizebox{0.30\textwidth}{!}{%
  \includegraphics{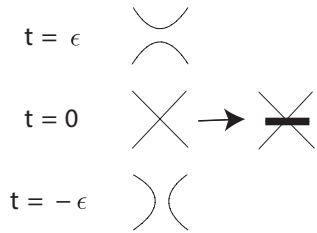} }
\caption{Marking of a vertex} \label{sleesan2:fig1}
\end{center}
\end{figure} 

The resulting marked graph $G$ is called a {\it marked graph} presenting $\mathcal L$.   Let $D$ be a diagram of $G$.  The diagram $D$ is clearly admissible, which is called a {\it marked graph diagram} (or {\it ch-diagram} (cf. \cite{So})) {\it presenting} $\mathcal L$.  Conversely, any admissible marked graph presents a surface-link.  

When $\mathcal L$ is an oriented surface-link, we choose an orientation for each edge of $\mathcal L' \cap \mathbb R^3_0$ that coincides with the induced orientation on the boundary of $\mathcal L' \cap \mathbb R^3 \times (-\infty, 0]$ by the orientation of $\mathcal L'$ inherited from the orientation of $\mathcal L$. The resulting oriented marked graph $G$ (or its diagram $D$) is called an {\it oriented marked graph} (or an {\it oriented marked graph diagram}) presenting  $\mathcal L$. 

It is known that two oriented marked graph diagrams present equivalent oriented surface-links if and only if they are related by a finite sequence of $11$ Yoshikawa moves shown in Figure~\ref{fig-moves-type-II-o} (cf. \cite{KJL2,KK,Sw}).  


\section{The $A_2$ bracket polynomial of links and tangled trivalent graphs}\label{sect-qa2-inv}

In this section, we review the $A_2$ bracket $\langle \cdot \rangle_{A_2}$ for regular isotopy of oriented link diagrams and tangled trivalent graph diagrams derived in \cite{Kup}. 
Although the $A_2$ bracket in \cite{Kup} is defined such that the value for the empty diagram is $1$, we here adapt another initial condition that the value of the trivial knot diagram is $1$.  

A {\it tangled trivalent graph diagram} (or an {\it $A_2$ freeway}, cf. \cite{Kup}) is an oriented link diagram in $S^2$ possibly with some trivalent vertices whose incident edges are oriented all inward or all outward as shown in Figure~\ref{triv-vrtx}. An example of a tangled trivalent graph diagram is in Figure~\ref{exmp-tvgd}. Throughout this paper we regard classical link diagrams as tangled trivalent graph diagrams without trivalent vertices otherwise specified. Two  tangled trivalent graph diagrams are said to be {\it regular isotopic} if they are related by a regular isotopy, which is defined to be a sequence of operations consisting of ambient isotopy of the 2-sphere $S^2$ and the combinatorial moves shown in Figure~\ref{reg-isotpy} with all possible orientations.

\begin{figure}[ht]
\begin{center}
\resizebox{0.40\textwidth}{!}{%
  \includegraphics{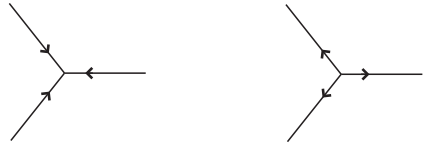} }
\caption{Trivalent vertices with orientation}\label{triv-vrtx}
\end{center}
\end{figure}

\begin{figure}[ht]
\begin{center}
\resizebox{0.30\textwidth}{!}{%
  \includegraphics{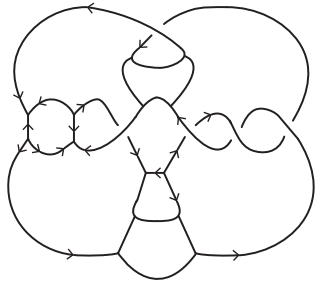} }
\caption{A tangled trivalent graph diagram}\label{exmp-tvgd}
\end{center}
\end{figure}

\begin{figure}[ht]
\begin{center}
\resizebox{0.60\textwidth}{!}{%
  \includegraphics{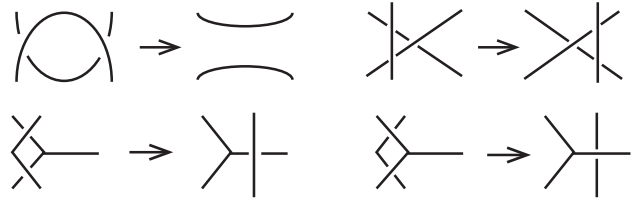} }
\caption{Moves on tangled trivalent graph diagrams}\label{reg-isotpy}
\end{center}
\end{figure}

In \cite{Kup}, G. Kuperberg derived an inductive, combinatorial definition of a polynomial valued invariant $\langle\cdot\rangle_{A_2}$ with values in the ring $\mathbb Z[q^{-\frac{1}{6}},q^{\frac{1}{6}}]$ of integral Laurent polynomials for regular isotopy classes of tangled trivalent graph diagrams. For our purpose, we present here the definition of $\langle\cdot\rangle_{A_2}$ with $q^{\frac{1}{6}}=a$. Moreover, we change the initial condition so that the trivial knot diagram has value $1$.  

We denote by $\xy (0,0) *\xycircle(3,3){-}, (3,0) *{\wedge}, \endxy$ or by $O$ the trivial knot diagram, by $O^\mu$ the trivial link diagram with $\mu$ components, 
by $D \sqcup D'$ a disjoint union of diagrams $D$ and $D'$. 

\begin{theorem}{\cite[Theorem 1.2]{Kup}} \label{thm-A2inv}
There is an invariant $\langle \cdot \rangle_{A_2}$ with values in the ring $\mathbb Z[a^{-1},a]$ of integral Laurent polynomials for regular isotopy of tangled trivalent graph diagrams, called the {\it $A_2$ bracket}, which is given by the following recursive rules:
\begin{itemize}
\item[{\rm ({\bf K0})}] $\langle O \rangle_{A_2} = 1$. 
\item[{\rm ({\bf K1})}] 
$\langle D \sqcup  O  \rangle_{A_2} 
= (a^{-6}+1+a^{6})\langle D \rangle_{A_2}$ for any diagram $D$. 
\item[{\rm ({\bf K2})}] 
$\langle ~\xy (-8,0);(-3,0) **@{-}, 
(0,0) *\xycircle(3,3){-}, 
(8,0);(3,0) **@{-}, 
(0,3) *{<}, (0,-3) *{<},(5.5,0) *{>}, 
(-5.5,0) *{>}, \endxy~ \rangle_{A_2} 
= (a^{-3}+a^{3})\langle ~\xy 
(-5,0);(5,0) **@{-}, (0,0) *{>}, 
\endxy~ \rangle_{A_2}.$
\item[{\rm ({\bf K3})}] $\langle ~\xy 
(-5,5);(-3,3) **@{-}, (5,-5);(3,-3) **@{-},
(-5,-5);(-3,-3) **@{-}, (5,5);(3,3) **@{-},
(-3,3);(3,3) **@{-}, (-3,3);(-3,-3) **@{-},
(-3,-3);(3,-3) **@{-}, (3,3);(3,-3) **@{-},
(-3.5,-3.5) *{\llcorner}, (3.5,3) *{\urcorner},
(-4.3,4.3) *{\lrcorner}, (4.5,-4.95) *{\ulcorner},
(0,3) *{<}, (0,-3) *{>},
(-3,0) *{\wedge}, (3,0) *{\vee},
\endxy~
\rangle_{A_2} = \langle ~\xy 
(-4,4);(4,4) **\crv{(0,-1)}, 
(4,-4);(-4,-4) **\crv{(0,1)}, 
(2.7,1.9)*{\urcorner}, 
(-2.7,-2.4)*{\llcorner}, 
\endxy~ \rangle_{A_2} 
+\langle ~\xy 
(-4,4);(-4,-4) **\crv{(1,0)},  
(4,4);(4,-4) **\crv{(-1,0)}, 
(-2,-1.9)*{\llcorner}, 
(2.5,2)*{\urcorner}, 
\endxy~  \rangle_{A_2}.$
\item[{\rm ({\bf K4})}] 
$\langle ~\xy 
(5,5);(-5,-5) **@{-} ?<*\dir{<},
(-5,5);(-2,2) **@{-} ?<*\dir{<}, 
(2,-2);(5,-5) **@{-}, 
\endxy~ 
\rangle_{A_2} = -a\langle ~\xy 
(-2,5);(2,2) **@{-}, (6,-5);(2,-2) **@{-},
(-2,-5);(2,-2) **@{-}, (6,5);(2,2) **@{-}, 
(2,2);(2,-2) **@{-},
(-0.6,-4.6) *{\urcorner}, 
(3.7,2.5) *{\urcorner},
(0.4,2.5) *{\ulcorner}, 
(4.4,-4.6) *{\ulcorner}, 
(2.05,0) *{\vee},
\endxy~ \rangle_{A_2} + a^{-2}\langle ~\xy 
(-4,4);(-4,-4) **\crv{(1,0)},  
(4,4);(4,-4) **\crv{(-1,0)}, 
(-2.5,1.9)*{\ulcorner}, 
(2.5,1.9)*{\urcorner}, 
\endxy~ \rangle_{A_2}.$
\item[{\rm ({\bf K5})}] $\langle 
~\xy 
(-5,5);(5,-5) **@{-} ?<*\dir{<},
(5,5);(2,2) **@{-} ?<*\dir{<}, 
(-2,-2);(-5,-5) **@{-}, 
\endxy~ \rangle_{A_2} 
= -a^{-1}\langle ~\xy 
(-2,5);(2,2) **@{-}, (6,-5);(2,-2) **@{-},
(-2,-5);(2,-2) **@{-}, (6,5);(2,2) **@{-}, 
(2,2);(2,-2) **@{-},
(-0.6,-4.6) *{\urcorner}, 
(3.7,2.5) *{\urcorner},
(0.4,2.5) *{\ulcorner}, 
(4.4,-4.6) *{\ulcorner}, 
(2.05,0) *{\vee},
\endxy~\rangle_{A_2} + a^{2}\langle ~\xy 
(-4,4);(-4,-4) **\crv{(1,0)},  
(4,4);(4,-4) **\crv{(-1,0)}, 
(-2.5,1.9)*{\ulcorner}, 
(2.5,1.9)*{\urcorner}, 
\endxy~ \rangle_{A_2}.$
\end{itemize}
\end{theorem}

In \cite{RT}, Reshetikhin and Turaev showed that for any simple Lie algebra $\mathfrak g,$ there exists an invariant $RT_{\mathfrak g}$ of appropriately colored tangled ribbon graphs.
Each edge is colored by an irreducible representation of $\mathfrak g$ and each vertex is colored by a tensor of a certain kind. The $A_2$ bracket $\langle \cdot \rangle_{A_2}$, with $\langle \emptyset \rangle_{A_2}=1$, is identically equal to $RT_{\mathfrak g}$ with $\mathfrak g=A_2$ if all edges of a tangled trivalent graph diagram are colored with the $3$-dimensional representation $V_{1,0}$ whose dual is $V_{0,1}$. The colors for the vertices can be recognized as the determinant, or the usual $3$-dimensional cross product. For details, see \cite{Kup}. In particular, the $A_2$ bracket $\langle \cdot \rangle_{A_2}$ for oriented link diagrams is essentially a specialization of the HOMFLY polynomial (cf. \cite{F}) with a normalization that makes it a regular isotopy invariant rather than an isotopy invariant. Actually, it follows from {\rm ({\bf K4})} and {\rm ({\bf K5})} that for any skein triple $(D_+, D_-, D_0)$, 
\begin{equation}\label{A2-R1-move-1}
a^{-1} \langle D_+ \rangle_{A_2} - a \langle D_-  \rangle_{A_2} 
= (a^{-3} - a^{3}) \langle  D_0   \rangle_{A_2}.
\end{equation}
Moreover, it is easy to check that
\begin{align}\label{A2-R1-move-2}
&\langle  ~\xy (1,-3);(6,2) **@{-},
(1,3);(3.5,0.5) **@{-},
(4.5,-0.5);(6,-2) **@{-},
(6,2);(10,2) **\crv{(8,4)},
(6,-2);(10,-2) **\crv{(8,-4)}, 
(10,2);(10,-2) **\crv{(11.5,0)},
(2.4,1.1) *{\ulcorner}, 
(5.6,1.1) *{\urcorner}, 
\endxy~\rangle_{A_2} 
= a^{-8}\langle~\xy 
(1.5,-3);(1.5,3) **\crv{(4,0)}, 
(2.4,1.5) *{\ulcorner}, 
\endxy~\rangle_{A_2} 
=\langle~\xy 
(1,-3);(6,2) **@{-},
(1,3);(3.5,0.5) **@{-},
(4.5,-0.5);(6,-2) **@{-},
(6,2);(10,2) **\crv{(8,4)},
(6,-2);(10,-2) **\crv{(8,-4)}, 
(10,2);(10,-2) **\crv{(11.5,0)},
(2.4,-1.5) *{\llcorner}, 
(5.6,-1.5) *{\lrcorner}, 
\endxy~\rangle_{A_2},\\
&\langle  ~\xy 
(4.5,0.6);(6,2) **@{-},
(1,-3);(3.3,-0.7) **@{-},
(1,3);(3.5,0.5) **@{-},
(3.5,0.5);(6,-2) **@{-},
(6,2);(10,2) **\crv{(8,4)},
(6,-2);(10,-2) **\crv{(8,-4)}, 
(10,2);(10,-2) **\crv{(11.5,0)},
(2.4,1.1) *{\ulcorner}, 
(5.6,1.1) *{\urcorner}, 
\endxy~\rangle_{A_2} 
= a^{8}\langle~\xy 
(1.5,-3);(1.5,3) **\crv{(4,0)}, 
(2.4,1.5) *{\ulcorner}, 
\endxy~\rangle_{A_2} 
=\langle~\xy 
(4.5,0.6);(6,2) **@{-},
(1,-3);(3.3,-0.7) **@{-},
(1,3);(3.5,0.5) **@{-},
(3.5,0.5);(6,-2) **@{-},
(6,2);(10,2) **\crv{(8,4)},
(6,-2);(10,-2) **\crv{(8,-4)}, 
(10,2);(10,-2) **\crv{(11.5,0)},
(2.4,-1.5) *{\llcorner}, 
(5.6,-1.5) *{\lrcorner}, 
\endxy~\rangle_{A_2}.\notag
\end{align}


\section{A polynomial for oriented marked graphs via $A_2$ bracket}\label{sect-poly-osl}

In this section, we define a polynomial invariant of  
oriented marked graphs using the $A_2$ bracket $\langle \cdot \rangle_{A_2}$ in the line of \cite{Le3}. 

\begin{definition}\label{defn-L-poly-osl}
Let $D$ be an oriented marked graph diagram or a tangled trivalent graph diagram. Let
$[[D]]=[[D]](a,x,y)$ be a polynomial in $\mathbb Z[a^{-1},a,x,y]$
defined by the following two axioms:

\begin{itemize}
\item[{\rm ({\bf L1})}] 
$[[D]] =\langle D \rangle_{A_2}$ 
if $D$ is a tangled trivalent graph diagram.
\item[{\rm ({\bf L2})}] 
$[[~\xy (-4,4);(4,-4) **@{-}, 
(4,4);(-4,-4) **@{-}, 
(3,3.2)*{\llcorner}, 
(-3,-3.4)*{\urcorner}, 
(-2.5,2)*{\ulcorner},
(2.5,-2.4)*{\lrcorner}, 
(3,-0.2);(-3,-0.2) **@{-},
(3,0);(-3,0) **@{-}, 
(3,0.2);(-3,0.2) **@{-}, 
\endxy~]] =
x[[~\xy (-4,4);(4,4) **\crv{(0,-1)}, 
(4,-4);(-4,-4) **\crv{(0,1)}, 
(-2.5,1.9)*{\ulcorner}, (2.5,-2.4)*{\lrcorner}, 
\endxy~]] + y[[~\xy (-4,4);(-4,-4) **\crv{(1,0)},  
(4,4);(4,-4) **\crv{(-1,0)}, 
(-2.5,1.9)*{\ulcorner}, (2.5,-2.4)*{\lrcorner}, 
\endxy~]],$
\end{itemize}
where ~\xy (-4,4);(4,-4) **@{-}, 
(4,4);(-4,-4) **@{-}, 
(3,3.2)*{\llcorner}, 
(-3,-3.4)*{\urcorner}, 
(-2.5,2)*{\ulcorner},
(2.5,-2.4)*{\lrcorner}, 
(3,-0.2);(-3,-0.2) **@{-},
(3,0);(-3,0) **@{-}, 
(3,0.2);(-3,0.2) **@{-}, 
\endxy,~\xy (-4,4);(4,4) **\crv{(0,-1)}, 
(4,-4);(-4,-4) **\crv{(0,1)}, 
(-2.5,1.9)*{\ulcorner}, (2.5,-2.4)*{\lrcorner},    
\endxy~ and ~\xy (-4,4);(-4,-4) **\crv{(1,0)},  
(4,4);(4,-4) **\crv{(-1,0)}, 
(-2.5,1.9)*{\ulcorner}, (2.5,-2.4)*{\lrcorner}, 
\endxy~ denote the small
parts of larger diagrams that are identical except the local sites indicated by the small parts.
\end{definition}

The {\it writhe} $w(D)$ of an oriented marked graph diagram $D$ is defined to be the sum of the signs of all crossings in $D$ defined by
$\mathrm{sign} \left( \xy (5,2.5);(0,-2.5) **@{-} ?<*\dir{<},
(5,-2.5);(3,-0.5) **@{-}, (0,2.5);(2,0.5) **@{-} ?<*\dir{<},
\endxy \right) = 1$ and $\mathrm{sign} \left( \xy (0,2.5);(5,-2.5)
**@{-} ?<*\dir{<}, (0,-2.5);(2,-0.5) **@{-}, (5,2.5);(3,0.5)
**@{-} ?<*\dir{<}, \endxy \right) = -1$ analogue to the writhe of a link diagram. 

\begin{definition}\label{defn-poly-iv-24}
Let $D$ be an oriented marked graph diagram. We define $\ll D\gg=\ll D\gg(a,x,y)$ to be a polynomial in variables $a, x$ and $y$ with integral coefficients given by 
\begin{equation*}
\ll D\gg=a^{8w(D)}[[D]](a, x, y).
\end{equation*}
\end{definition}

Let $D$ be an oriented marked graph diagram.
A {\it state} of $D$ is an assignment of $T_\infty$ or $T_0$ to each marked vertex in $D$. Let $\mathcal S(D)$ be the set of all states of $D$. For each state $\sigma \in \mathcal S(D),$ let $D_{\sigma}$ denote the oriented link 
diagram obtained from $D$ by
replacing marked vertices of $D$ with two trivial $2$-tangles
according to the assignment $T_\infty$ or $T_0$ by the
state $\sigma$ as follows:
$$\underset{~T_\infty}
{\xy (-4,4);(4,-4) **@{-}, 
(4,4);(-4,-4) **@{-}, 
(3,3.2)*{\llcorner}, 
(-3,-3.4)*{\urcorner}, 
(-2.5,2)*{\ulcorner},
(2.5,-2.4)*{\lrcorner}, 
(3,-0.2);(-3,-0.2) **@{-},
(3,0);(-3,0) **@{-}, 
(3,0.2);(-3,0.2) **@{-}, 
\endxy} \longrightarrow 
\xy (-4,4);(4,4) **\crv{(0,-1)}, 
(4,-4);(-4,-4) **\crv{(0,1)},  
(-2.5,1.9)*{\ulcorner}, 
(2.5,-2.4)*{\lrcorner},   
\endxy,~~~~~~
\underset{~T_0}
{\xy (-4,4);(4,-4) **@{-}, 
(4,4);(-4,-4) **@{-}, 
(3,3.2)*{\llcorner}, 
(-3,-3.4)*{\urcorner}, 
(-2.5,2)*{\ulcorner},
(2.5,-2.4)*{\lrcorner}, 
(3,-0.2);(-3,-0.2) **@{-},
(3,0);(-3,0) **@{-}, 
(3,0.2);(-3,0.2) **@{-}, 
\endxy} \longrightarrow 
\xy (-4,4);(-4,-4) **\crv{(1,0)},  
(4,4);(4,-4) **\crv{(-1,0)}, 
(-2.5,1.9)*{\ulcorner}, 
(2.5,-2.4)*{\lrcorner}, 
\endxy.$$
Then the skein relation $({\bf L2})$
leads the following {\it state-sum formula} for the polynomial
$\ll D\gg$:
\begin{equation*}
\ll D\gg = a^{8w(D)}\sum_{\sigma \in \mathcal S(D)}
x^{\sigma(\infty)}y^{\sigma(0)} 
\langle D_\sigma \rangle_{A_2}, 
\end{equation*}
where $\sigma(\infty)$ and $\sigma(0)$ denote
the numbers of the assignment $T_\infty$ and $T_0$ of the
state $\sigma,$ respectively. 
Since $w(D)= w(D_\sigma)$ for any $\sigma \in \mathcal S(D)$, 
we also have the following formula
\begin{equation}\label{state formulaB}
\ll D \gg = \sum_{\sigma \in \mathcal S(D)}
x^{\sigma(\infty)}y^{\sigma(0)}
\ll  D_\sigma \gg. 
\end{equation}

\begin{theorem}\label{thm-skein-rel}
The polynomial $\ll \cdot \gg$ is an invariant for oriented marked graphs, i.e., for an oriented marked graph diagram $D$, the polynomial $\ll D\gg$ is invariant under Yoshikawa moves $\Gamma_1, \Gamma'_1, \Gamma_2, \Gamma_3, \Gamma_4, \Gamma'_4$ and $\Gamma_5$.  Moreover, it satisfies the following.  
\begin{itemize}
\item[(1)] 
$\ll O \gg = 1.$
\item[(2)] 
$\ll D \sqcup O \gg = (a^{-6}+1+a^6)\ll D  \gg$ for any oriented link diagram $D$.
\item[(3)] 
$a^{-9} \ll D_+ \gg - a^9 \ll D_- \gg = (a^{-3}-a^3) \ll D_0 \gg$ for any  
skein triple $(D_+, D_-, D_0)$  of  link diagrams. 
\item[(4)] 
$\ll ~\xy (-4,4);(4,-4) **@{-}, 
(4,4);(-4,-4) **@{-}, 
(3,3.2)*{\llcorner}, 
(-3,-3.4)*{\urcorner}, 
(-2.5,2)*{\ulcorner},
(2.5,-2.4)*{\lrcorner}, 
(3,-0.2);(-3,-0.2) **@{-},
(3,0);(-3,0) **@{-}, 
(3,0.2);(-3,0.2) **@{-}, 
\endxy~\gg =
x\ll~\xy (-4,4);(4,4) **\crv{(0,-1)}, 
(4,-4);(-4,-4) **\crv{(0,1)},  
(-2.5,1.9)*{\ulcorner}, (2.5,-2.4)*{\lrcorner},   
\endxy~\gg + y \ll~\xy (-4,4);(-4,-4) **\crv{(1,0)},  
(4,4);(4,-4) **\crv{(-1,0)}, (-2.5,1.9)*{\ulcorner}, (2.5,-2.4)*{\lrcorner}, 
\endxy~\gg.$
\end{itemize}
\end{theorem}

\begin{proof}
For $\Gamma_1$ and $\Gamma_1'$, it follows from (\ref{A2-R1-move-2}) that  
$[[~\xy (1,-3);(6,2) **@{-},
(1,3);(3.5,0.5) **@{-},
(4.5,-0.5);(6,-2) **@{-},
(6,2);(10,2) **\crv{(8,4)},
(6,-2);(10,-2) **\crv{(8,-4)}, 
(10,2);(10,-2) **\crv{(11.5,0)},
(2.4,1.1) *{\ulcorner}, 
(5.6,1.1) *{\urcorner}, 
\endxy~]] 
= a^{-8}[[~\xy 
(1.5,-3);(1.5,3) **\crv{(4,0)}, 
(2.4,1.5) *{\ulcorner}, \endxy~]] 
=[[~\xy (1,-3);(6,2) **@{-}, 
(1,3);(3.5,0.5) **@{-}, 
(4.5,-0.5);(6,-2) **@{-},
(6,2);(10,2) **\crv{(8,4)}, 
(6,-2);(10,-2) **\crv{(8,-4)}, 
(10,2);(10,-2) **\crv{(11.5,0)}, 
(2.4,-1.5) *{\llcorner}, 
(5.6,-1.5) *{\lrcorner}, 
\endxy~]].$
Hence, we have
\begin{equation*}
\begin{split}
\ll~\xy (1,-3);(6,2) **@{-},
(1,3);(3.5,0.5) **@{-},
(4.5,-0.5);(6,-2) **@{-},
(6,2);(10,2) **\crv{(8,4)},
(6,-2);(10,-2) **\crv{(8,-4)}, 
(10,2);(10,-2) **\crv{(11.5,0)},
(2.4,1.1) *{\ulcorner}, 
(5.6,1.1) *{\urcorner}, 
\endxy~\gg
&=a^{8w(~\xy (1,-3);(6,2) **@{-},
(1,3);(3.5,0.5) **@{-},
(4.5,-0.5);(6,-2) **@{-},
(6,2);(10,2) **\crv{(8,4)},
(6,-2);(10,-2) **\crv{(8,-4)}, 
(10,2);(10,-2) **\crv{(11.5,0)},
(2.4,1.1) *{\ulcorner}, 
(5.6,1.1) *{\urcorner}, 
\endxy~)}
[[~~\xy (1,-3);(6,2) **@{-},
(1,3);(3.5,0.5) **@{-},
(4.5,-0.5);(6,-2) **@{-},
(6,2);(10,2) **\crv{(8,4)},
(6,-2);(10,-2) **\crv{(8,-4)}, 
(10,2);(10,-2) **\crv{(11.5,0)},
(2.4,1.1) *{\ulcorner}, 
(5.6,1.1) *{\urcorner}, 
\endxy~~ ]]
=a^{8w(~~\xy 
(1.5,-3);(1.5,3) **\crv{(4,0)}, 
(2.4,1.5) *{\ulcorner}, 
\endxy ~~)+8}(a^{-8})
[[~\xy (1.5,-3);(1.5,3) **\crv{(4,0)}, 
(2.4,1.5) *{\ulcorner}, 
\endxy ~]]\\
&=a^{8w(~~\xy 
(1.5,-3);(1.5,3) **\crv{(4,0)}, 
(2.4,1.5) *{\ulcorner}, 
\endxy ~~)}
[[~\xy (1.5,-3);(1.5,3) **\crv{(4,0)}, 
(2.4,1.5) *{\ulcorner}, 
\endxy ~]]
=\ll~\xy (1.5,-3);(1.5,3) **\crv{(4,0)}, 
(2.4,1.5) *{\ulcorner}, 
\endxy~\gg.
\end{split}
\end{equation*}
Similarly, we obtain 
$\ll~\xy (1,-3);(6,2) **@{-},
(1,3);(3.5,0.5) **@{-},
(4.5,-0.5);(6,-2) **@{-},
(6,2);(10,2) **\crv{(8,4)},
(6,-2);(10,-2) **\crv{(8,-4)}, 
(10,2);(10,-2) **\crv{(11.5,0)},
(2.4,-1.5) *{\llcorner}, 
(5.6,-1.5) *{\lrcorner}, 
\endxy~\gg 
=\ll~\xy (1.5,-3);(1.5,3) **\crv{(4,0)}, 
(2.5,-1.5) *{\llcorner}, 
\endxy~\gg.$

Since the $A_2$ bracket $\langle \cdot \rangle_{A_2}$ for oriented link diagrams (and tangled trivalent graph diagrams) and the writhe $w(D)$ are both regular isotopy invariants,  $\ll\cdot \gg$ is invariant under $\Gamma_2$ and $\Gamma_3$.  

The invariance of $[[\cdot]]$ under the moves $\Gamma_4, \Gamma'_4$ and $\Gamma_5$ are seen as below, and since the writhe $w(D)$ is also invariant under these moves, we obtain the invariance of $\ll\cdot\gg$ under $\Gamma_4, \Gamma'_4$ and $\Gamma_5$.

\begin{align*}
\xy (4,6)*{[[},
 (7,6);(23,6)  **\crv{(15,-2)}, 
 (10,0);(11.5,1.8) **@{-},
(12.5,3);(20,12) **@{-}, 
(10,12);(17.5,3) **@{-}, 
(18.5,1.8);(20,0) **@{-}, 
(13,6);(17,6) **@{-}, 
(13,6.1);(17,6.1) **@{-}, 
(13,5.9);(17,5.9)**@{-}, 
(13,6.2);(17,6.2) **@{-}, 
(13,5.8);(17,5.8) **@{-},
(13,3.1) *{\urcorner}, 
(17.5,8.9) *{\llcorner}, 
(19,1.1) *{\lrcorner},  
(21,3.5) *{\urcorner},
(13,8) *{\ulcorner}, 
(27,6) *{]]},
\endxy
&\xy (1,6)*{= x [[},
 (7,6);(23,6)  **\crv{(15,-2)}, 
(10,0);(11,1.8) **@{-},
(10,12);(20,12) **\crv{(15,6)},
(12.5,4);(17.5,4) **\crv{(15,9)},
(18.6,2);(20,0) **@{-}, 
(17.5,9.9) *{\llcorner}, 
(19,1.1) *{\lrcorner},  
(21,3.5) *{\urcorner},
(27,6) *{]]},
\endxy
\xy (1,6)*{+ y [[},
(7,6);(23,6)  **\crv{(15,-2)}, 
(10,0);(11,1.8) **@{-},
(10,12);(12,4) **\crv{(14,8)},
(20,12);(17.5,4) **\crv{(15,8)},
(18.6,2);(20,0) **@{-},
(19,1.1) *{\lrcorner},  
(21,3.5) *{\urcorner},
(10.7,10.7) *{\ulcorner}, 
(27,6) *{]]},
\endxy\\
&\xy (51,6)*{= x [[},
(57,6);(73,6)  **\crv{(65,14)}, 
(70,12);(68.5,10.2) **@{-},
(67.5,9);(62.5,9)**\crv{(65,4)}, 
(60,0);(70,0) **\crv{(65,7.5)},
(61.5,10.2);(60,12) **@{-}, 
(67.5,3) *{\lrcorner},  
(70.9,8) *{\lrcorner},
(61.3,10) *{\ulcorner}, 
(77,6)*{]]},
\endxy \xy (51,6)*{+ y [[},
(57,6);(73,6)  **\crv{(65,14)}, 
(70,12);(68.5,10.2) **@{-},
(67.5,9);(70,0) **\crv{(65,3.5)}, 
(60,0);(62.5,9) **\crv{(64,3.5)}, 
(61.5,10.2);(60,12) **@{-}, 
(67.5,3) *{\lrcorner},  
(70.9,8) *{\lrcorner},
(61.3,10) *{\ulcorner}, 
(77,6)*{]]},
\endxy
\xy (52,6)*{= [[},
(57,6);(73,6)  **\crv{(65,14)}, 
(70,12);(68.5,10.2) **@{-},
(67.5,9);(60,0) **@{-}, 
(70,0);(62.5,9) **@{-}, 
(61.5,10.2);(60,12) **@{-}, 
(63,6);(67,6) **@{-}, 
(63,6.1);(67,6.1) **@{-}, 
(63,5.9);(67,5.9)**@{-}, 
(63,6.2);(67,6.2) **@{-}, 
(63,5.8);(67,5.8) **@{-},
(62,2) *{\urcorner}, 
(67,8.3) *{\llcorner}, 
(67.5,3) *{\lrcorner},  
(70.9,8) *{\lrcorner},
(61.3,10) *{\ulcorner}, 
(77,6)*{]]},
\endxy
\end{align*}

\begin{align*}
\xy (5,6)*{[[},
(13,2.2);(17,2.2)  **\crv{(15,1.7)}, 
(7,6);(11,3)  **\crv{(10,3.5)}, 
(23,6);(19,3)  **\crv{(20,3.5)}, 
(10,0);(20,12) **@{-}, 
(10,12);(20,0) **@{-}, 
(13,6);(17,6) **@{-},
(13,6.1);(17,6.1) **@{-}, 
(13,5.9);(17,5.9)**@{-}, 
(13,6.2);(17,6.2) **@{-}, 
(13,5.8);(17,5.8) **@{-}, 
(13,3.1) *{\urcorner}, 
(17.5,8.9) *{\llcorner}, 
(19,1.1) *{\lrcorner},  
(21,3.5) *{\urcorner},
(13,8) *{\ulcorner},
(27,6)*{]]},
\endxy
&\xy (1,6)*{= x [[},
(13,2.2);(17,2.2)  **\crv{(15,1.7)}, 
(7,6);(11,3)  **\crv{(10,3.5)}, 
(23,6);(19,3)  **\crv{(20,3.5)}, 
(10,0);(20,0) **\crv{(15,10)},
(20,12);(10,12) **\crv{(15,5)},
(13,3.5) *{\urcorner}, 
(21,3.5) *{\urcorner},
(13,8) *{\ulcorner},
(27,6)*{]]},
\endxy
\xy (1,6)*{+ y [[},
(13,2.2);(17,2.2) **\crv{(15,1.7)}, 
(7,6);(11,3) **\crv{(10,3.5)}, 
(23,6);(19,3) **\crv{(20,3.5)}, 
(10,0);(10,12) **\crv{(16,5)},
(20,12);(20,0) **\crv{(14,5)},
(19,1.1) *{\lrcorner},  (21,3.5) *{\urcorner},
(12,9) *{\ulcorner},
(27,6)*{]]},
\endxy\\
&\xy (52,6)*{= x [[},
(63,9.8);(67,9.8)  **\crv{(65,10.3)}, 
  (57,6);(61,9)  **\crv{(60,8.5)}, 
  (73,6);(69,9)  **\crv{(70,8.5)}, 
(60,0);(70,0) **\crv{(65,9)},
(70,12);(60,12) **\crv{(65,3)},
(67,8.3) *{\llcorner}, 
(67.5,3) *{\lrcorner},  
(70.9,8) *{\lrcorner},
(77,6)*{]]},
\endxy
\xy (51,6)*{+ y [[},
(63,9.8);(67,9.8)  **\crv{(65,10.3)}, 
  (57,6);(61,9)  **\crv{(60,8.5)}, 
  (73,6);(69,9)  **\crv{(70,8.5)}, 
(60,0);(60,12) **\crv{(66,5)},
(70,12);(70,0) **\crv{(64,5)},
(67.5,3) *{\lrcorner},  
(70.9,8) *{\lrcorner},
(61.3,10) *{\ulcorner}, 
(77,6)*{]]},
\endxy
\xy (52,6)*{= [[},
(63,9.8);(67,9.8)  **\crv{(65,10.3)}, 
  (57,6);(61,9)  **\crv{(60,8.5)}, 
  (73,6);(69,9)  **\crv{(70,8.5)}, 
(70,12);(60,0) **@{-}, 
(70,0);(60,12) **@{-}, 
(63,6);(67,6) **@{-}, 
(63,6.1);(67,6.1) **@{-}, 
(63,5.9);(67,5.9)**@{-},
(63,6.2);(67,6.2) **@{-}, 
(63,5.8);(67,5.8) **@{-}, 
(62,2) *{\urcorner}, 
(67,8.3) *{\llcorner}, 
(67.5,3) *{\lrcorner},  
(70.9,8) *{\lrcorner},
(61.3,10) *{\ulcorner}, 
(77,6)*{]]},
\endxy
\end{align*}

\begin{align*}
\xy (4,4)*{[[},
(9,2);(13,6) **@{-}, 
(9,6);(10.5,4.5) **@{-},
(11.5,3.5);(13,2) **@{-}, 
(17,2);(21,6) **@{-}, 
(17,6);(21,2)**@{-}, 
(13,6);(17,6) **\crv{(15,8)}, 
(13,2);(17,2) **\crv{(15,0)},
(7,7);(9,6) **\crv{(8,7)}, 
(7,1);(9,2) **\crv{(8,1)}, 
(23,7);(21,6)**\crv{(22,7)}, 
(23,1);(21,2) **\crv{(22,1)}, 
(17,4);(21,4) **@{-}, 
(17,4.1);(21,4.1) **@{-}, 
(17,3.9);(21,3.9)**@{-}, 
(17,4.2);(21,4.2) **@{-}, 
(17,3.8);(21,3.8) **@{-},
(10,3) *{\llcorner},  
(12,3) *{\lrcorner}, 
(21,6) *{\llcorner},
(21,2.2) *{\lrcorner},
(27,4)*{]]},
\endxy
&\xy (1,4)*{=x [[},
(9,2);(13,6) **@{-}, 
(9,6);(10.5,4.5) **@{-},
(11.5,3.5);(13,2) **@{-}, 
(17,2);(21,2) **\crv{(19,4)}, 
(17,6);(21,6) **\crv{(19,4)}, 
(13,6);(17,6) **\crv{(15,8)}, 
(13,2);(17,2) **\crv{(15,0)},
(7,7);(9,6) **\crv{(8,7)}, 
(7,1);(9,2) **\crv{(8,1)}, 
(23,7);(21,6) **\crv{(22,7)}, 
(23,1);(21,2) **\crv{(22,1)}, 
(10,3) *{\llcorner},  
(12,3) *{\lrcorner}, 
(21,6) *{\llcorner},
(21,2.2) *{\lrcorner},
(27,4)*{]]},
\endxy
\xy (1,4)*{+ y [[},
(9,2);(13,6) **@{-}, 
(9,6);(10.5,4.5) **@{-},
(11.5,3.5);(13,2) **@{-}, 
(17,2);(17,6) **\crv{(18.5,4)}, 
(21,6);(21,2) **\crv{(19,4)}, 
(13,6);(17,6) **\crv{(15,8)}, 
(13,2);(17,2) **\crv{(15,0)},
(7,7);(9,6) **\crv{(8,7)}, 
(7,1);(9,2) **\crv{(8,1)}, 
(23,7);(21,6)**\crv{(22,7)}, 
(23,1);(21,2) **\crv{(22,1)}, 
(10,3) *{\llcorner},  
(12,3) *{\lrcorner}, 
(21,6) *{\llcorner},
(21,2.2) *{\lrcorner},
(27,4)*{]]},
\endxy\\
&\xy (51,4)*{=x [[},
(59,2);(63,2) **\crv{(61,4)}, 
(59,6);(63,6) **\crv{(61,4)}, 
(67,2);(71,6) **@{-},
(67,6);(68.5,4.5) **@{-}, 
(69.5,3.5);(71,2) **@{-}, 
(63,6);(67,6)**\crv{(65,8)}, 
(63,2);(67,2) **\crv{(65,0)}, 
(57,7);(59,6)**\crv{(58,7)}, 
(57,1);(59,2) **\crv{(58,1)}, 
(73,7);(71,6)**\crv{(72,7)}, 
(73,1);(71,2) **\crv{(72,1)}, 
(59.5,2.5) *{\llcorner},  
(59,6) *{\lrcorner}, 
(71,6) *{\llcorner},
(71,2.2) *{\lrcorner},
(77,4)*{]]},
\endxy
\xy (51,4)*{+ y [[},
(59,2);(59,6) **\crv{(61,4)}, 
(63,6);(63,2) **\crv{(61.5,4)},
(67,2);(71,6) **@{-},
(67,6);(68.5,4.5) **@{-}, 
(69.5,3.5);(71,2) **@{-}, 
(63,6);(67,6)**\crv{(65,8)}, 
(63,2);(67,2) **\crv{(65,0)}, 
(57,7);(59,6)**\crv{(58,7)}, 
(57,1);(59,2) **\crv{(58,1)}, 
(73,7);(71,6)**\crv{(72,7)}, 
(73,1);(71,2) **\crv{(72,1)}, 
(59.5,2.5) *{\llcorner},  
(59,6) *{\lrcorner}, 
(71,6) *{\llcorner},
(71,2.2) *{\lrcorner},
(77,4)*{]]},
\endxy
\xy (51,4)*{= [[},
(59,2);(63,6) **@{-}, 
(59,6);(63,2) **@{-}, 
(67,2);(71,6) **@{-},
(67,6);(68.5,4.5) **@{-}, 
(69.5,3.5);(71,2) **@{-}, 
(63,6);(67,6)**\crv{(65,8)}, 
(63,2);(67,2) **\crv{(65,0)}, 
(57,7);(59,6)**\crv{(58,7)}, 
(57,1);(59,2) **\crv{(58,1)}, 
(73,7);(71,6)**\crv{(72,7)}, 
(73,1);(71,2) **\crv{(72,1)}, 
(63,4);(59,4) **@{-}, 
(63,4.1);(59,4.1) **@{-}, 
(63,3.9);(59,3.9)**@{-}, 
(63,4.2);(59,4.2) **@{-}, 
(63,3.8);(59,3.8) **@{-},
(59.5,2.5) *{\llcorner},  
(59,6) *{\lrcorner}, 
(71,6) *{\llcorner},
(71,2.2) *{\lrcorner},
(77,4)*{]]},
\endxy
\end{align*}

The assertions (1) and  (2) follow from  (${\bf K0}$)  and (${\bf K1}$).  

(3) From (\ref{A2-R1-move-1}) and (${\bf L1}$), we see
$$a^{-1} [[D_+]] - a[[D_-]] = (a^{-3} - a^{3})[[D_0]].$$ 
Let $\lambda=w(D_0).$ Then we have
\begin{align*}
a^{-1}a^{8\lambda}[[D_+]] 
- aa^{8\lambda}[[D_-]] 
&= (a^{-3}-a^3)
a^{8\lambda}[[D_0]],\\
a^{-9}a^{8(\lambda+1)}[[D_+]] 
- a^9 a^{8(\lambda-1)}[[D_-]] 
&= (a^{-3}-a^3)
a^{8\lambda}[[D_0]],\\
a^{-9}\ll D_+ \gg - a^9 \ll D_- \gg 
&= (a^{-3}-a^3)\ll D_0 \gg.
\end{align*}

(4) Clearly, 
$w\bigg(~\xy 
(-4,4);(4,-4) **@{-}, 
(4,4);(-4,-4) **@{-}, 
(3,3.2)*{\llcorner}, 
(-3,-3.4)*{\urcorner}, 
(-2.5,2)*{\ulcorner},
(2.5,-2.4)*{\lrcorner}, 
(3,-0.2);(-3,-0.2) **@{-},
(3,0);(-3,0) **@{-}, 
(3,0.2);(-3,0.2) **@{-}, 
\endxy~\bigg)
=w\bigg(~\xy 
(-4,4);(4,4) **\crv{(0,-1)}, 
(4,-4);(-4,-4) **\crv{(0,1)},  
(-2.5,1.9)*{\ulcorner}, 
(2.5,-2.4)*{\lrcorner},   
\endxy~\bigg)
=w\bigg(~\xy 
(-4,4);(-4,-4) **\crv{(1,0)},  
(4,4);(4,-4) **\crv{(-1,0)}, 
(-2.5,1.9)*{\ulcorner}, 
(2.5,-2.4)*{\lrcorner}, 
\endxy~\bigg).$
It follows from (${\bf L2}$) that
\begin{align*}
\ll ~\xy (-4,4);(4,-4) **@{-}, 
(4,4);(-4,-4) **@{-}, 
(3,3.2)*{\llcorner}, 
(-3,-3.4)*{\urcorner}, 
(-2.5,2)*{\ulcorner},
(2.5,-2.4)*{\lrcorner}, 
(3,-0.2);(-3,-0.2) **@{-},
(3,0);(-3,0) **@{-}, 
(3,0.2);(-3,0.2) **@{-}, 
\endxy~\gg 
&=a^{8w(~\xy 
(-4,4);(4,-4) **@{-}, 
(4,4);(-4,-4) **@{-}, 
(3,3.2)*{\llcorner}, 
(-3,-3.4)*{\urcorner}, 
(-2.5,2)*{\ulcorner},
(2.5,-2.4)*{\lrcorner}, 
(3,-0.2);(-3,-0.2) **@{-},
(3,0);(-3,0) **@{-}, 
(3,0.2);(-3,0.2) **@{-}, 
\endxy~)}[[~\xy 
(-4,4);(4,-4) **@{-}, 
(4,4);(-4,-4) **@{-}, 
(3,3.2)*{\llcorner}, 
(-3,-3.4)*{\urcorner}, 
(-2.5,2)*{\ulcorner},
(2.5,-2.4)*{\lrcorner}, 
(3,-0.2);(-3,-0.2) **@{-},
(3,0);(-3,0) **@{-}, 
(3,0.2);(-3,0.2) **@{-}, 
\endxy~]]
=a^{8w(~\xy 
(-4,4);(4,-4) **@{-}, 
(4,4);(-4,-4) **@{-}, 
(3,3.2)*{\llcorner}, 
(-3,-3.4)*{\urcorner}, 
(-2.5,2)*{\ulcorner},
(2.5,-2.4)*{\lrcorner}, 
(3,-0.2);(-3,-0.2) **@{-},
(3,0);(-3,0) **@{-}, 
(3,0.2);(-3,0.2) **@{-}, 
\endxy~)}\bigg(
x[[~\xy 
(-4,4);(4,4) **\crv{(0,-1)}, 
(4,-4);(-4,-4) **\crv{(0,1)},  
(-2.5,1.9)*{\ulcorner}, 
(2.5,-2.4)*{\lrcorner},   
\endxy~]] 
+ y[[~\xy 
(-4,4);(-4,-4) **\crv{(1,0)},  
(4,4);(4,-4) **\crv{(-1,0)}, 
(-2.5,1.9)*{\ulcorner}, 
(2.5,-2.4)*{\lrcorner}, 
\endxy~]]\bigg)\\
&=xa^{8w(~\xy 
(-4,4);(4,-4) **@{-}, 
(4,4);(-4,-4) **@{-}, 
(3,3.2)*{\llcorner}, 
(-3,-3.4)*{\urcorner}, 
(-2.5,2)*{\ulcorner},
(2.5,-2.4)*{\lrcorner}, 
(3,-0.2);(-3,-0.2) **@{-},
(3,0);(-3,0) **@{-}, 
(3,0.2);(-3,0.2) **@{-}, 
\endxy~)}[[~\xy 
(-4,4);(4,4) **\crv{(0,-1)}, 
(4,-4);(-4,-4) **\crv{(0,1)},  
(-2.5,1.9)*{\ulcorner}, 
(2.5,-2.4)*{\lrcorner},   
\endxy~]] 
+ ya^{8w(~\xy 
(-4,4);(4,-4) **@{-}, 
(4,4);(-4,-4) **@{-}, 
(3,3.2)*{\llcorner}, 
(-3,-3.4)*{\urcorner}, 
(-2.5,2)*{\ulcorner},
(2.5,-2.4)*{\lrcorner}, 
(3,-0.2);(-3,-0.2) **@{-},
(3,0);(-3,0) **@{-}, 
(3,0.2);(-3,0.2) **@{-}, 
\endxy~)}[[~\xy 
(-4,4);(-4,-4) **\crv{(1,0)},  
(4,4);(4,-4) **\crv{(-1,0)}, 
(-2.5,1.9)*{\ulcorner}, 
(2.5,-2.4)*{\lrcorner}, 
\endxy~]]\\
&=xa^{8w(~\xy 
(-4,4);(4,4) **\crv{(0,-1)}, 
(4,-4);(-4,-4) **\crv{(0,1)},  
(-2.5,1.9)*{\ulcorner}, 
(2.5,-2.4)*{\lrcorner},   
\endxy~)}[[~\xy 
(-4,4);(4,4) **\crv{(0,-1)}, 
(4,-4);(-4,-4) **\crv{(0,1)},  
(-2.5,1.9)*{\ulcorner}, 
(2.5,-2.4)*{\lrcorner},   
\endxy~]] 
+ ya^{8w(~\xy 
(-4,4);(-4,-4) **\crv{(1,0)},  
(4,4);(4,-4) **\crv{(-1,0)}, 
(-2.5,1.9)*{\ulcorner}, 
(2.5,-2.4)*{\lrcorner}, 
\endxy~)}[[~\xy 
(-4,4);(-4,-4) **\crv{(1,0)},  
(4,4);(4,-4) **\crv{(-1,0)}, 
(-2.5,1.9)*{\ulcorner}, 
(2.5,-2.4)*{\lrcorner}, 
\endxy~]]\\
&=x\ll~\xy 
(-4,4);(4,4) **\crv{(0,-1)}, 
(4,-4);(-4,-4) **\crv{(0,1)},  
(-2.5,1.9)*{\ulcorner}, 
(2.5,-2.4)*{\lrcorner},   
\endxy~\gg 
+ y\ll~\xy 
(-4,4);(-4,-4) **\crv{(1,0)},  
(4,4);(4,-4) **\crv{(-1,0)}, 
(-2.5,1.9)*{\ulcorner}, 
(2.5,-2.4)*{\lrcorner}, 
\endxy~\gg.
\end{align*}
This completes the proof.
\end{proof}


\begin{example} Here are examples of the polynomials for oriented links. 

\begin{itemize}
\item[(1)] 
$
\ll  \xy
(12,2);(10.4,0.4) **@{-}, 
(9.6,-0.4);(8,-2) **@{-},
(12,-2);(8,2) **@{-},
(5.6,-0.4);(4,-2) **@{-}, 
(8,2);(6.6,0.4) **@{-},
(8,-2);(4,2) **@{-},
(4,2);(12,2) **\crv{(8,6)},
(4,-2);(12,-2) **\crv{(8,-6)},
(11,0.5)*{\urcorner}, 
(8.9,0.5)*{\ulcorner}, 
\endxy\gg 
~ =a^{-18}\ll O^2 \gg + 
(a^{-6}-a^{-12})\ll O \gg $ \\ 
$
 = a^{-18}(a^{-6}+1+a^{6})+(a^{-6}-a^{-12}) 
 = a^{-24}+a^{-18}+a^{-6}. 
$

\item[(2)] 
$
\ll \xy
(12,2);(10.4,0.4) **@{-}, 
(9.6,-0.4);(8,-2) **@{-},
(12,-2);(8,2) **@{-},
(5.6,-0.4);(4,-2) **@{-}, 
(8,2);(6.6,0.4) **@{-},
(8,-2);(4,2) **@{-},
(4,2);(12,2) **\crv{(8,6)},
(4,-2);(12,-2) **\crv{(8,-6)}, 
(11,0.5)*{\urcorner}, 
(11,-0.8)*{\lrcorner},  
\endxy\gg
~ = a^{18}\ll O^2 \gg + 
(a^{6}-a^{12}) \ll O \gg $ \\
$  =  a^{18}(a^{-6}+1+a^{6})+(a^{6}-a^{12}) 
 =  a^{24}+a^{18}+a^{6}.
$

\item[(3)]
$
\ll \xy
(12,2);(10.4,0.4) **@{-}, 
(9.6,-0.4);(8,-2) **@{-},
(12,-2);(8,2) **@{-},
(5.6,-0.4);(4,-2) **@{-}, 
(8,2);(6.6,0.4) **@{-},
(8,-2);(4,2) **@{-}, 
(11,0.5)*{\urcorner}, 
(11,-0.8)*{\lrcorner},  
(4,2);(2.4,0.4) **@{-}, 
(1.6,-0.4);(0,-2) **@{-},
(4,-2);(0,2) **@{-},
(0,2);(12,2) **\crv{(5,6)},
(0,-2);(12,-2) **\crv{(5,-6)},
\endxy\gg
 = a^{18}\ll O \gg+
(a^{6}-a^{12})\ll
\xy
(12,2);(10.4,0.4) **@{-}, 
(9.6,-0.4);(8,-2) **@{-},
(12,-2);(8,2) **@{-},
(5.6,-0.4);(4,-2) **@{-}, 
(8,2);(6.6,0.4) **@{-},
(8,-2);(4,2) **@{-},
(4,2);(12,2) **\crv{(8,6)},
(4,-2);(12,-2) **\crv{(8,-6)}, 
(11,0.5)*{\urcorner}, 
(11,-0.8)*{\lrcorner},  
\endxy\gg $ \\
$ =  a^{18}+(a^{6}-a^{12})(a^{24}+a^{18}+a^{6})
 =  a^{12}+a^{24}-a^{36}. 
 $

\item[(4)] 
$ 
\ll  \xy
(12,2);(10.4,0.4) **@{-}, 
(9.6,-0.4);(8,-2) **@{-},
(12,-2);(8,2) **@{-},
(5.6,-0.4);(4,-2) **@{-}, 
(8,2);(6.6,0.4) **@{-},
(8,-2);(4,2) **@{-},
(11,0.5)*{\urcorner}, 
(8.9,0.5)*{\ulcorner}, 
(4,2);(2.4,0.4) **@{-}, 
(1.6,-0.4);(0,-2) **@{-},
(4,-2);(0,2) **@{-},
(-3.6,-0.4);(-4,-2) **@{-}, 
(0,2);(-1.6,0.4) **@{-},
(0,-2);(-4,2) **@{-},
(-2.2,-0.4);(-4,-2) **@{-},
(-4,2);(12,2) **\crv{(5,6)},
(-4,-2);(12,-2) **\crv{(5,-6)},
\endxy\gg
 = a^{-18} \ll \xy
(12,2);(10.4,0.4) **@{-}, 
(9.6,-0.4);(8,-2) **@{-},
(12,-2);(8,2) **@{-},
(5.6,-0.4);(4,-2) **@{-}, 
(8,2);(6.6,0.4) **@{-},
(8,-2);(4,2) **@{-},
(4,2);(12,2) **\crv{(8,6)},
(4,-2);(12,-2) **\crv{(8,-6)},
(11,0.5)*{\urcorner}, (8.9,0.5)*{\ulcorner}, 
\endxy \gg 
+ (a^{-6}-a^{-12}) \ll O \gg $ \\ 
$=  a^{-18}(a^{-24}+a^{-18}+a^{-6})+(a^{-6}-a^{-12})  
 =  a^{-42}+a^{-36}+a^{-24}-a^{-12}+a^{-6}.
 $

\item[(5)] 
$ \ll  \xy
(12,2);(10.4,0.4) **@{-}, 
(9.6,-0.4);(8,-2) **@{-},
(12,-2);(8,2) **@{-},
(5.6,-0.4);(4,-2) **@{-}, 
(8,2);(6.6,0.4) **@{-},
(8,-2);(4,2) **@{-},
(11,0.5)*{\urcorner}, (8.9,0.5)*{\ulcorner}, 
(4,-2);(-2,-2) **\crv{(1,-5)},
(4,4);(2,-2) **\crv{(0,1)},
(-2,2);(1.5,4) **\crv{(-1,4.5)},
(-2,-2);(-2,2) **\crv{(-3,0)},
(4,-4);(12,-2) **\crv{(10,-6)},
(4,4);(12,2) **\crv{(10,6)},
\endxy\gg
 = a^{18} \ll  O \gg
+ (a^6-a^{12}) \ll \xy
(12,2);(10.4,0.4) **@{-}, 
(9.6,-0.4);(8,-2) **@{-},
(12,-2);(8,2) **@{-},
(5.6,-0.4);(4,-2) **@{-}, 
(8,2);(6.6,0.4) **@{-},
(8,-2);(4,2) **@{-},
(4,2);(12,2) **\crv{(8,6)},
(4,-2);(12,-2) **\crv{(8,-6)},
(11,0.5)*{\urcorner}, (8.9,0.5)*{\ulcorner}, 
\endxy\gg $ \\ 
$ =  a^{18} +(a^6-a^{12})(a^{-24}+a^{-18}+a^{-6}) 
 =  a^{-18}-a^{-6}+1-a^{6}+a^{18}. $

\item[(6)] 
$ 
\ll \xy 
(2,0);(4,-2) **@{-}, 
(-2,0);(-4,-2) **@{-},  
(2.4,0.4);(4,2) **@{-},
(0.1,-1.9);(1.7,-0.2) **@{-},  
(0.1,-1.9);(-1.7,-0.3) **@{-},  
(-2.4,0.4);(-4,2) **@{-},
(-8,2);(-6.4,0.4) **@{-}, 
(-5.6,-0.4);(-4,-2) **@{-}, 
(-8,-2);(-4,2) **@{-},
(12,2);(10.4,0.4) **@{-}, 
(9.6,-0.4);(8,-2) **@{-},
(12,-2);(8,2) **@{-},
(5.6,-0.4);(4,-2) **@{-}, 
(8,2);(6.6,0.4) **@{-},
(8,-2);(4,2) **@{-},
(-5,0.4)*{\urcorner}, 
(-5,-0.8)*{\lrcorner}, 
(11,0.5)*{\urcorner}, 
(11,-0.8)*{\lrcorner},  
(-8,2);(-2,4) **\crv{(-10,5)}, 
(2,4);(12,2) **\crv{(14,5)},
(-2,0);(-2,4) **\crv{(1,2)}, 
(2,4);(2,0) **\crv{(-1,2)},
(-8,-2);(0,-5) **\crv{(-10,-5)},
(0,-5);(12,-2) **\crv{(14,-5)},
\endxy\gg
 =a^{-18} \ll
O \sqcup 
\xy
(12,2);(10.4,0.4) **@{-}, 
(9.6,-0.4);(8,-2) **@{-},
(12,-2);(8,2) **@{-},
(5.6,-0.4);(4,-2) **@{-}, 
(8,2);(6.6,0.4) **@{-},
(8,-2);(4,2) **@{-},
(11,0.5)*{\urcorner}, 
(11,-0.8)*{\lrcorner},  
(4,2);(2.4,0.4) **@{-}, 
(1.6,-0.4);(0,-2) **@{-},
(4,-2);(0,2) **@{-},
(0,2);(12,2) **\crv{(5,6)},
(0,-2);(12,-2) **\crv{(5,-6)},
\endxy\gg + (a^{-6}-a^{-12}) 
\ll \xy
(12,2);(10.4,0.4) **@{-}, 
(9.6,-0.4);(8,-2) **@{-},
(12,-2);(8,2) **@{-},
(5.6,-0.4);(4,-2) **@{-}, 
(8,2);(6.6,0.4) **@{-},
(8,-2);(4,2) **@{-},
(11,0.5)*{\urcorner}, 
(11,-0.8)*{\lrcorner},  
(4,2);(2.4,0.4) **@{-}, 
(1.6,-0.4);(0,-2) **@{-},
(4,-2);(0,2) **@{-},
(0,2);(12,2) **\crv{(5,6)},
(0,-2);(12,-2) **\crv{(5,-6)},
\endxy \gg $ \\
$ =  (a^{-18}(a^{-6}+1+a^{6})+a^{-6}-a^{-12})(a^{12}+a^{24}-a^{36})$ \\ 
$ =  (a^{-24}+a^{-18}+a^{-6})(a^{12}+a^{24}-a^{36}). $

\item[(7)] 
$ \ll \xy 
(2,0);(4,-2) **@{-}, 
(-2,0);(-4,-2) **@{-},  
(2.4,0.4);(4,2) **@{-},
(0.1,-1.9);(1.7,-0.2) **@{-},  
(0.1,-1.9);(-1.7,-0.3) **@{-},  
(-2.4,0.4);(-4,2) **@{-},
(-12,2);(-10.4,0.4) **@{-}, 
(-9.6,-0.4);(-8,-2) **@{-}, 
(-12,-2);(-8,2) **@{-},
(-8,2);(-6.4,0.4) **@{-}, 
(-5.6,-0.4);(-4,-2) **@{-}, 
(-8,-2);(-4,2) **@{-},
(12,2);(10.4,0.4) **@{-}, 
(9.6,-0.4);(8,-2) **@{-},
(12,-2);(8,2) **@{-},
(5.6,-0.4);(4,-2) **@{-}, 
(8,2);(6.6,0.4) **@{-},
(8,-2);(4,2) **@{-},
(-9,0.4)*{\urcorner}, 
(-9,-0.8)*{\lrcorner}, 
(11,0.5)*{\urcorner}, 
(11,-0.8)*{\lrcorner},  
(-12,2);(-2,4) **\crv{(-14,5)}, 
(2,4);(12,2) **\crv{(14,5)},
(-2,0);(-2,4) **\crv{(1,2)}, 
(2,4);(2,0) **\crv{(-1,2)},
(-12,-2);(0,-5) **\crv{(-14,-5)},
(0,-5);(12,-2) **\crv{(14,-5)},
 \endxy\gg
=a^{-18}\ll\xy
(12,2);(10.4,0.4) **@{-}, 
(9.6,-0.4);(8,-2) **@{-},
(12,-2);(8,2) **@{-},
(5.6,-0.4);(4,-2) **@{-}, 
(8,2);(6.6,0.4) **@{-},
(8,-2);(4,2) **@{-},
(11,0.5)*{\urcorner}, 
(11,-0.8)*{\lrcorner},  
(4,2);(2.4,0.4) **@{-}, 
(1.6,-0.4);(0,-2) **@{-},
(4,-2);(0,2) **@{-},
(0,2);(12,2) **\crv{(5,6)},
(0,-2);(12,-2) **\crv{(5,-6)},
\endxy\gg
+(a^{-6}-a^{-12})\ll\xy 
(2,0);(4,-2) **@{-}, 
(-2,0);(-4,-2) **@{-},  
(2.4,0.4);(4,2) **@{-},
(0.1,-1.9);(1.7,-0.2) **@{-},  
(0.1,-1.9);(-1.7,-0.3) **@{-},  
(-2.4,0.4);(-4,2) **@{-},
(-8,2);(-6.4,0.4) **@{-}, 
(-5.6,-0.4);(-4,-2) **@{-}, 
(-8,-2);(-4,2) **@{-},
(12,2);(10.4,0.4) **@{-}, 
(9.6,-0.4);(8,-2) **@{-},
(12,-2);(8,2) **@{-},
(5.6,-0.4);(4,-2) **@{-}, 
(8,2);(6.6,0.4) **@{-},
(8,-2);(4,2) **@{-},
(-5,0.4)*{\urcorner}, 
(-5,-0.8)*{\lrcorner}, 
(11,0.5)*{\urcorner}, 
(11,-0.8)*{\lrcorner},  
(-8,2);(-2,4) **\crv{(-10,5)}, 
(2,4);(12,2) **\crv{(14,5)},
(-2,0);(-2,4) **\crv{(1,2)}, 
(2,4);(2,0) **\crv{(-1,2)},
(-8,-2);(0,-5) **\crv{(-10,-5)},
(0,-5);(12,-2) **\crv{(14,-5)},
 \endxy\gg $ \\ 
$ =  a^{-18}(a^{12}+a^{24}-a^{36})  
 + (a^{-6}-a^{-12})(a^{-24}+a^{-18}+a^{-6})(a^{12}+a^{24}-a^{36}) $ \\
$  = (a^{-24}+a^{-12}-a^{-36})(a^{12}+a^{24}-a^{36}). $ 
\end{itemize}

\end{example}


\begin{example}\label{examp-8_1-1}
Consider the diagram $8_1$ of a 
spun $2$-knot of the trefoil in Yoshikawa's table \cite{Yo} with the orientation indicated below. From Theorem \ref{thm-skein-rel}, it follows that
\begin{align*}
\ll~&\xy 
(-12,8);(-2,8) **@{-}, 
(2,8);(12,8) **@{-}, 
(-12,8);(-12,2) **@{-}, 
(12,8);(12,2) **@{-}, 
(-12,-8);(-2,-8) **@{-}, 
(-2,-8);(12,-8) **@{-}, 
(-12,-8);(-12,-2) **@{-}, 
(12,-8);(12,-2) **@{-},
(-2,8);(2,4) **@{-}, 
(-2,4);(2,8) **@{-}, 
(-2,4);(4,-2) **@{-}, 
(2,4);(-4,-2) **@{-}, 
(2.4,0.4);(4,2) **@{-},
(0.1,-1.9);(1.7,-0.2) **@{-},  
(0.1,-1.9);(-1.7,-0.3) **@{-},  
(-2.4,0.4);(-4,2) **@{-},
(-12,2);(-10.4,0.4) **@{-}, 
(-9.6,-0.4);(-8,-2) **@{-}, 
(-12,-2);(-8,2) **@{-},
(-8,2);(-6.4,0.4) **@{-}, 
(-5.6,-0.4);(-4,-2) **@{-}, 
(-8,-2);(-4,2) **@{-},
(12,2);(10.4,0.4) **@{-}, 
(9.6,-0.4);(8,-2) **@{-},
(12,-2);(8,2) **@{-},
(5.6,-0.4);(4,-2) **@{-}, 
(8,2);(6.6,0.4) **@{-},
(8,-2);(4,2) **@{-},
(-0.1,4.5);(-0.1,7.5) **@{-},
(0,4.5);(0,7.5) **@{-},
(0.1,4.5);(0.1,7.5) **@{-},
(-1.5,1.9);(1.5,1.9) **@{-},
(-1.5,2);(1.5,2) **@{-},
(-1.5,2.1);(1.5,2.1) **@{-}, 
(-9,0.4)*{\urcorner}, 
(-9,-0.8)*{\lrcorner}, 
(11,0.5)*{\urcorner}, 
(11,-0.8)*{\lrcorner}, 
(0,12) *{8_1},
\endxy\gg 
= x^2
\ll~\xy 
(-12,8);(-2,8) **@{-}, 
(2,8);(12,8) **@{-}, 
(-12,8);(-12,2) **@{-}, 
(12,8);(12,2) **@{-}, 
(-12,-8);(-2,-8) **@{-}, 
(-2,-8);(12,-8) **@{-}, 
(-12,-8);(-12,-2) **@{-}, 
(12,-8);(12,-2) **@{-},
(2,0);(4,-2) **@{-}, 
(-2,0);(-4,-2) **@{-}, 
(2.4,0.4);(4,2) **@{-},
(0.1,-1.9);(1.7,-0.2) **@{-},  
(0.1,-1.9);(-1.7,-0.3) **@{-},  
(-2.4,0.4);(-4,2) **@{-},
(-12,2);(-10.4,0.4) **@{-}, 
(-9.6,-0.4);(-8,-2) **@{-}, 
(-12,-2);(-8,2) **@{-},
(-8,2);(-6.4,0.4) **@{-}, 
(-5.6,-0.4);(-4,-2) **@{-}, 
(-8,-2);(-4,2) **@{-},
(12,2);(10.4,0.4) **@{-}, 
(9.6,-0.4);(8,-2) **@{-},
(12,-2);(8,2) **@{-},
(5.6,-0.4);(4,-2) **@{-}, 
(8,2);(6.6,0.4) **@{-},
(8,-2);(4,2) **@{-},
(-9,0.4)*{\urcorner}, 
(-9,-0.8)*{\lrcorner}, 
(11,0.5)*{\urcorner}, 
(11,-0.8)*{\lrcorner}, 
(-2,4);(-2,8) **\crv{(1,6)}, 
(2,8);(2,4) **\crv{(-1,6)},
(-2,4);(2,4) **\crv{(0,1)}, 
(-2,0);(2,0) **\crv{(0,3)},
 \endxy\gg 
 +xy
 \ll~\xy 
(-12,8);(-2,8) **@{-}, 
(2,8);(12,8) **@{-}, 
(-12,8);(-12,2) **@{-}, 
(12,8);(12,2) **@{-}, 
(-12,-8);(-2,-8) **@{-}, 
(-2,-8);(12,-8) **@{-}, 
(-12,-8);(-12,-2) **@{-}, 
(12,-8);(12,-2) **@{-},
(2,0);(4,-2) **@{-}, 
(-2,0);(-4,-2) **@{-},  
(2.4,0.4);(4,2) **@{-},
(0.1,-1.9);(1.7,-0.2) **@{-},  
(0.1,-1.9);(-1.7,-0.3) **@{-},  
(-2.4,0.4);(-4,2) **@{-},
(-12,2);(-10.4,0.4) **@{-}, 
(-9.6,-0.4);(-8,-2) **@{-}, 
(-12,-2);(-8,2) **@{-},
(-8,2);(-6.4,0.4) **@{-}, 
(-5.6,-0.4);(-4,-2) **@{-}, 
(-8,-2);(-4,2) **@{-},
(12,2);(10.4,0.4) **@{-}, 
(9.6,-0.4);(8,-2) **@{-},
(12,-2);(8,2) **@{-},
(5.6,-0.4);(4,-2) **@{-}, 
(8,2);(6.6,0.4) **@{-},
(8,-2);(4,2) **@{-},
(-9,0.4)*{\urcorner}, 
(-9,-0.8)*{\lrcorner}, 
(11,0.5)*{\urcorner}, 
(11,-0.8)*{\lrcorner}, 
(-2,4);(-2,8) **\crv{(1,6)}, 
(2,8);(2,4) **\crv{(-1,6)},
(-2,0);(-2,4) **\crv{(1,2)}, 
(2,4);(2,0) **\crv{(-1,2)},
 \endxy\gg \\
 &\hskip 3.3cm + yx
 \ll~\xy 
(-12,8);(-2,8) **@{-}, 
(2,8);(12,8) **@{-}, 
(-12,8);(-12,2) **@{-}, 
(12,8);(12,2) **@{-}, 
(-12,-8);(-2,-8) **@{-}, 
(-2,-8);(12,-8) **@{-}, 
(-12,-8);(-12,-2) **@{-}, 
(12,-8);(12,-2) **@{-},
(2,0);(4,-2) **@{-}, 
(-2,0);(-4,-2) **@{-}, 
(2.4,0.4);(4,2) **@{-},
(0.1,-1.9);(1.7,-0.2) **@{-},  
(0.1,-1.9);(-1.7,-0.3) **@{-},  
(-2.4,0.4);(-4,2) **@{-},
(-12,2);(-10.4,0.4) **@{-}, 
(-9.6,-0.4);(-8,-2) **@{-}, 
(-12,-2);(-8,2) **@{-},
(-8,2);(-6.4,0.4) **@{-}, 
(-5.6,-0.4);(-4,-2) **@{-}, 
(-8,-2);(-4,2) **@{-},
(12,2);(10.4,0.4) **@{-}, 
(9.6,-0.4);(8,-2) **@{-},
(12,-2);(8,2) **@{-},
(5.6,-0.4);(4,-2) **@{-}, 
(8,2);(6.6,0.4) **@{-},
(8,-2);(4,2) **@{-},
(-9,0.4)*{\urcorner}, 
(-9,-0.8)*{\lrcorner}, 
(11,0.5)*{\urcorner}, 
(11,-0.8)*{\lrcorner},  
(-2,8);(2,8) **\crv{(0,5)}, 
(-2,4);(2,4) **\crv{(0,7)},
(-2,4);(2,4) **\crv{(0,1)}, 
(-2,0);(2,0) **\crv{(0,3)},
 \endxy\gg + y^2
 \ll~\xy 
(-12,8);(-2,8) **@{-}, 
(2,8);(12,8) **@{-}, 
(-12,8);(-12,2) **@{-}, 
(12,8);(12,2) **@{-}, 
(-12,-8);(-2,-8) **@{-}, 
(-2,-8);(12,-8) **@{-}, 
(-12,-8);(-12,-2) **@{-}, 
(12,-8);(12,-2) **@{-},
(2,0);(4,-2) **@{-}, 
(-2,0);(-4,-2) **@{-},  
(2.4,0.4);(4,2) **@{-},
(0.1,-1.9);(1.7,-0.2) **@{-},  
(0.1,-1.9);(-1.7,-0.3) **@{-},  
(-2.4,0.4);(-4,2) **@{-},
(-12,2);(-10.4,0.4) **@{-}, 
(-9.6,-0.4);(-8,-2) **@{-}, 
(-12,-2);(-8,2) **@{-},
(-8,2);(-6.4,0.4) **@{-}, 
(-5.6,-0.4);(-4,-2) **@{-}, 
(-8,-2);(-4,2) **@{-},
(12,2);(10.4,0.4) **@{-}, 
(9.6,-0.4);(8,-2) **@{-},
(12,-2);(8,2) **@{-},
(5.6,-0.4);(4,-2) **@{-}, 
(8,2);(6.6,0.4) **@{-},
(8,-2);(4,2) **@{-},
(-9,0.4)*{\urcorner}, 
(-9,-0.8)*{\lrcorner}, 
(11,0.5)*{\urcorner}, 
(11,-0.8)*{\lrcorner}, 
(-2,0);(-2,4) **\crv{(1,2)}, 
(2,4);(2,0) **\crv{(-1,2)},
(-2,8);(2,8) **\crv{(0,5)}, 
(-2,4);(2,4) **\crv{(0,7)},
 \endxy\gg\\
= & x^2 \ll  O^2 \gg + xy \ll \xy 
(2,0);(4,-2) **@{-}, 
(-2,0);(-4,-2) **@{-},  
(2.4,0.4);(4,2) **@{-},
(0.1,-1.9);(1.7,-0.2) **@{-},  
(0.1,-1.9);(-1.7,-0.3) **@{-},  
(-2.4,0.4);(-4,2) **@{-},
(-12,2);(-10.4,0.4) **@{-}, 
(-9.6,-0.4);(-8,-2) **@{-}, 
(-12,-2);(-8,2) **@{-},
(-8,2);(-6.4,0.4) **@{-}, 
(-5.6,-0.4);(-4,-2) **@{-}, 
(-8,-2);(-4,2) **@{-},
(12,2);(10.4,0.4) **@{-}, 
(9.6,-0.4);(8,-2) **@{-},
(12,-2);(8,2) **@{-},
(5.6,-0.4);(4,-2) **@{-}, 
(8,2);(6.6,0.4) **@{-},
(8,-2);(4,2) **@{-},
(-9,0.4)*{\urcorner}, 
(-9,-0.8)*{\lrcorner}, 
(11,0.5)*{\urcorner}, 
(11,-0.8)*{\lrcorner}, 
(-12,2);(-2,4) **\crv{(-14,5)}, 
(2,4);(12,2) **\crv{(14,5)},
(-2,0);(-2,4) **\crv{(1,2)}, 
(2,4);(2,0) **\crv{(-1,2)},
(-12,-2);(0,-5) **\crv{(-14,-5)},
(0,-5);(12,-2) **\crv{(14,-5)},
 \endxy\gg 
+ yx \ll O^3 \gg + y^2 \ll O^2 \gg\\
= & (a^{-6}+1+a^{6})(x^2+y^2)  + (a^{-24}+a^{-12}-a^{-36})(a^{12}+a^{24}-a^{36}) xy \\
& + (a^{-6}+1+a^{6})^2 xy.
 \end{align*}
 \end{example}

\begin{example}\label{examp-9_1-1}
Consider the diagram $9_1$ of a 
ribbon $2$-knot  associated with $6_1$ knot in Yoshikawa's table \cite{Yo} with the orientation indicated below.  
\begin{align*}
\ll~&\xy 
(-2,12);(2,8) **@{-}, 
(-2,8);(-0.7,9.5) **@{-},
(0.7,10.6);(2,12) **@{-},
(-12,12);(-2,12) **@{-}, 
(2,12);(12,12) **@{-}, 
(-12,12);(-12,2) **@{-}, 
(12,12);(12,2) **@{-},
(-12,-8);(-2,-8) **@{-}, 
(-2,-8);(12,-8) **@{-}, 
(-12,-8);(-12,-2) **@{-}, 
(12,-8);(12,-2) **@{-},
(-2,8);(2,4) **@{-}, 
(-2,4);(2,8) **@{-}, 
(-2,4);(4,-2) **@{-}, 
(2,4);(-4,-2) **@{-}, 
(2.4,0.4);(4,2) **@{-},
(0.1,-1.9);(1.7,-0.2) **@{-},  
(0.1,-1.9);(-1.7,-0.3) **@{-},  
(-2.4,0.4);(-4,2) **@{-},
(-12,2);(-10.4,0.4) **@{-}, 
(-9.6,-0.4);(-8,-2) **@{-}, 
(-12,-2);(-8,2) **@{-},
(-8,2);(-6.4,0.4) **@{-}, 
(-5.6,-0.4);(-4,-2) **@{-}, 
(-8,-2);(-4,2) **@{-},
(12,2);(10.4,0.4) **@{-}, 
(9.6,-0.4);(8,-2) **@{-},
(12,-2);(8,2) **@{-},
(5.6,-0.4);(4,-2) **@{-}, 
(8,2);(6.6,0.4) **@{-},
(8,-2);(4,2) **@{-},
(-0.1,4.5);(-0.1,7.5) **@{-},
(0,4.5);(0,7.5) **@{-},
(0.1,4.5);(0.1,7.5) **@{-},
(-1.5,1.9);(1.5,1.9) **@{-},
(-1.5,2);(1.5,2) **@{-},
(-1.5,2.1);(1.5,2.1) **@{-},
(-3,-1.2)*{\llcorner}, 
(-11,-1.2)*{\llcorner}, 
(-8.9,-1.2)*{\lrcorner}, 
(11,0.5)*{\urcorner}, 
(8.9,0.5)*{\ulcorner}, 
(0,16) *{9_1},
\endxy\gg 
= x^2
\ll~\xy 
(-2,12);(2,8) **@{-}, 
(-2,8);(-0.7,9.5) **@{-},
(0.7,10.6);(2,12) **@{-},
(-12,12);(-2,12) **@{-}, 
(2,12);(12,12) **@{-}, 
(-12,12);(-12,2) **@{-}, 
(12,12);(12,2) **@{-},
(-12,-8);(-2,-8) **@{-}, 
(-2,-8);(12,-8) **@{-}, 
(-12,-8);(-12,-2) **@{-}, 
(12,-8);(12,-2) **@{-},
(2,0);(4,-2) **@{-}, 
(-2,0);(-4,-2) **@{-}, 
(2.4,0.4);(4,2) **@{-},
(0.1,-1.9);(1.7,-0.2) **@{-},  
(0.1,-1.9);(-1.7,-0.3) **@{-},  
(-2.4,0.4);(-4,2) **@{-},
(-12,2);(-10.4,0.4) **@{-}, 
(-9.6,-0.4);(-8,-2) **@{-}, 
(-12,-2);(-8,2) **@{-},
(-8,2);(-6.4,0.4) **@{-}, 
(-5.6,-0.4);(-4,-2) **@{-}, 
(-8,-2);(-4,2) **@{-},
(12,2);(10.4,0.4) **@{-}, 
(9.6,-0.4);(8,-2) **@{-},
(12,-2);(8,2) **@{-},
(5.6,-0.4);(4,-2) **@{-}, 
(8,2);(6.6,0.4) **@{-},
(8,-2);(4,2) **@{-},
(-3,-1.2)*{\llcorner}, (-11,-1.2)*{\llcorner}, (-8.9,-1.2)*{\lrcorner}, (11,0.5)*{\urcorner}, (8.9,0.5)*{\ulcorner}, 
(-2,4);(-2,8) **\crv{(1,6)}, 
(2,8);(2,4) **\crv{(-1,6)},
(-2,4);(2,4) **\crv{(0,1)}, 
(-2,0);(2,0) **\crv{(0,3)},
 \endxy\gg 
 +xy
 \ll~\xy 
(-2,12);(2,8) **@{-}, 
(-2,8);(-0.7,9.5) **@{-},
(0.7,10.6);(2,12) **@{-},
(-12,12);(-2,12) **@{-}, 
(2,12);(12,12) **@{-}, 
(-12,12);(-12,2) **@{-}, 
(12,12);(12,2) **@{-},
(-12,-8);(-2,-8) **@{-}, 
(-2,-8);(12,-8) **@{-}, 
(-12,-8);(-12,-2) **@{-}, 
(12,-8);(12,-2) **@{-},
(2,0);(4,-2) **@{-}, 
(-2,0);(-4,-2) **@{-},  
(2.4,0.4);(4,2) **@{-},
(0.1,-1.9);(1.7,-0.2) **@{-},  
(0.1,-1.9);(-1.7,-0.3) **@{-},  
(-2.4,0.4);(-4,2) **@{-},
(-12,2);(-10.4,0.4) **@{-}, 
(-9.6,-0.4);(-8,-2) **@{-}, 
(-12,-2);(-8,2) **@{-},
(-8,2);(-6.4,0.4) **@{-}, 
(-5.6,-0.4);(-4,-2) **@{-}, 
(-8,-2);(-4,2) **@{-},
(12,2);(10.4,0.4) **@{-}, 
(9.6,-0.4);(8,-2) **@{-},
(12,-2);(8,2) **@{-},
(5.6,-0.4);(4,-2) **@{-}, 
(8,2);(6.6,0.4) **@{-},
(8,-2);(4,2) **@{-},
(-3,-1.2)*{\llcorner}, (-11,-1.2)*{\llcorner}, (-8.9,-1.2)*{\lrcorner}, (11,0.5)*{\urcorner}, (8.9,0.5)*{\ulcorner}, 
(-2,4);(-2,8) **\crv{(1,6)}, 
(2,8);(2,4) **\crv{(-1,6)},
(-2,0);(-2,4) **\crv{(1,2)}, 
(2,4);(2,0) **\crv{(-1,2)},
 \endxy\gg \\
 &\hskip 3.3cm + yx
 \ll~\xy 
(-2,12);(2,8) **@{-}, 
(-2,8);(-0.7,9.5) **@{-},
(0.7,10.6);(2,12) **@{-},
(-12,12);(-2,12) **@{-}, 
(2,12);(12,12) **@{-}, 
(-12,12);(-12,2) **@{-}, 
(12,12);(12,2) **@{-},
(-12,-8);(-2,-8) **@{-}, 
(-2,-8);(12,-8) **@{-}, 
(-12,-8);(-12,-2) **@{-}, 
(12,-8);(12,-2) **@{-},
(2,0);(4,-2) **@{-}, 
(-2,0);(-4,-2) **@{-}, 
(2.4,0.4);(4,2) **@{-},
(0.1,-1.9);(1.7,-0.2) **@{-},  
(0.1,-1.9);(-1.7,-0.3) **@{-},  
(-2.4,0.4);(-4,2) **@{-},
(-12,2);(-10.4,0.4) **@{-}, 
(-9.6,-0.4);(-8,-2) **@{-}, 
(-12,-2);(-8,2) **@{-},
(-8,2);(-6.4,0.4) **@{-}, 
(-5.6,-0.4);(-4,-2) **@{-}, 
(-8,-2);(-4,2) **@{-},
(12,2);(10.4,0.4) **@{-}, 
(9.6,-0.4);(8,-2) **@{-},
(12,-2);(8,2) **@{-},
(5.6,-0.4);(4,-2) **@{-}, 
(8,2);(6.6,0.4) **@{-},
(8,-2);(4,2) **@{-},
(-3,-1.2)*{\llcorner}, 
(-11,-1.2)*{\llcorner}, 
(-8.9,-1.2)*{\lrcorner}, 
(11,0.5)*{\urcorner}, 
(8.9,0.5)*{\ulcorner}, 
(-2,8);(2,8) **\crv{(0,5)}, 
(-2,4);(2,4) **\crv{(0,7)},
(-2,4);(2,4) **\crv{(0,1)}, 
(-2,0);(2,0) **\crv{(0,3)},
 \endxy\gg + y^2
 \ll~\xy 
(-2,12);(2,8) **@{-}, 
(-2,8);(-0.7,9.5) **@{-},
(0.7,10.6);(2,12) **@{-},
(-12,12);(-2,12) **@{-}, 
(2,12);(12,12) **@{-}, 
(-12,12);(-12,2) **@{-}, 
(12,12);(12,2) **@{-},
(-12,-8);(-2,-8) **@{-}, 
(-2,-8);(12,-8) **@{-}, 
(-12,-8);(-12,-2) **@{-}, 
(12,-8);(12,-2) **@{-},
(2,0);(4,-2) **@{-}, 
(-2,0);(-4,-2) **@{-},  
(2.4,0.4);(4,2) **@{-},
(0.1,-1.9);(1.7,-0.2) **@{-},  
(0.1,-1.9);(-1.7,-0.3) **@{-},  
(-2.4,0.4);(-4,2) **@{-},
(-12,2);(-10.4,0.4) **@{-}, 
(-9.6,-0.4);(-8,-2) **@{-}, 
(-12,-2);(-8,2) **@{-},
(-8,2);(-6.4,0.4) **@{-}, 
(-5.6,-0.4);(-4,-2) **@{-}, 
(-8,-2);(-4,2) **@{-},
(12,2);(10.4,0.4) **@{-}, 
(9.6,-0.4);(8,-2) **@{-},
(12,-2);(8,2) **@{-},
(5.6,-0.4);(4,-2) **@{-}, 
(8,2);(6.6,0.4) **@{-},
(8,-2);(4,2) **@{-},
(-3,-1.2)*{\llcorner}, (-11,-1.2)*{\llcorner}, (-8.9,-1.2)*{\lrcorner}, (11,0.5)*{\urcorner}, (8.9,0.5)*{\ulcorner}, 
(-2,0);(-2,4) **\crv{(1,2)}, 
(2,4);(2,0) **\crv{(-1,2)},
(-2,8);(2,8) **\crv{(0,5)}, 
(-2,4);(2,4) **\crv{(0,7)},
 \endxy\gg\\
= & x^2 \ll O^2 \gg 
+ xy \bigg[a^{18}
\ll\xy
(12,2);(10.4,0.4) **@{-}, 
(9.6,-0.4);(8,-2) **@{-},
(12,-2);(8,2) **@{-},
(5.6,-0.4);(4,-2) **@{-}, 
(8,2);(6.6,0.4) **@{-},
(8,-2);(4,2) **@{-},
(11,0.5)*{\urcorner}, (8.9,0.5)*{\ulcorner}, 
(4,-2);(-2,-2) **\crv{(1,-5)},
(4,4);(2,-2) **\crv{(0,1)},
(-2,2);(1.5,4) **\crv{(-1,4.5)},
(-2,-2);(-2,2) **\crv{(-3,0)},
(4,-4);(12,-2) **\crv{(10,-6)},
(4,4);(12,2) **\crv{(10,6)},
\endxy
\gg 
+ (a^6-a^{12})\ll
\xy
(12,2);(10.4,0.4) **@{-}, 
(9.6,-0.4);(8,-2) **@{-},
(12,-2);(8,2) **@{-},
(5.6,-0.4);(4,-2) **@{-}, 
(8,2);(6.6,0.4) **@{-},
(8,-2);(4,2) **@{-},
(4,2);(12,2) **\crv{(8,6)},
(4,-2);(12,-2) **\crv{(8,-6)},
(11,0.5)*{\urcorner}, (8.9,0.5)*{\ulcorner}, 
\endxy
\gg\bigg]\\
&  + yx \ll O^3 \gg
+ y^2 \ll  O^2 \gg\\
= &  (a^{-6}+1+a^{6})(x^2+y^2)+(a^{-6}+1+a^{6})^2 xy \\
& + \bigg[a^{18}(a^{-18}-a^{-6}+1-a^{6}+a^{18}) +(a^6-a^{12})(a^{-24}+a^{-18}+a^{-6})\bigg] xy\\ 
= &  (a^{-6}+1+a^{6})(x^2+y^2) \\
& +  (a^{-18}+a^{-12}+a^{-6}+5+a^6+a^{18}-a^{24}+a^{36}) xy.\\
 \end{align*}

\end{example}

\begin{example}\label{examp-10_2-1} 
Consider the diagram $10_2$ of 
a $2$-twist spun $2$-knot  of the trefoil in Yoshikawa's table \cite{Yo} with the orientation indicated below.  
\begin{align*}
\ll~&\xy (-12,8);(-2,8) **@{-}, (2,8);(12,8) **@{-}, (-12,8);(-12,2) **@{-}, 
(12,8);(12,2) **@{-}, (-12,-8);(-2,-8) **@{-}, (2,-8);(12,-8) **@{-}, 
(-12,-8);(-12,-2) **@{-}, (12,-8);(12,-2) **@{-},
(-2,8);(2,4) **@{-}, (-2,4);(-0.4,5.4) **@{-},
(0.4,6.4);(2,8) **@{-}, (-2,4);(4,-2) **@{-}, (2,4);(-4,-2) **@{-}, (-2,-8);(2,-4) **@{-}, (-2,-4);(2,-8) **@{-},
(-2,-4);(1.4,-0.4) **@{-}, (2.4,0.4);(4,2) **@{-},
(2,-4);(0.6,-2.4) **@{-}, (-0.4,-1.6);(-1.7,-0.3) **@{-},  (-2.4,0.4);(-4,2) **@{-},
(-12,2);(-10.4,0.4) **@{-}, (-9.6,-0.4);(-8,-2) **@{-}, (-12,-2);(-8,2) **@{-},
(-8,2);(-6.4,0.4) **@{-}, (-5.6,-0.4);(-4,-2) **@{-}, (-8,-2);(-4,2) **@{-},
(12,2);(10.4,0.4) **@{-}, (9.6,-0.4);(8,-2) **@{-},
(12,-2);(8,2) **@{-},
(5.6,-0.4);(4,-2) **@{-}, (8,2);(6.6,0.4) **@{-},
(8,-2);(4,2) **@{-},
(-1.5,1.9);(1.5,1.9) **@{-},
(-1.5,2);(1.5,2) **@{-},
(-1.5,2.1);(1.5,2.1) **@{-},
(-1.5,-5.9);(1.5,-5.9) **@{-},
(-1.5,-6);(1.5,-6) **@{-},
(-1.5,-6.1);(1.5,-6.1) **@{-},
(-1,6.4)*{\ulcorner}, (-1.4,-8)*{\urcorner}, (-3,-1.2)*{\llcorner}, (-11,-1.2)*{\llcorner}, (-8.9,-1.2)*{\lrcorner}, (11,0.5)*{\urcorner}, (8.9,0.5)*{\ulcorner}, (1,-3)*{\lrcorner},
(0,12) *{10_2},
 \endxy\gg 
 = x^2 \ll~\xy (-12,8);(-2,8) **@{-}, (2,8);(12,8) **@{-}, (-12,8);(-12,2) **@{-}, (12,8);(12,2) **@{-}, (-12,-8);(-2,-8) **@{-}, (2,-8);(12,-8) **@{-}, (-12,-8);(-12,-2) **@{-}, (12,-8);(12,-2) **@{-},
(-2,8);(2,4) **@{-}, (-2,4);(-0.4,5.4) **@{-},
(0.4,6.4);(2,8) **@{-}, (2,0);(4,-2) **@{-}, (-2,0);(-4,-2) **@{-},
(-2,-4);(1.4,-0.4) **@{-}, (2.4,0.4);(4,2) **@{-},
(2,-4);(0.6,-2.4) **@{-}, (-0.4,-1.6);(-1.7,-0.3) **@{-},  (-2.4,0.4);(-4,2) **@{-},
(-12,2);(-10.4,0.4) **@{-}, (-9.6,-0.4);(-8,-2) **@{-}, (-12,-2);(-8,2) **@{-},
(-8,2);(-6.4,0.4) **@{-}, (-5.6,-0.4);(-4,-2) **@{-}, (-8,-2);(-4,2) **@{-},
(12,2);(10.4,0.4) **@{-}, (9.6,-0.4);(8,-2) **@{-},
(12,-2);(8,2) **@{-},
(5.6,-0.4);(4,-2) **@{-}, (8,2);(6.6,0.4) **@{-},
(8,-2);(4,2) **@{-},
(-1,6.4)*{\ulcorner}, (-1.4,-8)*{\urcorner}, (-3,-1.2)*{\llcorner}, (-11,-1.2)*{\llcorner}, (-8.9,-1.2)*{\lrcorner}, (11,0.5)*{\urcorner}, (8.9,0.5)*{\ulcorner}, (1,-3)*{\lrcorner},
(-2,4);(2,4) **\crv{(0,1)}, (-2,0);(2,0) **\crv{(0,3)},
(-2,-4);(2,-4) **\crv{(0,-7)}, (-2,-8);(2,-8) **\crv{(0,-5)},
 \endxy~\gg
 + xy \ll~\xy (-12,8);(-2,8) **@{-}, (2,8);(12,8) **@{-}, (-12,8);(-12,2) **@{-}, (12,8);(12,2) **@{-}, (-12,-8);(-2,-8) **@{-}, (2,-8);(12,-8) **@{-}, (-12,-8);(-12,-2) **@{-}, (12,-8);(12,-2) **@{-},
(-2,8);(2,4) **@{-}, (-2,4);(-0.4,5.4) **@{-},
(0.4,6.4);(2,8) **@{-}, (2,0);(4,-2) **@{-}, (-2,0);(-4,-2) **@{-},
(-2,-4);(1.4,-0.4) **@{-}, (2.4,0.4);(4,2) **@{-},
(2,-4);(0.6,-2.4) **@{-}, (-0.4,-1.6);(-1.7,-0.3) **@{-},  (-2.4,0.4);(-4,2) **@{-},
(-12,2);(-10.4,0.4) **@{-}, (-9.6,-0.4);(-8,-2) **@{-}, (-12,-2);(-8,2) **@{-},
(-8,2);(-6.4,0.4) **@{-}, (-5.6,-0.4);(-4,-2) **@{-}, (-8,-2);(-4,2) **@{-},
(12,2);(10.4,0.4) **@{-}, (9.6,-0.4);(8,-2) **@{-},
(12,-2);(8,2) **@{-},
(5.6,-0.4);(4,-2) **@{-}, (8,2);(6.6,0.4) **@{-},
(8,-2);(4,2) **@{-},
(-1,6.4)*{\ulcorner}, (-1.4,-8)*{\urcorner}, (-3,-1.2)*{\llcorner}, (-11,-1.2)*{\llcorner}, (-8.9,-1.2)*{\lrcorner}, (11,0.5)*{\urcorner}, (8.9,0.5)*{\ulcorner}, (1,-3)*{\lrcorner},
(-2,4);(2,4) **\crv{(0,1)}, (-2,0);(2,0) **\crv{(0,3)},
(-2,-4);(-2,-8) **\crv{(1,-6)}, (2,-8);(2,-4) **\crv{(-1,-6)},
 \endxy~\gg\\
&\hskip 2.95cm +yx \ll\xy (-12,8);(-2,8) **@{-}, (2,8);(12,8) **@{-}, (-12,8);(-12,2) **@{-}, (12,8);(12,2) **@{-}, (-12,-8);(-2,-8) **@{-}, (2,-8);(12,-8) **@{-}, (-12,-8);(-12,-2) **@{-}, (12,-8);(12,-2) **@{-},
(-2,8);(2,4) **@{-}, (-2,4);(-0.4,5.4) **@{-},
(0.4,6.4);(2,8) **@{-}, (2,0);(4,-2) **@{-}, (-2,0);(-4,-2) **@{-},
(-2,-4);(1.4,-0.4) **@{-}, (2.4,0.4);(4,2) **@{-},
(2,-4);(0.6,-2.4) **@{-}, (-0.4,-1.6);(-1.7,-0.3) **@{-},  (-2.4,0.4);(-4,2) **@{-},
(-12,2);(-10.4,0.4) **@{-}, (-9.6,-0.4);(-8,-2) **@{-}, (-12,-2);(-8,2) **@{-},
(-8,2);(-6.4,0.4) **@{-}, (-5.6,-0.4);(-4,-2) **@{-}, (-8,-2);(-4,2) **@{-},
(12,2);(10.4,0.4) **@{-}, (9.6,-0.4);(8,-2) **@{-},
(12,-2);(8,2) **@{-},
(5.6,-0.4);(4,-2) **@{-}, (8,2);(6.6,0.4) **@{-},
(8,-2);(4,2) **@{-},
(-1,6.4)*{\ulcorner}, (-1.4,-8)*{\urcorner}, (-3,-1.2)*{\llcorner}, (-11,-1.2)*{\llcorner}, (-8.9,-1.2)*{\lrcorner}, (11,0.5)*{\urcorner}, (8.9,0.5)*{\ulcorner}, (1,-3)*{\lrcorner},
(-2,4);(-2,0) **\crv{(1,2)}, (2,0);(2,4) **\crv{(-1,2)},
(-2,-4);(2,-4) **\crv{(0,-7)}, (-2,-8);(2,-8) **\crv{(0,-5)},
 \endxy\gg
 ~+y^2 \ll\xy (-12,8);(-2,8) **@{-}, (2,8);(12,8) **@{-}, (-12,8);(-12,2) **@{-}, (12,8);(12,2) **@{-}, (-12,-8);(-2,-8) **@{-}, (2,-8);(12,-8) **@{-}, (-12,-8);(-12,-2) **@{-}, (12,-8);(12,-2) **@{-},
(-2,8);(2,4) **@{-}, (-2,4);(-0.4,5.4) **@{-},
(0.4,6.4);(2,8) **@{-}, (2,0);(4,-2) **@{-}, (-2,0);(-4,-2) **@{-},
(-2,-4);(1.4,-0.4) **@{-}, (2.4,0.4);(4,2) **@{-},
(2,-4);(0.6,-2.4) **@{-}, (-0.4,-1.6);(-1.7,-0.3) **@{-},  (-2.4,0.4);(-4,2) **@{-},
(-12,2);(-10.4,0.4) **@{-}, (-9.6,-0.4);(-8,-2) **@{-}, (-12,-2);(-8,2) **@{-},
(-8,2);(-6.4,0.4) **@{-}, (-5.6,-0.4);(-4,-2) **@{-}, (-8,-2);(-4,2) **@{-},
(12,2);(10.4,0.4) **@{-}, (9.6,-0.4);(8,-2) **@{-},
(12,-2);(8,2) **@{-},
(5.6,-0.4);(4,-2) **@{-}, (8,2);(6.6,0.4) **@{-},
(8,-2);(4,2) **@{-},
(-1,6.4)*{\ulcorner}, (-1.4,-8)*{\urcorner}, (-3,-1.2)*{\llcorner}, (-11,-1.2)*{\llcorner}, (-8.9,-1.2)*{\lrcorner}, (11,0.5)*{\urcorner}, (8.9,0.5)*{\ulcorner}, (1,-3)*{\lrcorner},
(-2,4);(-2,0) **\crv{(1,2)}, (2,0);(2,4) **\crv{(-1,2)},
(-2,-4);(-2,-8) **\crv{(1,-6)}, (2,-8);(2,-4) **\crv{(-1,-6)},
 \endxy~\gg\\
= & x^2 \ll O^2 \gg
+ xy \bigg[a^{18}\ll O \gg + (a^6-a^{12}) \ll \xy
(12,2);(10.4,0.4) **@{-}, 
(9.6,-0.4);(8,-2) **@{-},
(12,-2);(8,2) **@{-},
(5.6,-0.4);(4,-2) **@{-}, 
(8,2);(6.6,0.4) **@{-},
(8,-2);(4,2) **@{-},
(11,0.5)*{\urcorner}, (8.9,0.5)*{\ulcorner}, 
(4,2);(2.4,0.4) **@{-}, 
(1.6,-0.4);(0,-2) **@{-},
(4,-2);(0,2) **@{-},
(-3.6,-0.4);(-4,-2) **@{-}, 
(0,2);(-1.6,0.4) **@{-},
(0,-2);(-4,2) **@{-},
(-2.2,-0.4);(-4,-2) **@{-},
(-4,2);(12,2) **\crv{(5,6)},
(-4,-2);(12,-2) **\crv{(5,-6)},
\endxy\gg\bigg]\\
&+ xy\bigg[a^{18}
\ll\xy
(12,2);(10.4,0.4) **@{-}, 
(9.6,-0.4);(8,-2) **@{-},
(12,-2);(8,2) **@{-},
(5.6,-0.4);(4,-2) **@{-}, 
(8,2);(6.6,0.4) **@{-},
(8,-2);(4,2) **@{-},
(11,0.5)*{\urcorner}, (8.9,0.5)*{\ulcorner}, 
(4,-2);(-2,-2) **\crv{(1,-5)},
(4,4);(2,-2) **\crv{(0,1)},
(-2,2);(1.5,4) **\crv{(-1,4.5)},
(-2,-2);(-2,2) **\crv{(-3,0)},
(4,-4);(12,-2) **\crv{(10,-6)},
(4,4);(12,2) **\crv{(10,6)},
\endxy
\gg
+(a^6-a^{12}) \ll
\xy
(12,2);(10.4,0.4) **@{-}, 
(9.6,-0.4);(8,-2) **@{-},
(12,-2);(8,2) **@{-},
(5.6,-0.4);(4,-2) **@{-}, 
(8,2);(6.6,0.4) **@{-},
(8,-2);(4,2) **@{-},
(4,2);(12,2) **\crv{(8,6)},
(4,-2);(12,-2) **\crv{(8,-6)},
(11,0.5)*{\urcorner}, (8.9,0.5)*{\ulcorner}, 
\endxy
\gg\bigg]\\
& + y^2 \bigg[ a^{18} \ll \xy
(12,2);(10.4,0.4) **@{-}, 
(9.6,-0.4);(8,-2) **@{-},
(12,-2);(8,2) **@{-},
(5.6,-0.4);(4,-2) **@{-}, 
(8,2);(6.6,0.4) **@{-},
(8,-2);(4,2) **@{-},
(4,2);(12,2) **\crv{(8,6)},
(4,-2);(12,-2) **\crv{(8,-6)},
(11,0.5)*{\urcorner}, (8.9,0.5)*{\ulcorner}, 
\endxy\gg
+ (a^6-a^{12}) \ll O \gg \bigg]\\
= 
& x^2(a^{-6}+1+a^6) +xy \bigg[ a^{18} + 
(a^6-a^{12})(a^{-42}+a^{-36}+a^{-24}-a^{-12}+a^{-6}) \bigg] \\
& + xy \bigg[ a^{18}(a^{-18}-a^{-6}+1-a^{6}+a^{18})   
+(a^6-a^{12})(a^{-24}+a^{-18}+a^{-6}) \bigg] \\
& + y^2\bigg[a^{18}(a^{-24}+a^{-18}+a^{-6})+(a^6-a^{12}) \bigg] \\
=
& (a^{-6}+1+a^6)(x^2+y^2) \\ 
& +(a^{-36}-a^{-24}+2a^{-18} - a^{-12}-2a^{-6}+4-2a^{6}
-a^{12}+2a^{18}-a^{24}+a^{36}) xy. 
\end{align*}
\end{example}


\section{Behaviors under Yoshikawa moves $\Gamma_6$, $\Gamma'_6$, $\Gamma_7$ and $\Gamma_8$}\label{sect-on-inv-osl}

In this section we investigate behaviors of $\ll\cdot\gg$ under Yoshikawa moves $\Gamma_6$, $\Gamma'_6$, $\Gamma_7$ and $\Gamma_8$.  

\begin{proposition}\label{lem-m06} 
The moves $\Gamma_6$ and $\Gamma'_6$ change  
the polynomial $\ll\cdot\gg$ as follows.  
\begin{align*}
&\ll~\xy
(1,3);(6,-2) **@{-}, 
(1,-3);(6,2) **@{-},
(6,2);(10,2) **\crv{(8,4)}, 
(6,-2);(10,-2) **\crv{(8,-4)}, 
(10,2);(10,-2) **\crv{(11.5,0)}, 
(1.5,-3) *{\urcorner}, 
(1.5,1.9) *{\ulcorner},
(9.5,-3)*{\urcorner},
(4,2);(4,-2) **@{-}, 
(4.1,2);(4.1,-2) **@{-}, 
(3.9,2);(3.9,-2) **@{-},
\endxy~\gg 
=\bigg((a^{-6}+1+a^6)x+y\bigg)
\ll~\xy (1.5,-3);(1.5,3) **\crv{(4,0)}, 
(2.4,1.5) *{\ulcorner}, 
\endxy~\gg,
\\
&\ll~\xy (1,3);(6,-2) **@{-}, 
(1,-3);(6,2) **@{-},
(6,2);(10,2) **\crv{(8,4)}, 
(6,-2);(10,-2)**\crv{(8,-4)}, 
(10,2);(10,-2) **\crv{(11.5,0)}, 
(1.5,-3) *{\urcorner}, 
(1.5,1.9) *{\ulcorner},
(9.5,-3)*{\urcorner},
(2,0);(6,0) **@{-}, 
(2,0.1);(6,0.1) **@{-},
(2,0.2);(6,0.2) **@{-}, 
(2,-0.1);(6,-0.1) **@{-},
(2,-0.2);(6,-0.2) **@{-},
\endxy~\gg 
=\bigg(x+(a^{-6}+1+a^6)y\bigg)
\ll~\xy (1.5,-3);(1.5,3) **\crv{(4,0)}, 
(2.4,1.5) *{\ulcorner}, 
\endxy~\gg.
\end{align*}
\end{proposition}

\begin{proof}
For $\Gamma_6$ and $\Gamma'_6$, we have
\begin{align*}
\ll ~\xy
(1,3);(6,-2) **@{-}, 
(1,-3);(6,2) **@{-},
(6,2);(10,2) **\crv{(8,4)}, 
(6,-2);(10,-2)**\crv{(8,-4)}, 
(10,2);(10,-2) **\crv{(11.5,0)}, 
(1.5,-3) *{\urcorner}, 
(1.5,1.9) *{\ulcorner},
(9.5,-3)*{\urcorner},
(4,2);(4,-2) **@{-}, 
(4.1,2);(4.1,-2) **@{-}, 
(3.9,2);(3.9,-2)**@{-},
\endxy~ \gg 
&= x \ll ~\xy 
(1.5,-3);(1.5,3) **\crv{(4,0)}, 
(2.4,1.5) *{\ulcorner}, 
(6,2);(6,-2) **\crv{(4.5,0)},
(6,2);(10,2) **\crv{(8,4)},
(6,-2);(10,-2) **\crv{(8,-4)}, 
(10,2);(10,-2) **\crv{(11.5,0)},
(5.6,-1.5) *{\lrcorner},
\endxy~ \gg  
+ y \ll ~\xy 
(1.5,-3);(1.5,3) **\crv{(4,0)}, 
(2.4,1.5) *{\ulcorner}, 
\endxy~ \gg 
=\bigg((a^{-6}+1+a^6)x+y\bigg) \ll ~\xy 
(1.5,-3);(1.5,3) **\crv{(4,0)}, 
(2.4,1.5) *{\ulcorner}, 
\endxy~ \gg, \\ 
\ll ~\xy 
(1,3);(6,-2) **@{-}, 
(1,-3);(6,2) **@{-},
(6,2);(10,2) **\crv{(8,4)}, 
(6,-2);(10,-2)**\crv{(8,-4)}, 
(10,2);(10,-2) **\crv{(11.5,0)}, 
(1.5,-3) *{\urcorner}, 
(1.5,1.9) *{\ulcorner},
(9.5,-3)*{\urcorner},
(2,0);(6,0) **@{-}, 
(2,0.1);(6,0.1) **@{-},
(2,0.2);(6,0.2) **@{-}, 
(2,-0.1);(6,-0.1) **@{-},
(2,-0.2);(6,-0.2) **@{-},
\endxy~ \gg  
&= x \ll ~\xy 
(1.5,-3);(1.5,3) **\crv{(4,0)}, 
(2.4,1.5) *{\ulcorner}, 
\endxy~ \gg 
+ y \ll ~\xy 
(6,2);(6,-2) **\crv{(4.5,0)},
(6,2);(10,2) **\crv{(8,4)},
(6,-2);(10,-2) **\crv{(8,-4)}, 
(10,2);(10,-2) **\crv{(11.5,0)},
(5.6,-1.5) *{\lrcorner},
(1.5,-3);(1.5,3) **\crv{(4,0)}, 
(2.4,1.5) *{\ulcorner}, 
\endxy~ \gg
=\bigg(x+(a^{-6}+1+a^6)y\bigg) \ll ~\xy 
(1.5,-3);(1.5,3) **\crv{(4,0)}, 
(2.4,1.5) *{\ulcorner}, 
\endxy~\gg. 
\end{align*}
This completes the proof. 
\end{proof}

In order to consider behaviors of the polynomial $\ll\cdot\gg$ under Yoshikawa moves $\Gamma_7$ and $\Gamma_8$, we prepare lemmas. 

By an {\it $n$-tangle diagram} ($n\geq 1$) we mean an oriented link diagram or a tangled trivalent graph diagram $\mathcal T$ in the rectangle $I^2=[0,1]\times [0,1]$ in $\mathbb R^2$ such that $\mathcal T$ transversely intersect with $(0,1)\times\{0\}$ and $(0,1)\times\{1\}$ in $n$ distinct points, respectively, called the {\it endpoints} of $\mathcal T$. The {\it boundary} of an $n$-tangle diagram $\mathcal T$ is defined to be the boundary of $I^2$ together with the $2n$ endpoints equipped with inward or outward pointing normals that coincide with the orientations on intersecting arcs of $\mathcal T$. In Figure \ref{fig-bdary-t}, (a) is the boundary of a $3$-tangle diagram, and (b) is the boundary of a $4$-tangle diagram. 

\begin{figure}[ht]
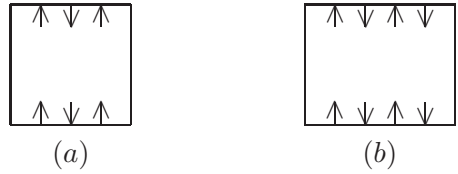

\centerline{\xy
(-8,8);(8,8) **@{-},
(-8,-8);(8,-8) **@{-},
(8,-8);(8,8) **@{-},
(-8,-8);(-8,8) **@{-},
(-4,5);(-4,8) **@{-}, (-3.95,6.7) *{\wedge},
(-4,-5);(-4,-8) **@{-}, (-3.97,-6) *{\wedge},
(0,5);(0,8) **@{-}, (0.1,6) *{\vee},
(0,-5);(0,-8) **@{-}, (0.1,-7) *{\vee},
(4,5);(4,8) **@{-}, (4.05,6.7) *{\wedge},
(4,-5);(4,-8) **@{-}, (4.04,-6) *{\wedge},
(0,-12) *{(a)},
\endxy
\qquad\qquad\qquad
\xy
(-8,8);(12,8) **@{-},
(-8,-8);(12,-8) **@{-},
(12,-8);(12,8) **@{-},
(-8,-8);(-8,8) **@{-},
(-4,5);(-4,8) **@{-}, (-3.95,6.7) *{\wedge},
(-4,-5);(-4,-8) **@{-}, (-3.97,-6) *{\wedge},
(0,5);(0,8) **@{-}, (0.1,6) *{\vee},
(0,-5);(0,-8) **@{-}, (0.1,-7) *{\vee},
(4,5);(4,8) **@{-}, (4.05,6.7) *{\wedge},
(4,-5);(4,-8) **@{-}, (4.04,-6) *{\wedge},
(8,5);(8,8) **@{-}, (8.1,6) *{\vee},
(8,-5);(8,-8) **@{-}, (8.1,-7) *{\vee},
(2,-12) *{(b)},
\endxy}
\caption{Boundaries of $3, 4$-tangle diagrams}\label{fig-bdary-t}
\end{figure}

\begin{lemma}\label{3-tangle basis}
Let $\mathcal T$ be a $3$-tangle diagram with the boundary (a) in Figure~\ref{fig-bdary-t} such that there are no crossings, $2$-gons and $4$-gons and that 
there are no connected components as diagrams in ${\rm Int}D^2$. Then $\mathcal T$ is one of the six fundamental $3$-tangle diagrams $f_0, f_1, \ldots, f_5$ shown in Figure~\ref{fig-3tbas}.

\begin{figure}[h]
\begin{center}
\resizebox{0.70\textwidth}{!}{%
  \includegraphics{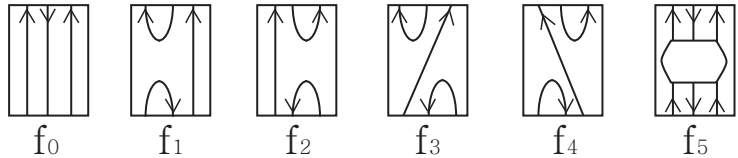} }
\caption{Fundamental $3$-tangle diagrams}\label{fig-3tbas}
\end{center}
\end{figure}
\end{lemma}

\begin{lemma}\label{4-tangle basis}
Let $\mathcal T$ be a $4$-tangle diagram with the boundary (b) in Figure~\ref{fig-bdary-t} such that there are no crossings, $2$-gons and $4$-gons and that there are no connected components as diagrams in ${\rm Int}D^2$. Then $\mathcal T$ is one of the $23$ fundamental $4$-tangle diagrams $g_0, g_1, \ldots, g_{22}$ shown in Figure~\ref{fig-4tbas}.
\begin{figure}[h]
\begin{center}
\resizebox{0.70\textwidth}{!}{%
  \includegraphics{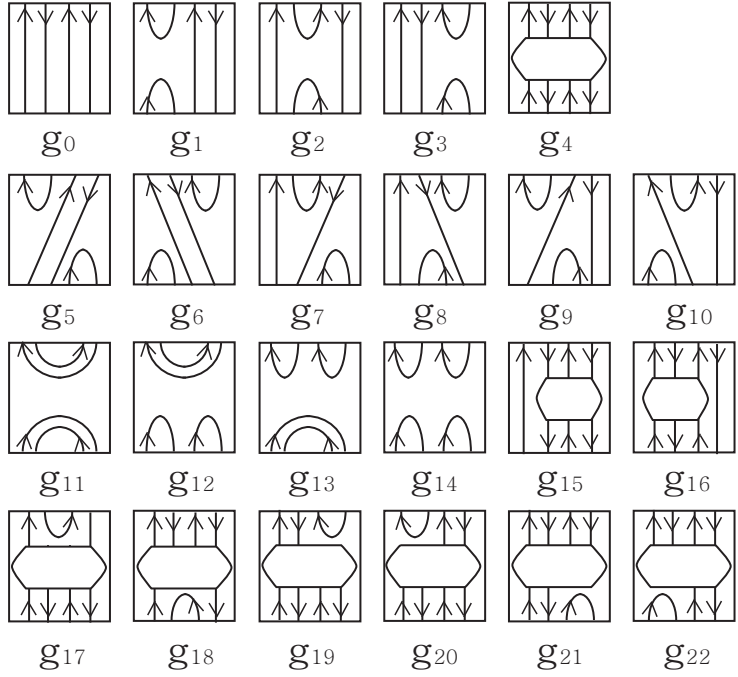} }
\caption{Fundamental $4$-tangle diagrams} \label{fig-4tbas}
\end{center}
\end{figure}
\end{lemma}

Lemmas \ref{3-tangle basis} and \ref{4-tangle basis} are proved in the end of this paper.

The following proposition gives the behavior of the polynomial $\ll\cdot\gg$ under a Yoshikawa move $\Gamma_7$.

\begin{proposition}\label{lem-m07}
Let $D$ and $D'$ be oriented marked graph diagrams such that $D'$ is obtained from $D$ by a Yoshikawa move $\Gamma_7$ as depicted in Figure~\ref{fig-m07}. Then 
$$\ll D\gg - \ll D'\gg \, = \Delta(a)xy\psi(a,x,y),$$
where $\psi(a,x,y)$ is a polynomial in $\mathbb Z[a^{-1},a,x,y]$ and 
\begin{equation}\label{gen-poly-7-8}
\Delta(a) ~ := ~ (a^{-6}+1+a^6)^2-1 
~ = ~ a^{-12}(a^{12}+1)(a^{6}+1)^2. 
\end{equation}
\end{proposition}

\begin{figure}[ht]
\begin{center}
\resizebox{0.50\textwidth}{!}{%
  \includegraphics{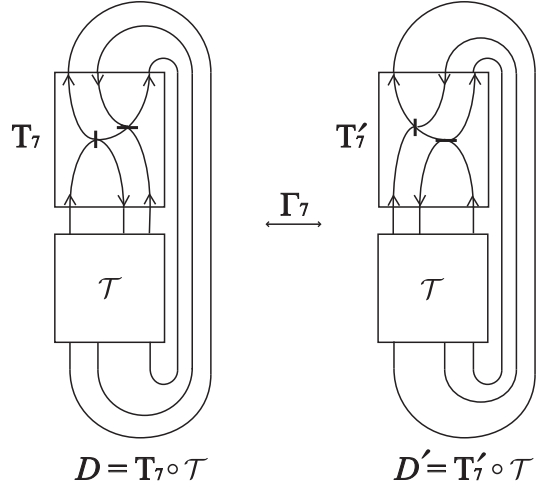}}
\caption{Yoshikawa move $\Gamma_7$} \label{fig-m07}
\end{center}
\end{figure}

\begin{proof}
Applying the axioms ({\bf L1}) and ({\bf L2}) in Definition \ref{defn-L-poly-osl} and ({\bf K1})--({\bf K5}) in Definition \ref{thm-A2inv} to the 3-tangle diagram $\mathcal T$ in $D=T_7\circ \mathcal T$, we can express $[[D]]$ as a linear combination of polynomials $[[T_7\circ U_k]] (1 \leq k \leq m)$ for some integer $m \geq 1$, where each $U_k$ is a $3$-tangle diagram 
satisfying the assumption on $\mathcal T$ in Lemma~\ref{3-tangle basis}. 
By the lemma,  we see that $U_k$ is one of the fundamental $3$-tangle diagrams $f_0, f_1, \dots, f_5$ in Figure~\ref{fig-3tbas}. Hence we have
$$[[D]]=[[T_7\circ\mathcal T]]=\sum_{i=0}^5\psi_i(a,x,y)[[T_7 \circ f_i]],
$$
where $\psi_i(a,x,y)$ is a polynomial in $\mathbb Z[a,a^{-1},x,y].$ Similarly, we have
$$[[D']] =[[T'_7\circ\mathcal T]]=\sum_{i=0}^5\psi_i(a,x,y)[[T'_7 \circ f_i]].$$
This gives
\begin{equation}\label{eq1-pf-m7}
[[D]]-[[D']]=\sum_{i=0}^5\psi_i(a,x,y)\bigg([[T_7 \circ f_i]]-[[T'_7 \circ f_i]]\bigg).
\end{equation}
By a straightforward computation, we obtain 
\begin{align*}
[[T_7 \circ f_i]] &=[[f_2\circ f_i]]x^2+[[f_0\circ f_i]]xy+[[f_4\circ f_i]]yx+[[f_1\circ f_i]]y^2,\\
[[T_7' \circ f_i]] &=[[f_2\circ f_i]]x^2+[[f_0\circ f_i]]xy+[[f_3\circ f_i]]yx+[[f_1\circ f_i]]y^2.
\end{align*}
Let $A=a^{-6}+1+a^6$ and $B=a^{-3}+a^{3}$. Then it is easily checked that 
$ [[f_3\circ f_i]] $ and $ [[f_4\circ f_i]] $ with $0 \leq i \leq 5$ are as in 
Table~\ref{table:1}.  
Hence 
\begin{align*}
[[T_7 \circ f_i]] - [[T_7' \circ f_i]] 
&=xy\bigg([[f_4\circ f_i]]-[[f_3\circ f_i]]\bigg)\\
&=\left\{
    \begin{array}{ll}
      0, & \hbox{$i=0, 1, 2, 5;$} \\
  xy(A^2 -1) = xy \Delta(a),  & \hbox{$i=3$;} \\
  - xy(A^2 -1) = - xy \Delta(a) , & \hbox{$i=4$.} 
    \end{array}
  \right.
\end{align*}
Therefore it follows from (\ref{eq1-pf-m7}) that 
$ [[D]]-[[D']]= \Delta(a)xy\psi'(a,x,y), $
where $\psi'(a,x,y)=\psi_3(a,x,y)- \psi_4(a,x,y).$ 
Finally, since $w(D)=w(D')$, we obtain 
$ \ll D\gg-\ll D'\gg \, = \Delta(a)xy\psi(a,x,y), $
where $\psi(a,x,y)=a^{8w(D)}\psi'(a,x,y)$. 
\end{proof}

\begin{table}[h]
$$
\begin{array}{|c|llllll|}
\hline
\circ & f_0 & f_1 & f_2 & f_3 & f_4 & f_5\\
\hline
f_3 & 1 & A & A & 1     & A^2 & B^3\\
f_4 & 1 & A & A & A^2 & 1     & B^3\\
\hline
\end{array}
$$
\caption{$ [[f_3\circ f_i]] $ and $ [[f_4\circ f_i]] $}
\label{table:1}
\end{table}

Now we investigate the behavior of $\ll\cdot\gg$ under a Yoshikawa move $\Gamma_8$.

\begin{proposition}\label{lem-m08}
Let $D$ and $D'$ be oriented marked graph diagrams such that $D'$ is obtained from $D$ by a Yoshikawa move $\Gamma_8$ as depicted in Figure~\ref{fig-m08}. Then 
$$\ll D\gg - \ll D'\gg \, =   
(a^{-3}-a^{3})\Delta(a)xy\varphi(a,x,y),$$
where  $\varphi(a,x,y)$ is a polynomial in $\mathbb Z[a^{-1},a,x,y]$ and $\Delta(a)$ is the polynomial in (\ref{gen-poly-7-8}).
\begin{figure}[ht]
\begin{center}
\resizebox{0.50\textwidth}{!}{%
  \includegraphics{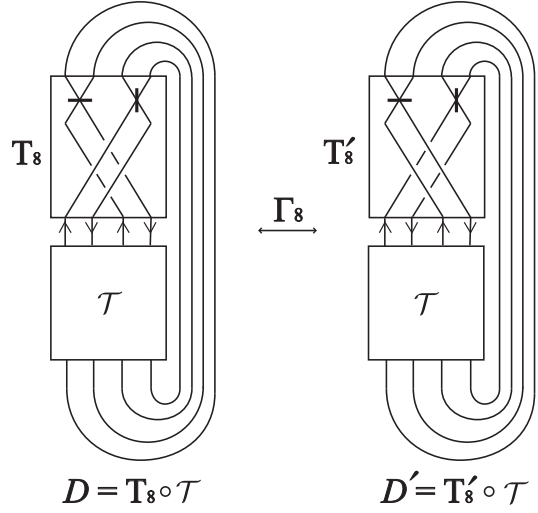}}
\caption{Yoshikawa move $\Gamma_8$}\label{fig-m08}
\end{center}
\end{figure}
\end{proposition}

\begin{proof}
Applying the axioms (${\bf L1}$), (${\bf L2}$)  and (${\bf K1}$)--(${\bf K5}$) to the 4-tangle diagram $\mathcal T$ in $D=T_8\circ \mathcal T$, we can express $[[D]]$ as a linear combination of $[[T_8\circ V_k]] (1\leq k\leq m)$ for some integer $m\geq 1$, where each $V_k$ is a $4$-tangle diagram 
satisfying the assumption on $\mathcal T$ in Lemma~\ref{4-tangle basis}. 
By the lemma, we see that 
 $V_k$ is one of the fundamental $4$-tangle diagrams $g_0, g_1, \dots, g_{22}$ in Figure~\ref{fig-4tbas}. Hence we have
$$[[D]] =[[T_8\circ\mathcal T]] =\sum_{i=0}^{22}\varphi_i(a,x,y)[[T_8 \circ g_i]],
$$
where $\varphi_i(a,x,y) \in \mathbb Z[a^{-1},a,x,y].$ Similarly, we obtain
$$[[D']] =[[T'_8\circ\mathcal T]]=\sum_{i=0}^{22}\varphi_i(a,x,y)[[T'_8 \circ g_i]].$$
This gives
\begin{equation}\label{eq1-pf-m8}
[[D]] -[[D']] =\sum_{i=0}^{22}\varphi_i(a,x,y)
\bigg([[T_8 \circ g_i]] -[[T'_8 \circ g_i]]\bigg).
\end{equation}
By a straightforward computation, we obtain 
\begin{align*}
[[T_8 \circ g_i]] &=[[g_5\circ g_i]] x^2+[[g_{14}\circ g_i]] xy+[[g\circ g_i]] yx+[[g_6\circ g_i]] y^2,\\
[[T_8' \circ g_i]] &=[[g_5\circ g_i]] x^2+[[g_{14}\circ g_i]] xy+[[g^*\circ g_i]] yx+[[g_6\circ g_i]] y^2,
\end{align*}
where $g$ and $g^*$ are $4$-tangle diagrams shown in Figure~\ref{fig-sharp-m08}.
\begin{figure}[ht]
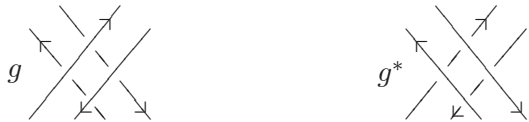

\centerline{\xy 
(11,15);(14.2,11) **@{-}, 
(15.8,9);(17.4,7) **@{-},
(19,5);(23,0) **@{-}, 
(7,12);(11.2,7) **@{-},
(12.7,5.2);(14.4,3.2) **@{-}, 
(15.7,1.6);(17,0) **@{-},
(13,0);(23,12) **@{-}, 
(7,0);(19,15) **@{-},
 (9.1,9) *{\ulcorner}, 
 (14.5,1.6) *{\llcorner}, 
 (17,12) *{\urcorner}, (21.7,1.6) *{\lrcorner},(5,6)*{g},
(69,15);(65.8,11) **@{-}, 
(64.2,9);(62.6,7) **@{-}, 
(61,5);(57,0) **@{-}, 
(73,12);(68.8,7) **@{-},
(67.3,5.2);(65.6,3.2) **@{-}, 
(64.3,1.6);(63,0) **@{-},
(67,0);(57,12) **@{-}, 
(73,0);(61,15) **@{-},
 (59.1,9) *{\ulcorner}, 
 (64,1.2) *{\llcorner}, 
 (67,12) *{\urcorner}, (71.7,1.6) *{\lrcorner}, (55,6)*{g^*},
\endxy}
\caption{$4$-tangle diagrams $g$ and $g^*$}\label{fig-sharp-m08}
\end{figure}

Using (${\bf L1}$) and (${\bf K2}$)--(${\bf K5}$), we obtain
\begin{align*}
[[g\circ g_i]] 
&= a^{-6}[[g_0\circ g_i]] + [[g_2\circ g_i]] + [[g_4\circ g_i]] + [[g_7\circ g_i]] + [[g_8\circ g_i]]\\
& + [[g_9\circ g_i]] + [[g_{10}\circ g_i]] + a^6 [[g_{11}\circ g_i]] - a^{-3}[[g_{15}\circ g_i]]\\  
&-a^{-3}[[g_{16}\circ g_i]] -a^3[[g_{17}\circ g_i]] -a^3[[g_{18}\circ g_i]].
\end{align*}
Since $g^*$ is the mirror image of $g$, it is seen from (${\bf K4}$) and (${\bf K5}$) that 
\begin{align*}
[[g^*\circ g_i]] &= a^{6}[[g_0\circ g_i]] + [[g_2\circ g_i]] + [[g_4\circ g_i]] + [[g_7\circ g_i]] + [[g_8\circ g_i]]\\
& + [[g_9\circ g_i]] + [[g_{10}\circ g_i]] + a^{-6} [[g_{11}\circ g_i]] - a^{3}[[g_{15}\circ g_i]]\\  
&-a^{3}[[g_{16}\circ g_i]] -a^{-3}[[g_{17}\circ g_i]] -a^{-3}[[g_{18}\circ g_i]].
\end{align*}
This gives that
\begin{align*}
[[g\circ g_i]] -[[g^*\circ g_i]]
& = (a^{-3}-a^{3})\bigg[(a^{-3}+a^{3})\bigg([[g_0\circ g_i]] - [[g_{11}\circ g_i]]\bigg) \\
& -[[g_{15}\circ g_i]]-[[g_{16}\circ g_i]]+[[g_{17}\circ g_i]]+[[g_{18}\circ g_i]]\bigg].
\end{align*}
By simple, but tedious calculations, we obtain Table~\ref{table:2} for $[[g_k\circ g_i]]$ with $k=0,11,15,16,17,18$ and $0\leq i\leq 22$. 


\begin{table}[h] 
\centerline{${\small \begin{array}{|l|llllll|}
\hline
\circ & g_0 & g_{11} & g_{15} & g_{16} & g_{17} & g_{18}\\
\hline
g_0 & A^3 & A & AB^3 & AB^3 & B^3 & B^3\\
g_1 & A^2 & 1 &   B^3 & AB^3 & B^3 & B^3\\
g_2 & A^2 & A^2 & AB^3 & AB^3 & AB^3 & AB^3\\
g_3 & A^2 & 1 & AB^3 & B^3 & B^3 & B^3\\
g_4 & B^4 & B^4 & 2B^3+B^5 & 2B^3+B^5 & 2B^3+B^5 & 2B^3+B^5\\
g_5 & 1 & 1 & B^3 & B^3 & B^3 & B^3\\
g_6 & 1 & 1 & B^3 & B^3 & B^3 & B^3\\
g_7 & A & A & AB^3 & B^3 & B^3 & AB^3\\
g_8 & A & A & AB^3 & B^3 & AB^3 & B^3\\
g_9 & A & A & B^3 & AB^3 & AB^3 & B^3\\
g_{10} & A & A & B^3 & AB^3 & B^3 & AB^3\\
g_{11} & A & A^3 & B^3 & B^3 & AB^3 & AB^3\\
g_{12} & 1 & A^2 & B^3 & B^3 & B^3 & AB^3\\
g_{13} & 1 & A^2 & B^3 & B^3 & AB^3 & B^3\\
g_{14} & A & A & B^3 & B^3 & B^3 & B^3\\
g_{15} & AB^3 & B^3 & 2AB^2+AB^4 & 2B^4 & 2B^4 & 2B^4\\
g_{16} & AB^3 & B^3 & 2B^4 & 2AB^2+AB^4 & 2B^4 & 2B^4\\
g_{17} & B^3 & AB^3 & 2B^4 & 2B^4 & 2B^4 & 2AB^2+AB^4\\
g_{18} & B^3 & AB^3 & 2B^4 & 2B^4 & 2AB^2+AB^4 & 2B^4\\
g_{19} & B^3 & B^3 & 2B^4 & 2B^2+B^4 & 2B^4 & 2B^2+B^4\\
g_{20} & B^3 & B^3 & 2B^2+B^4 & 2B^4 & 2B^4 & 2B^2+B^4\\
g_{21} & B^3 & B^3 & 2B^4 & 2B^2+B^4 & 2B^2+B^4 & 2B^4\\
g_{22} & B^3 & B^3 & 2B^2+B^4 & 2B^4 & 2B^2+B^4 & 2B^4\\
\hline
\end{array}}$}
\caption{$[[g_k\circ g_i]]$}
\label{table:2}
\end{table}

Thus it follows from (\ref{eq1-pf-m8}) and the identity $B^2=A+1$ that
\begin{align*}
&[[T_8 \circ g_i]] - [[T_8' \circ g_i]]
=([[g\circ g_i]] - [[g^*\circ g_i]])xy\\
&=\left\{
    \begin{array}{ll}
      ~~(a^{-6}-a^{6})(A-2)(A-1)(A+1)xy, & \hbox{$i=0;$} \\
      -(a^{-6}-a^{6})(A-2)(A-1)(A+1)xy, & \hbox{$i=11;$} \\
     -(a^{-3}-a^{3})(A-1)(A+1)xy, & \hbox{$i=15, 16;$} \\
     ~~(a^{-3}-a^{3})(A-1)(A+1)xy, & \hbox{$i=17, 18;$} \\
     ~~~0 & \hbox{otherwise.}
    \end{array}
  \right.\\
  &=\left\{
 \begin{array}{ll}
 ~~(a^{-3}+a^{3})(a^{-6}-1+a^{6})(a^{-3}-a^{3})\Delta(a)xy, & \hbox{$i=0;$} \\
   -(a^{-3}+a^{3})(a^{-6}-1+a^{6})(a^{-3}-a^{3})\Delta(a)xy, & \hbox{$i=11;$} \\
     -(a^{-3}-a^{3})\Delta(a)xy, & \hbox{$i=15, 16;$} \\
     ~~(a^{-3}-a^{3})\Delta(a)xy, & \hbox{$i=17, 18;$} \\
     ~~~0 & \hbox{otherwise.}
    \end{array}
  \right.
\end{align*}
Therefore it follows from (\ref{eq1-pf-m8}) that 
\begin{equation*}
[[D]]-[[D']]= (a^{-3}-a^{3})\Delta(a)xy\varphi'(a,x,y),
  \end{equation*}
for a polynomial $\varphi'(a,x,y)$ in $\mathbb Z[a^{-1},a,x,y]$. 
Since $w(D)=w(D')$, we obtain \begin{equation*}
\ll D\gg-\ll D'\gg \, = (a^{-3}-a^{3})\Delta(a)xy\varphi(a,x,y),
  \end{equation*}
where $\varphi(a,x,y)=a^{8w(D)}\varphi'(a,x,y)$. 
\end{proof}

\begin{theorem}\label{thm-inv-osl-original}
Let $\mathcal L$ be an oriented surface-link and let $D$ be an oriented marked graph diagram presenting $\mathcal L$. Then the polynomial $\ll D\gg \in \mathbb Z[a^{-1}, a, x,y]$, modulo the ideal generated by $\Delta(a)$, is an invariant of $\mathcal L$ up to multiplication by powers of $(a^{-6}+1+a^6)x + y$ and $x + (a^{-6}+1+a^6)y$.  
\end{theorem}

\begin{proof} 
It follows from Theorem~\ref{thm-skein-rel} and 
Propositions~\ref{lem-m06}, \ref{lem-m07} and \ref{lem-m08}.  
\end{proof}


\section{Specializations of the polynomial $\ll \cdot \gg$}\label{sect-special}

Here we consider some specializations of the polynomial $\ll \cdot \gg$. 

Let $m$ be a non-negative integer and let $I(a^m+1)$ be the ideal $(a^m+1)\mathbb Z[a^{-1},a,x,y]$ of $\mathbb Z[a^{-1},a,x,y]$ 
generated by $a^m+1$. We abbreviate $f + I(a^m+1)$ as $f$ for $f \in \mathbb Z[a^{-1},a,x,y]$ unless it makes confusion. 

For an oriented marked graph diagram $D$, we denote by 
$\ll D \gg_{a^m+1}$ the polynomial $\ll D \gg$ modulo the ideal $I(a^m+1)$, i.e.,  
$$
\ll D \gg_{a^m+1}  ~=~ \ll D \gg + I(a^m+1) =~ \ll D \gg 
~ \in \mathbb Z[a^{-1},a,x,y]/ I(a^m+1). 
$$

It follows from  Theorem~\ref{thm-skein-rel} that $\ll \cdot \gg_{a^m+1}$ is an invariant for oriented marked graphs satisfying the same conditions with (1)--(4) in 
Theorem~\ref{thm-skein-rel}.  In particular,  
for any positive integer $\mu$ and for any skein triple $(D_+, D_-, D_0)$ of  link diagrams, 
\begin{equation}\label{eqn:m=m(1)}
\ll O^\mu \gg_{a^m+1} =  (a^{-6}+1+a^6)^{\mu-1} 
\end{equation} 
and 
\begin{equation}\label{eqn:m=m(2)}
a^{-9}  \ll D_+ \gg_{a^m+1} - a^9 \ll D_- \gg_{a^m+1} =  (a^{-3}-a^3)  \ll D_0 \gg_{a^m+1}. 
\end{equation}

First we are interested in a case that $m=6$ or $m=12$, since if $a^6+1=0$ or $a^{12}+1=0$ then $\Delta(a)=0$, where $\Delta(a)$ is the polynomial in (\ref{gen-poly-7-8}).  

\vspace{0.5cm}

Suppose $m=6$.  Then, (\ref{eqn:m=m(1)}) implies that 
for any positive integer $\mu$, 
\begin{equation}\label{eqn:m=6(1)}
\ll O^\mu \gg_{a^6+1} = (-1)^{\mu-1}. 
\end{equation} 
Since $a^3$ is a unit in ${\mathbb Z}[a^{-1},a,x,y]/ I(a^6+1)$,   (\ref{eqn:m=m(2)}) implies that 
for any skein triple $(D_+, D_-, D_0)$ of  link diagrams, 
\begin{equation}\label{eqn:m=6(2)}
\ll D_+ \gg_{a^6+1} + \ll D_- \gg_{a^6+1} = -2 \ll D_0 \gg_{a^6+1}. 
\end{equation}

\begin{lemma}\label{lem:m=6}
Let $D$ be an oriented  link diagram and let $\# D$ denote the number of components of the link presented by $D$. 
Then $\ll D \gg_{a^6+1}$ is $1$ if $\# D$ is odd, or $-1$ if $\# D$ is even, i.e., 
$\ll D \gg_{a^6+1} = (-1)^{\# D -1}$. 
\end{lemma} 

\begin{proof} 
When we restrict  $\ll \cdot \gg_{a^6+1}$ to the family of oriented link diagrams, it is an oriented link invariant satisfying (\ref{eqn:m=6(1)}) and (\ref{eqn:m=6(2)}).  On the other hand, $(-1)^{\# D -1}$ is also such an invariant.  We see that 
 $\ll D \gg_{a^6+1} = (-1)^{\# D -1}$ for any link diagram $D$ by  considering a skein tree. 
\end{proof}

\begin{example}\label{example:m=6:8_1-1}
Consider the diagram $8_1$ of a 
spun trefoil in Yoshikawa's table \cite{Yo}. From Lemma~\ref{lem:m=6}, it follows that
\begin{align*}
\ll ~&\xy 
(-12,8);(-2,8) **@{-}, 
(2,8);(12,8) **@{-}, 
(-12,8);(-12,2) **@{-}, 
(12,8);(12,2) **@{-}, 
(-12,-8);(-2,-8) **@{-}, 
(-2,-8);(12,-8) **@{-}, 
(-12,-8);(-12,-2) **@{-}, 
(12,-8);(12,-2) **@{-},
(-2,8);(2,4) **@{-}, 
(-2,4);(2,8) **@{-}, 
(-2,4);(4,-2) **@{-}, 
(2,4);(-4,-2) **@{-}, 
(2.4,0.4);(4,2) **@{-},
(0.1,-1.9);(1.7,-0.2) **@{-},  
(0.1,-1.9);(-1.7,-0.3) **@{-},  
(-2.4,0.4);(-4,2) **@{-},
(-12,2);(-10.4,0.4) **@{-}, 
(-9.6,-0.4);(-8,-2) **@{-}, 
(-12,-2);(-8,2) **@{-},
(-8,2);(-6.4,0.4) **@{-}, 
(-5.6,-0.4);(-4,-2) **@{-}, 
(-8,-2);(-4,2) **@{-},
(12,2);(10.4,0.4) **@{-}, 
(9.6,-0.4);(8,-2) **@{-},
(12,-2);(8,2) **@{-},
(5.6,-0.4);(4,-2) **@{-}, 
(8,2);(6.6,0.4) **@{-},
(8,-2);(4,2) **@{-},
(-0.1,4.5);(-0.1,7.5) **@{-},
(0,4.5);(0,7.5) **@{-},
(0.1,4.5);(0.1,7.5) **@{-},
(-1.5,1.9);(1.5,1.9) **@{-},
(-1.5,2);(1.5,2) **@{-},
(-1.5,2.1);(1.5,2.1) **@{-}, 
(-9,0.4)*{\urcorner}, 
(-9,-0.8)*{\lrcorner}, 
(11,0.5)*{\urcorner}, 
(11,-0.8)*{\lrcorner}, 
(0,12) *{8_1},
\endxy~ \gg_{a^6+1} \\ 
& = x^2
\ll  ~\xy 
(-12,8);(-2,8) **@{-}, 
(2,8);(12,8) **@{-}, 
(-12,8);(-12,2) **@{-}, 
(12,8);(12,2) **@{-}, 
(-12,-8);(-2,-8) **@{-}, 
(-2,-8);(12,-8) **@{-}, 
(-12,-8);(-12,-2) **@{-}, 
(12,-8);(12,-2) **@{-},
(2,0);(4,-2) **@{-}, 
(-2,0);(-4,-2) **@{-}, 
(2.4,0.4);(4,2) **@{-},
(0.1,-1.9);(1.7,-0.2) **@{-},  
(0.1,-1.9);(-1.7,-0.3) **@{-},  
(-2.4,0.4);(-4,2) **@{-},
(-12,2);(-10.4,0.4) **@{-}, 
(-9.6,-0.4);(-8,-2) **@{-}, 
(-12,-2);(-8,2) **@{-},
(-8,2);(-6.4,0.4) **@{-}, 
(-5.6,-0.4);(-4,-2) **@{-}, 
(-8,-2);(-4,2) **@{-},
(12,2);(10.4,0.4) **@{-}, 
(9.6,-0.4);(8,-2) **@{-},
(12,-2);(8,2) **@{-},
(5.6,-0.4);(4,-2) **@{-}, 
(8,2);(6.6,0.4) **@{-},
(8,-2);(4,2) **@{-},
(-9,0.4)*{\urcorner}, 
(-9,-0.8)*{\lrcorner}, 
(11,0.5)*{\urcorner}, 
(11,-0.8)*{\lrcorner}, 
(-2,4);(-2,8) **\crv{(1,6)}, 
(2,8);(2,4) **\crv{(-1,6)},
(-2,4);(2,4) **\crv{(0,1)}, 
(-2,0);(2,0) **\crv{(0,3)},
 \endxy~ \gg_{a^6+1} 
 +xy
\ll ~\xy 
(-12,8);(-2,8) **@{-}, 
(2,8);(12,8) **@{-}, 
(-12,8);(-12,2) **@{-}, 
(12,8);(12,2) **@{-}, 
(-12,-8);(-2,-8) **@{-}, 
(-2,-8);(12,-8) **@{-}, 
(-12,-8);(-12,-2) **@{-}, 
(12,-8);(12,-2) **@{-},
(2,0);(4,-2) **@{-}, 
(-2,0);(-4,-2) **@{-},  
(2.4,0.4);(4,2) **@{-},
(0.1,-1.9);(1.7,-0.2) **@{-},  
(0.1,-1.9);(-1.7,-0.3) **@{-},  
(-2.4,0.4);(-4,2) **@{-},
(-12,2);(-10.4,0.4) **@{-}, 
(-9.6,-0.4);(-8,-2) **@{-}, 
(-12,-2);(-8,2) **@{-},
(-8,2);(-6.4,0.4) **@{-}, 
(-5.6,-0.4);(-4,-2) **@{-}, 
(-8,-2);(-4,2) **@{-},
(12,2);(10.4,0.4) **@{-}, 
(9.6,-0.4);(8,-2) **@{-},
(12,-2);(8,2) **@{-},
(5.6,-0.4);(4,-2) **@{-}, 
(8,2);(6.6,0.4) **@{-},
(8,-2);(4,2) **@{-},
(-9,0.4)*{\urcorner}, 
(-9,-0.8)*{\lrcorner}, 
(11,0.5)*{\urcorner}, 
(11,-0.8)*{\lrcorner}, 
(-2,4);(-2,8) **\crv{(1,6)}, 
(2,8);(2,4) **\crv{(-1,6)},
(-2,0);(-2,4) **\crv{(1,2)}, 
(2,4);(2,0) **\crv{(-1,2)},
 \endxy~ \gg_{a^6+1} \\ 
 &\hskip 1cm +yx
\ll  ~\xy 
(-12,8);(-2,8) **@{-}, 
(2,8);(12,8) **@{-}, 
(-12,8);(-12,2) **@{-}, 
(12,8);(12,2) **@{-}, 
(-12,-8);(-2,-8) **@{-}, 
(-2,-8);(12,-8) **@{-}, 
(-12,-8);(-12,-2) **@{-}, 
(12,-8);(12,-2) **@{-},
(2,0);(4,-2) **@{-}, 
(-2,0);(-4,-2) **@{-}, 
(2.4,0.4);(4,2) **@{-},
(0.1,-1.9);(1.7,-0.2) **@{-},  
(0.1,-1.9);(-1.7,-0.3) **@{-},  
(-2.4,0.4);(-4,2) **@{-},
(-12,2);(-10.4,0.4) **@{-}, 
(-9.6,-0.4);(-8,-2) **@{-}, 
(-12,-2);(-8,2) **@{-},
(-8,2);(-6.4,0.4) **@{-}, 
(-5.6,-0.4);(-4,-2) **@{-}, 
(-8,-2);(-4,2) **@{-},
(12,2);(10.4,0.4) **@{-}, 
(9.6,-0.4);(8,-2) **@{-},
(12,-2);(8,2) **@{-},
(5.6,-0.4);(4,-2) **@{-}, 
(8,2);(6.6,0.4) **@{-},
(8,-2);(4,2) **@{-},
(-9,0.4)*{\urcorner}, 
(-9,-0.8)*{\lrcorner}, 
(11,0.5)*{\urcorner}, 
(11,-0.8)*{\lrcorner},  
(-2,8);(2,8) **\crv{(0,5)}, 
(-2,4);(2,4) **\crv{(0,7)},
(-2,4);(2,4) **\crv{(0,1)}, 
(-2,0);(2,0) **\crv{(0,3)},
 \endxy~ \gg_{a^6+1} + y^2
\ll  ~\xy 
(-12,8);(-2,8) **@{-}, 
(2,8);(12,8) **@{-}, 
(-12,8);(-12,2) **@{-}, 
(12,8);(12,2) **@{-}, 
(-12,-8);(-2,-8) **@{-}, 
(-2,-8);(12,-8) **@{-}, 
(-12,-8);(-12,-2) **@{-}, 
(12,-8);(12,-2) **@{-},
(2,0);(4,-2) **@{-}, 
(-2,0);(-4,-2) **@{-},  
(2.4,0.4);(4,2) **@{-},
(0.1,-1.9);(1.7,-0.2) **@{-},  
(0.1,-1.9);(-1.7,-0.3) **@{-},  
(-2.4,0.4);(-4,2) **@{-},
(-12,2);(-10.4,0.4) **@{-}, 
(-9.6,-0.4);(-8,-2) **@{-}, 
(-12,-2);(-8,2) **@{-},
(-8,2);(-6.4,0.4) **@{-}, 
(-5.6,-0.4);(-4,-2) **@{-}, 
(-8,-2);(-4,2) **@{-},
(12,2);(10.4,0.4) **@{-}, 
(9.6,-0.4);(8,-2) **@{-},
(12,-2);(8,2) **@{-},
(5.6,-0.4);(4,-2) **@{-}, 
(8,2);(6.6,0.4) **@{-},
(8,-2);(4,2) **@{-},
(-9,0.4)*{\urcorner}, 
(-9,-0.8)*{\lrcorner}, 
(11,0.5)*{\urcorner}, 
(11,-0.8)*{\lrcorner}, 
(-2,0);(-2,4) **\crv{(1,2)}, 
(2,4);(2,0) **\crv{(-1,2)},
(-2,8);(2,8) **\crv{(0,5)}, 
(-2,4);(2,4) **\crv{(0,7)},
 \endxy~ \gg_{a^6+1} \\
 & = x^2 \ll O^2 \gg_{a^6+1} 
 + xy \ll ~\xy 
(2,0);(4,-2) **@{-}, 
(-2,0);(-4,-2) **@{-},  
(2.4,0.4);(4,2) **@{-},
(0.1,-1.9);(1.7,-0.2) **@{-},  
(0.1,-1.9);(-1.7,-0.3) **@{-},  
(-2.4,0.4);(-4,2) **@{-},
(-12,2);(-10.4,0.4) **@{-}, 
(-9.6,-0.4);(-8,-2) **@{-}, 
(-12,-2);(-8,2) **@{-},
(-8,2);(-6.4,0.4) **@{-}, 
(-5.6,-0.4);(-4,-2) **@{-}, 
(-8,-2);(-4,2) **@{-},
(12,2);(10.4,0.4) **@{-}, 
(9.6,-0.4);(8,-2) **@{-},
(12,-2);(8,2) **@{-},
(5.6,-0.4);(4,-2) **@{-}, 
(8,2);(6.6,0.4) **@{-},
(8,-2);(4,2) **@{-},
(-9,0.4)*{\urcorner}, 
(-9,-0.8)*{\lrcorner}, 
(11,0.5)*{\urcorner}, 
(11,-0.8)*{\lrcorner}, 
(-12,2);(-2,4) **\crv{(-14,5)}, 
(2,4);(12,2) **\crv{(14,5)},
(-2,0);(-2,4) **\crv{(1,2)}, 
(2,4);(2,0) **\crv{(-1,2)},
(-12,-2);(0,-5) **\crv{(-14,-5)},
(0,-5);(12,-2) **\crv{(14,-5)},
 \endxy~ \gg_{a^6+1} \\
&\hskip 1cm +yx \ll O^3  \gg_{a^6+1} 
+y^2 \ll O^2 \gg_{a^6+1} \\
&= -x^2+xy+yx-y^2=-(x-y)^2. 
\end{align*}
\end{example}

\begin{theorem}\label{prop:m=6}
Let $D$ be an oriented marked graph diagram with $h$ marked vertices, and let $L_+(D)$ be the positive resolution of $D$.  Then 
\begin{equation*}
\ll D \gg_{a^6+1}  = \epsilon (x-y)^h,
\end{equation*}
where $\epsilon = (-1)^{\# L_+(D)-1}$. 
\end{theorem}

\begin{proof} 
By (\ref{state formulaB}), 
\begin{equation}\label{eqn:state:m=6}
\ll D \gg_{a^6+1} = \sum_{\sigma \in \mathcal S(D)}
x^{\sigma(\infty)}y^{\sigma(0)}
\ll  D_\sigma \gg_{a^6+1}.   
\end{equation}
Let $\sigma^\ast$ be the state assigning 
$T_\infty$ to every marked vertex of $D$.  Then 
$D_{\sigma^\ast} = L_+(D)$.  By Lemma~\ref{lem:m=6}, $\ll  D_{\sigma^\ast} \gg_{a^6+1} =  (-1)^{\# L_+(D)-1}= \epsilon$.  
Since $\sigma^\ast(\infty)=h$ and  $\sigma^\ast(0)=0$, the state $\sigma^\ast$ contributes 
$\epsilon x^h$ in the right hand side of (\ref{eqn:state:m=6}).  
Let $\sigma$ be a state of $D$, which is obtained from $\sigma^\ast$ by switching  $T_\infty$ and $T_0$ on $k$ marked vertices.  Then $\# D_\sigma - \# L_+(D) \equiv  
k \mod 2$, and by Lemma~\ref{lem:m=6}, $\ll  D_{\sigma} \gg_{a^6+1} =  (-1)^{\# D_{\sigma} -1} =  (-1)^{\# L_+(D) -1 +k}= \epsilon (-1)^k$.  
 Since $\sigma(\infty)=h-k$ and  $\sigma(0)=k$, 
the contribution $x^{\sigma(\infty)}y^{\sigma(0)} \ll  D_\sigma \gg_{a^6+1}$ 
of $\sigma$ is 
$ \epsilon (-1)^k x^{h-k}y^k$.  Since every state $\sigma$ is obtained from $\sigma^\ast$ by choosing each subset of the marked vertices of $D$ and 
switching  $T_\infty$ and $T_0$ there, we see that  $\ll D \gg_{a^6+1} = \epsilon (x-y)^h$.  
\end{proof}

\begin{remark}
From Theorem~\ref{prop:m=6}, we see that all information the invariant $\ll D \gg_{a^6+1}$ has is the number of marked vertices of $D$ and the parity of $\# L_+(D)$. 
By this reason or by Propositions~\ref{lem-m07} and \ref{lem-m08} with 
$\Delta(a)=0$, we see that  $\ll D \gg_{a^6+1}$ is invariant under Yoshikawa moves $\Gamma_7$ and $\Gamma_8$.  
In order to make it invariant under Yoshikawa move $\Gamma_6$, we may consider it up to multiplication by powers of $-x+y$ and $x-y$.  However, this makes $\ll D \gg_{a^6+1}$ the same value for all $D$.  
\end{remark}

Suppose $m=12$.   Then 
for any positive integer $\mu$, 
\begin{equation}\label{eqn:m=12(1)}
\ll O^\mu \gg_{a^{12}+1} =1  
\end{equation} 
and for any skein triple $(D_+, D_-, D_0)$ of  link diagrams, 
\begin{equation}\label{eqn:m=12(2)}
-a^3 \ll D_+ \gg_{a^{12}+1} +  a^{-3} \ll D_- \gg_{a^{12}+1} = (a^{-3} - a^3) \ll D_0 \gg_{a^{12}+1}. 
\end{equation}

\begin{lemma}\label{lem:m=12}
For any  link diagram $D$, 
$\ll D \gg_{a^{12}+1} = 1$. 
\end{lemma} 

\begin{proof} 
When we restrict  $\ll \cdot \gg_{a^{12}+1}$ to the family of oriented link diagrams, it is an oriented link invariant satisfying (\ref{eqn:m=12(1)}) and (\ref{eqn:m=12(2)}).  On the other hand, the constant function $1$ is also such an invariant.  Since the coefficients $-a^3$ and $a^{-3}$ in the left hand side of (\ref{eqn:m=12(2)}) are units in 
${\mathbb Z}[a^{-1}, a, x, y]/ I(a^{12}+1)$, 
we have $\ll D \gg_{a^{12}+1} = 1$ for any link diagram $D$. 
\end{proof}

\begin{theorem}\label{prop:m=12}
Let $D$ be an oriented marked graph diagram with $h$ marked vertices.  Then 
\begin{equation*}
\ll D \gg_{a^{12}+1}  =  (x+y)^h.
\end{equation*}
\end{theorem}

\begin{proof} 
By (\ref{state formulaB}), 
\begin{equation*}
\ll D \gg_{a^{12}+1} = \sum_{\sigma \in \mathcal S(D)}
x^{\sigma(\infty)}y^{\sigma(0)}
\ll  D_\sigma \gg_{a^{12}+1}.   
\end{equation*} 
By Lemma~\ref{lem:m=12}, for any $\sigma$, 
$\ll  D_{\sigma} \gg_{a^{12}+1} = 1$.  
Since every state $\sigma$ is obtained by choosing each subset of the marked vertices of $D$ for assignment of 
 $T_\infty$, we see that  $\ll D \gg_{a^{12}+1} =  (x+y)^h$.  
\end{proof}

\begin{remark}
By Theorem~\ref{prop:m=12}, the invariant $\ll D \gg_{a^{12}+1}$ is determined by the number of marked vertices.  It is invariant under Yoshikawa moves $\Gamma_7$ and $\Gamma_8$. 
 In order to make it invariant under Yoshikawa move $\Gamma_6$, we may consider it up to multiplication by powers of $x+y$.  However, this makes $\ll D \gg_{a^{12}+1}$ the same value for all $D$.  
 \end{remark}

Now we consider another case that $m=9$.  

Suppose $m=9$.  Then 
for any positive integer $\mu$, 
\begin{equation*}
\ll O^\mu \gg_{a^9+1} =  (a^{-6}+1+a^6)^{\mu-1} = 
\left\{
\begin{array}{ll}
1  &    (\mu=1)  \\
3^{\mu-2}(a^{-6}+1+a^6)  &     (\mu \geq 2)
\end{array}
\right.
\end{equation*}
and for any skein triple $(D_+, D_-, D_0)$ of  link diagrams, 
\begin{equation*}
\ll D_+ \gg_{a^9+1} - \ll D_- \gg_{a^9+1} = (a^3 - a^{-3}) \ll D_0 \gg_{a^9+1}. 
\end{equation*}

Instead of studying $\ll D_+ \gg_{a^9+1}$, we here discuss a weaker version as follows.  

Consider the ideal $I(a^9+1, a^{-6} + 1 + a^6)$ of  $\mathbb Z[a^{-1},a,x,y]$ 
generated by $a^9+1$ and $a^{-6} + 1 + a^6$.  We denote by 
$\ll D_+ \gg_{a^9+1}^\ast$ the polynomial $\ll D \gg$ modulo the ideal $I(a^9+1, a^{-6} + 1 + a^6)$, i.e.,  
$$
\ll D \gg_{a^9+1}^\ast   ~=~  \ll D \gg + I(a^9+1, a^{-6} + 1 + a^6) ~   \in \mathbb Z[a^{-1},a,x,y]/ I(a^9+1, a^{-6} + 1 + a^6). 
$$ 

Note that the ideal $I(a^9+1, a^{-6} + 1 + a^6)$ is equal to the ideal $I(a^6-a^3+1)$. 

Then $\ll \cdot \gg_{a^9+1}^\ast$ is an invariant of oriented marked graphs satisfying that 
for any positive integer $\mu$, 
\begin{equation}\label{eqn:m=9*(1)}
\ll O^\mu \gg_{a^9+1}^\ast =   
\left\{
\begin{array}{ll}
1  &    (\mu=1)  \\
0  &     (\mu \geq 2)
\end{array}
\right.
\end{equation}
and  for any skein triple $(D_+, D_-, D_0)$ of  link diagrams, 
\begin{equation}\label{eqn:m=9*(2)}
\ll D_+ \gg_{a^9+1}^\ast - \ll D_- \gg_{a^9+1}^\ast = (a^3 - a^{-3}) \ll D_0 \gg_{a^9+1}^\ast. 
\end{equation}

\begin{lemma}\label{lem:m=9*}
Let $D$ be an oriented  link diagram and let $\nabla[D](z)$ be the Conway polynomial.  Then  
$\ll D \gg_{a^9+1}^\ast = \nabla[D] (a^3 - a^{-3})$.   
\end{lemma} 

\begin{proof} 
When we restrict  $\ll \cdot \gg_{a^9+1}^\ast$ to the family of oriented link diagrams, it is an oriented link invariant satisfying (\ref{eqn:m=9*(1)}) and (\ref{eqn:m=9*(2)}).  On the other hand, $\nabla[\cdot] (a^3 - a^{-3})$ is also such an invariant.  Thus we have $\ll D \gg_{a^9+1}^\ast = \nabla[D] (a^3 - a^{-3})$. 
\end{proof}

\begin{theorem}\label{prop:m=9*}
Let $D$ be an oriented marked graph diagram.  Then 
\begin{equation*}
\ll D \gg_{a^9+1}^\ast = \sum_{\sigma \in \mathcal S(D)}
x^{\sigma(\infty)}y^{\sigma(0)}
\nabla[{D_\sigma}] (a^3 - a^{-3}).   
\end{equation*} 
\end{theorem}

\begin{proof} 
By (\ref{state formulaB}), 
\begin{equation*}
\ll D \gg_{a^9+1}^\ast = \sum_{\sigma \in \mathcal S(D)}
x^{\sigma(\infty)}y^{\sigma(0)}
\ll  D_\sigma \gg_{a^9+1}^\ast.  
\end{equation*}
By Lemma~\ref{lem:m=9*}, we have the result. 
\end{proof}

\begin{example}\label{example:m=9*:8_1-1}
Consider the diagram $8_1$ of a 
spun trefoil in Yoshikawa's table \cite{Yo}. 
\begin{align*}
\ll ~&\xy 
(-12,8);(-2,8) **@{-}, 
(2,8);(12,8) **@{-}, 
(-12,8);(-12,2) **@{-}, 
(12,8);(12,2) **@{-}, 
(-12,-8);(-2,-8) **@{-}, 
(-2,-8);(12,-8) **@{-}, 
(-12,-8);(-12,-2) **@{-}, 
(12,-8);(12,-2) **@{-},
(-2,8);(2,4) **@{-}, 
(-2,4);(2,8) **@{-}, 
(-2,4);(4,-2) **@{-}, 
(2,4);(-4,-2) **@{-}, 
(2.4,0.4);(4,2) **@{-},
(0.1,-1.9);(1.7,-0.2) **@{-},  
(0.1,-1.9);(-1.7,-0.3) **@{-},  
(-2.4,0.4);(-4,2) **@{-},
(-12,2);(-10.4,0.4) **@{-}, 
(-9.6,-0.4);(-8,-2) **@{-}, 
(-12,-2);(-8,2) **@{-},
(-8,2);(-6.4,0.4) **@{-}, 
(-5.6,-0.4);(-4,-2) **@{-}, 
(-8,-2);(-4,2) **@{-},
(12,2);(10.4,0.4) **@{-}, 
(9.6,-0.4);(8,-2) **@{-},
(12,-2);(8,2) **@{-},
(5.6,-0.4);(4,-2) **@{-}, 
(8,2);(6.6,0.4) **@{-},
(8,-2);(4,2) **@{-},
(-0.1,4.5);(-0.1,7.5) **@{-},
(0,4.5);(0,7.5) **@{-},
(0.1,4.5);(0.1,7.5) **@{-},
(-1.5,1.9);(1.5,1.9) **@{-},
(-1.5,2);(1.5,2) **@{-},
(-1.5,2.1);(1.5,2.1) **@{-}, 
(-9,0.4)*{\urcorner}, 
(-9,-0.8)*{\lrcorner}, 
(11,0.5)*{\urcorner}, 
(11,-0.8)*{\lrcorner}, 
(0,12) *{8_1},
\endxy~ \gg_{a^9+1}^\ast  \\ 
& = x^2
\ll  ~\xy 
(-12,8);(-2,8) **@{-}, 
(2,8);(12,8) **@{-}, 
(-12,8);(-12,2) **@{-}, 
(12,8);(12,2) **@{-}, 
(-12,-8);(-2,-8) **@{-}, 
(-2,-8);(12,-8) **@{-}, 
(-12,-8);(-12,-2) **@{-}, 
(12,-8);(12,-2) **@{-},
(2,0);(4,-2) **@{-}, 
(-2,0);(-4,-2) **@{-}, 
(2.4,0.4);(4,2) **@{-},
(0.1,-1.9);(1.7,-0.2) **@{-},  
(0.1,-1.9);(-1.7,-0.3) **@{-},  
(-2.4,0.4);(-4,2) **@{-},
(-12,2);(-10.4,0.4) **@{-}, 
(-9.6,-0.4);(-8,-2) **@{-}, 
(-12,-2);(-8,2) **@{-},
(-8,2);(-6.4,0.4) **@{-}, 
(-5.6,-0.4);(-4,-2) **@{-}, 
(-8,-2);(-4,2) **@{-},
(12,2);(10.4,0.4) **@{-}, 
(9.6,-0.4);(8,-2) **@{-},
(12,-2);(8,2) **@{-},
(5.6,-0.4);(4,-2) **@{-}, 
(8,2);(6.6,0.4) **@{-},
(8,-2);(4,2) **@{-},
(-9,0.4)*{\urcorner}, 
(-9,-0.8)*{\lrcorner}, 
(11,0.5)*{\urcorner}, 
(11,-0.8)*{\lrcorner}, 
(-2,4);(-2,8) **\crv{(1,6)}, 
(2,8);(2,4) **\crv{(-1,6)},
(-2,4);(2,4) **\crv{(0,1)}, 
(-2,0);(2,0) **\crv{(0,3)},
 \endxy~ \gg_{a^9+1}^\ast  
 +xy
\ll ~\xy 
(-12,8);(-2,8) **@{-}, 
(2,8);(12,8) **@{-}, 
(-12,8);(-12,2) **@{-}, 
(12,8);(12,2) **@{-}, 
(-12,-8);(-2,-8) **@{-}, 
(-2,-8);(12,-8) **@{-}, 
(-12,-8);(-12,-2) **@{-}, 
(12,-8);(12,-2) **@{-},
(2,0);(4,-2) **@{-}, 
(-2,0);(-4,-2) **@{-},  
(2.4,0.4);(4,2) **@{-},
(0.1,-1.9);(1.7,-0.2) **@{-},  
(0.1,-1.9);(-1.7,-0.3) **@{-},  
(-2.4,0.4);(-4,2) **@{-},
(-12,2);(-10.4,0.4) **@{-}, 
(-9.6,-0.4);(-8,-2) **@{-}, 
(-12,-2);(-8,2) **@{-},
(-8,2);(-6.4,0.4) **@{-}, 
(-5.6,-0.4);(-4,-2) **@{-}, 
(-8,-2);(-4,2) **@{-},
(12,2);(10.4,0.4) **@{-}, 
(9.6,-0.4);(8,-2) **@{-},
(12,-2);(8,2) **@{-},
(5.6,-0.4);(4,-2) **@{-}, 
(8,2);(6.6,0.4) **@{-},
(8,-2);(4,2) **@{-},
(-9,0.4)*{\urcorner}, 
(-9,-0.8)*{\lrcorner}, 
(11,0.5)*{\urcorner}, 
(11,-0.8)*{\lrcorner}, 
(-2,4);(-2,8) **\crv{(1,6)}, 
(2,8);(2,4) **\crv{(-1,6)},
(-2,0);(-2,4) **\crv{(1,2)}, 
(2,4);(2,0) **\crv{(-1,2)},
 \endxy~ \gg_{a^9+1}^\ast \\ 
 &\hskip 1cm +yx
\ll  ~\xy 
(-12,8);(-2,8) **@{-}, 
(2,8);(12,8) **@{-}, 
(-12,8);(-12,2) **@{-}, 
(12,8);(12,2) **@{-}, 
(-12,-8);(-2,-8) **@{-}, 
(-2,-8);(12,-8) **@{-}, 
(-12,-8);(-12,-2) **@{-}, 
(12,-8);(12,-2) **@{-},
(2,0);(4,-2) **@{-}, 
(-2,0);(-4,-2) **@{-}, 
(2.4,0.4);(4,2) **@{-},
(0.1,-1.9);(1.7,-0.2) **@{-},  
(0.1,-1.9);(-1.7,-0.3) **@{-},  
(-2.4,0.4);(-4,2) **@{-},
(-12,2);(-10.4,0.4) **@{-}, 
(-9.6,-0.4);(-8,-2) **@{-}, 
(-12,-2);(-8,2) **@{-},
(-8,2);(-6.4,0.4) **@{-}, 
(-5.6,-0.4);(-4,-2) **@{-}, 
(-8,-2);(-4,2) **@{-},
(12,2);(10.4,0.4) **@{-}, 
(9.6,-0.4);(8,-2) **@{-},
(12,-2);(8,2) **@{-},
(5.6,-0.4);(4,-2) **@{-}, 
(8,2);(6.6,0.4) **@{-},
(8,-2);(4,2) **@{-},
(-9,0.4)*{\urcorner}, 
(-9,-0.8)*{\lrcorner}, 
(11,0.5)*{\urcorner}, 
(11,-0.8)*{\lrcorner},  
(-2,8);(2,8) **\crv{(0,5)}, 
(-2,4);(2,4) **\crv{(0,7)},
(-2,4);(2,4) **\crv{(0,1)}, 
(-2,0);(2,0) **\crv{(0,3)},
 \endxy~ \gg_{a^9+1}^\ast  
 + y^2
\ll  ~\xy 
(-12,8);(-2,8) **@{-}, 
(2,8);(12,8) **@{-}, 
(-12,8);(-12,2) **@{-}, 
(12,8);(12,2) **@{-}, 
(-12,-8);(-2,-8) **@{-}, 
(-2,-8);(12,-8) **@{-}, 
(-12,-8);(-12,-2) **@{-}, 
(12,-8);(12,-2) **@{-},
(2,0);(4,-2) **@{-}, 
(-2,0);(-4,-2) **@{-},  
(2.4,0.4);(4,2) **@{-},
(0.1,-1.9);(1.7,-0.2) **@{-},  
(0.1,-1.9);(-1.7,-0.3) **@{-},  
(-2.4,0.4);(-4,2) **@{-},
(-12,2);(-10.4,0.4) **@{-}, 
(-9.6,-0.4);(-8,-2) **@{-}, 
(-12,-2);(-8,2) **@{-},
(-8,2);(-6.4,0.4) **@{-}, 
(-5.6,-0.4);(-4,-2) **@{-}, 
(-8,-2);(-4,2) **@{-},
(12,2);(10.4,0.4) **@{-}, 
(9.6,-0.4);(8,-2) **@{-},
(12,-2);(8,2) **@{-},
(5.6,-0.4);(4,-2) **@{-}, 
(8,2);(6.6,0.4) **@{-},
(8,-2);(4,2) **@{-},
(-9,0.4)*{\urcorner}, 
(-9,-0.8)*{\lrcorner}, 
(11,0.5)*{\urcorner}, 
(11,-0.8)*{\lrcorner}, 
(-2,0);(-2,4) **\crv{(1,2)}, 
(2,4);(2,0) **\crv{(-1,2)},
(-2,8);(2,8) **\crv{(0,5)}, 
(-2,4);(2,4) **\crv{(0,7)},
 \endxy~ \gg_{a^9+1}^\ast \\ 
 &= x^2 \ll O^2 \gg_{a^9+1}^\ast  
 + xy \ll ~\xy 
(2,0);(4,-2) **@{-}, 
(-2,0);(-4,-2) **@{-},  
(2.4,0.4);(4,2) **@{-},
(0.1,-1.9);(1.7,-0.2) **@{-},  
(0.1,-1.9);(-1.7,-0.3) **@{-},  
(-2.4,0.4);(-4,2) **@{-},
(-12,2);(-10.4,0.4) **@{-}, 
(-9.6,-0.4);(-8,-2) **@{-}, 
(-12,-2);(-8,2) **@{-},
(-8,2);(-6.4,0.4) **@{-}, 
(-5.6,-0.4);(-4,-2) **@{-}, 
(-8,-2);(-4,2) **@{-},
(12,2);(10.4,0.4) **@{-}, 
(9.6,-0.4);(8,-2) **@{-},
(12,-2);(8,2) **@{-},
(5.6,-0.4);(4,-2) **@{-}, 
(8,2);(6.6,0.4) **@{-},
(8,-2);(4,2) **@{-},
(-9,0.4)*{\urcorner}, 
(-9,-0.8)*{\lrcorner}, 
(11,0.5)*{\urcorner}, 
(11,-0.8)*{\lrcorner}, 
(-12,2);(-2,4) **\crv{(-14,5)}, 
(2,4);(12,2) **\crv{(14,5)},
(-2,0);(-2,4) **\crv{(1,2)}, 
(2,4);(2,0) **\crv{(-1,2)},
(-12,-2);(0,-5) **\crv{(-14,-5)},
(0,-5);(12,-2) **\crv{(14,-5)},
 \endxy~ \gg_{a^9+1}^\ast \\
&\hskip 1cm +yx \ll O^3  \gg_{a^9+1}^\ast  
+y^2 \ll O^2 \gg_{a^9+1}^\ast \\
&= 
  xy \nabla [ ~\xy 
(2,0);(4,-2) **@{-}, 
(-2,0);(-4,-2) **@{-},  
(2.4,0.4);(4,2) **@{-},
(0.1,-1.9);(1.7,-0.2) **@{-},  
(0.1,-1.9);(-1.7,-0.3) **@{-},  
(-2.4,0.4);(-4,2) **@{-},
(-12,2);(-10.4,0.4) **@{-}, 
(-9.6,-0.4);(-8,-2) **@{-}, 
(-12,-2);(-8,2) **@{-},
(-8,2);(-6.4,0.4) **@{-}, 
(-5.6,-0.4);(-4,-2) **@{-}, 
(-8,-2);(-4,2) **@{-},
(12,2);(10.4,0.4) **@{-}, 
(9.6,-0.4);(8,-2) **@{-},
(12,-2);(8,2) **@{-},
(5.6,-0.4);(4,-2) **@{-}, 
(8,2);(6.6,0.4) **@{-},
(8,-2);(4,2) **@{-},
(-9,0.4)*{\urcorner}, 
(-9,-0.8)*{\lrcorner}, 
(11,0.5)*{\urcorner}, 
(11,-0.8)*{\lrcorner}, 
(-12,2);(-2,4) **\crv{(-14,5)}, 
(2,4);(12,2) **\crv{(14,5)},
(-2,0);(-2,4) **\crv{(1,2)}, 
(2,4);(2,0) **\crv{(-1,2)},
(-12,-2);(0,-5) **\crv{(-14,-5)},
(0,-5);(12,-2) **\crv{(14,-5)},
 \endxy~ ](a^3 - a^{-3}) \\
&= 9xy. 
\end{align*}
\end{example}

\begin{remark}
The polynomial $\ll D \gg_{a^9+1}^\ast$ $(\in \mathbb Z[a^{-1},a,x,y]/ I(a^9+1, a^{-6} + 1 + a^6))$ is not 
invariant under Yoshikawa moves $\Gamma_6, \Gamma'_6, \Gamma_7$ and $\Gamma_8$.  
When we evaluate it with $x=1$ and $y=1$, it becomes invariant under $\Gamma_6$ and $\Gamma'_6$.  
In Section~\ref{sect:ribbon} we will observe that $\ll D \gg_{a^9+1}^\ast$ $(\in \mathbb Z[a^{-1},a,x,y]/ I(a^9+1, a^{-6} + 1 + a^6))$ 
is an invariant when we restrict to \lq\lq ribbon marked graphs\rq\rq.  
\end{remark} 

\section{Ribbon marked graphs}
\label{sect:ribbon}

Let $D$ be a  marked graph diagram.  We call a pair of marked vertices of $D$ 
a {\it ribbon pair} if they are the vertices of 
a bigon in $D$ and the markers are not parallel. 

\begin{definition} 
A marked  graph diagram is called  {\it ribbon} if the marked vertices are divided into ribbon pairs.  A marked graph is called {\it ribbon} if there is a ribbon marked graph diagram presenting the marked graph. 
\end{definition} 

For example, the marked graph diagrams $8_1$ and $9_1$ in Examples~\ref{examp-8_1-1} and \ref{examp-9_1-1} are ribbon.  

A surface-link is called {\it ribbon} if it is obtained from a trivial $2$-link by surgery along some $1$-handles attaching it.  It is known that for any admissible ribbon marked graph diagram $D$, the surface-link ${\mathcal L}(D)$ presented by $D$ is a ribbon surface-link, and conversely that  
for any ribbon surface-link $\mathcal L$ there is an admissible ribbon marked graph diagram $D$ such that ${\mathcal L}(D)$ is equivalent to $\mathcal L$.  

When $D$ is ribbon, the positive resolution $L_+(D)$ and the negative resolution $L_-(D)$ are isotopic diagrams.  

Let $D$ be an oriented ribbon marked graph diagram presenting a ribbon $2$-knot.  
Let $n$ be the number of ribbon pairs of marked vertices.  The number of marked vertices is $2n$, and $L_+(D)$ and $L_-(D)$ are diagrams of a trivial link with $n+1$ components. 

\begin{theorem} \label{thm:m=9:C}
Let $D$ be an admissible oriented ribbon marked graph diagram with $2n$ marked vertices.  If ${\mathcal L}(D)$ is a $2$-knot, then $\ll D \gg_{a^9+1}^\ast $ is a multiple of $(xy)^n$, and 
$(xy)^{-n} \ll D \gg_{a^9+1}^\ast $ is determined from the equivalence class of the ribbon $2$-knot. 

Thus, the restriction of the invariant $\ll \cdot \gg_{a^9+1}^\ast$ of marked graphs   to ribbon marked graphs presenting $2$-knots is an invariant of ribbon $2$-knots, 
up to multiplication by powers of $xy$. 
\end{theorem}

\begin{proof} 
 (1) Assume that $n \geq 1$. 
Let $\sigma_\ast$ (or $\sigma_{\ast\ast}$, resp.) be the state of $D$ 
assigning $T(\infty)$ (or $T(0)$, resp.) 
to every marked vertex.  
There is a unique state, say $\sigma_0$, in ${\mathcal S}(D) -\{ \sigma_\ast, 
\sigma_{\ast\ast} \}$ such that $D_{\sigma_0}$ is a diagram of a knot.  
For any state $\sigma$ in ${\mathcal S}(D) -\{ \sigma_\ast, 
\sigma_{\ast\ast}, \sigma_0\}$, $D_\sigma$ is a disjoint union 
$D_{\sigma_0} \sqcup O^\mu$ for some $\mu = \mu(\sigma) \geq 1$.  
By Proposition~\ref{prop:m=9*},  
\begin{equation*}
\ll D \gg_{a^9+1}^\ast = \sum_{\sigma \in \mathcal S(D)}
x^{\sigma(\infty)}y^{\sigma(0)} 
\nabla[D_\sigma](a^3 - a^{-3}) = 
x^{\sigma_0(\infty)}y^{\sigma_0(0)} 
\nabla[D_{\sigma_0}](a^3 - a^{-3}). 
\end{equation*}
Since $\sigma_0$ assigns $T(\infty)$ to one marked vertex and $T(0)$ to another marked vertex for each ribbon pair of marked vertices, we have 
$\sigma_0(\infty) = \sigma_0(0)=n$.  Thus 
\begin{equation*}
\ll D \gg_{a^9+1}^\ast  = (xy)^n \nabla[D_{\sigma_0}](a^3 - a^{-3}). 
\end{equation*}
Let $K$ be the knot presented by the knot diagram $D_{\sigma_0}$. 
Then 
\begin{equation}\label{eqn:ribbon}
(xy)^{-n} \ll D \gg_{a^9+1}^\ast = \nabla [K](a^3 - a^{-3}). 
\end{equation}

(2) When $n=0$, $D$ is a diagram of a trivial knot and  $\ll D \gg_{a^9+1}^\ast=1$. 
The $2$-knot ${\mathcal L}(D)$ is a trivial $2$-knot.  Let $K$ be the trivial knot.  Then the same equality with (\ref{eqn:ribbon}) holds. 

In either case (1) or (2), the ribbon $2$-knot ${\mathcal L}(D)$ is equivalent to a $2$-knot in a normal form in the sense of \cite{KSS} such that it is symmetric with respect to ${\mathbb R}^3_0$ and the cross-sectional knot appearing at ${\mathbb R}^3_0$ is $K$.  
For such a knot $K$ it is known that the Conway polynomial 
$\nabla[K](z)$ is determined from the equivalence class of the ribbon $2$-knot. 
In particular, when the $2$-knot is trivial,  $\nabla[K](z)=1$.  
Thus  we see that $(xy)^{-n} \ll D \gg_{a^9+1}^\ast$ is determined from the equivalence class of the ribbon $2$-knot ${\mathcal L}(D)$. 
 \end{proof}
 
Let $\varpi=\exp(\frac{2\pi \sqrt{-1}}{18}) \in \mathbb C$, a primitive $18$th root of unity.  Since $\varpi^9 +1 = \varpi^{-6} +1 + \varpi^6 =0$, evaluation of 
$\ll D \gg_{a^9+1}^\ast$ with $a = \varpi$ is well-defined.  Define $P_9^\ast(D)$ by 
\begin{equation*}
P_9^\ast(D) = \ll D \gg_{a^9+1}^\ast |_{a = \varpi} \quad \in {\mathbb C}[x,y]. 
\end{equation*}
Since $\varpi^{-3} = \varpi^3 = \sqrt{-3}$, it follows from Theorem~\ref{prop:m=9*} that 
\begin{equation*}
P_9^\ast(D) = \sum_{\sigma \in \mathcal S(D)}
x^{\sigma(\infty)}y^{\sigma(0)}
\nabla[{D_\sigma}] (\sqrt{-3}).   
\end{equation*} 

From Theorem~\ref{thm:m=9:C} and its proof, 
we obtain the following. 

\begin{theorem} \label{thm:m=9:D}
Let $D$ be an admissible oriented ribbon marked graph diagram with $2n$ marked vertices.  If ${\mathcal L}(D)$ is a $2$-knot, then $P_9^\ast(D)$ is a monomial $c(xy)^n$ and the coefficient $c \in {\mathbb C}$ is determined from the equivalence class of the ribbon $2$-knot. 
When the $2$-knot is a trivial $2$-knot, then $c=1$. 
\end{theorem}

For example, for the marked graph diagram $8_1$ in Example~\ref{examp-8_1-1}, the complex number $c$ is $9$.   Thus we see that the $2$-knot is not equivalent to a trivial $2$-knot. 


\section{Proof of Lemma~\ref{3-tangle basis}}
\label{sect-pf-prop-e3t}

Let $G$ be an oriented tangled trivalent graph diagram in a $2$-disk $D^2$ with the end points $s_1, \ldots, s_6$ as shown in Figure \ref{fig-3tmd2} such that there are no connected components as diagrams in ${\rm Int}D^2$. If $G$ has no crossings, then $G$ forms a tiling, say $G^\ast$, of $D^2$, where we regard $\{s_1, \dots, s_6 \}$ and $\{s_i s_{i+1} | 1 \leq i \leq 6\}$ as trivalent vertices and edges of the tiling $G^\ast$. Here $s_i s_{i+1}$ means an edge on $\partial D^2$ whose end points are $s_i$ and $s_{i+1}$, and we assume $s_7=s_1$. 

\begin{figure}[h]
\begin{center}
\resizebox{0.20\textwidth}{!}{%
  \includegraphics{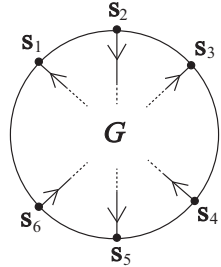}}
\caption{An oriented tangled trivalent graph diagram $G$ in $D^2$}\label{fig-3tmd2}
\end{center}
\end{figure}

Let $V, E$ and $F$ be the numbers of the vertices, edges and faces of the tiling $G^\ast$ of $D^2$, respectively, i.e., $V$ and $E$ are the number of vertices and edges of $G^\ast$, and $F$ is the number of regions of $D^2 \backslash G^\ast$. 
For each integer $a\geq1$, let $F_a$ denote the number of faces of the tiling $G^\ast$ that are $a$-gons. By the definition of oriented tangled trivalent graph diagram, there are only $2i$-gons  ($i\geq 1$).
Then 
\begin{align*}
&F=F_2+F_4+F_6+F_8+F_{10}+F_{12}+\cdots,\\
&E=\frac{1}{2} (6+2F_2+4F_4+6F_6+8F_8+10F_{10}+12F_{12}+\cdots) =3+T,\\
&V=\frac{2}{3}E =\frac{2}{3}(3+T)=2+\frac{2}{3}T,~\text{where}~T:=\sum^{\infty}_{i=1}i F_{2i}.
\end{align*}
Considering the Euler characteristic of the tiling $G^\ast$ of $D^2$, we have
\begin{align*}
1&=V-E+F=2+\frac{2}{3}T -(3+T)+F = -1-\frac{1}{3}T +F\\
&=-1+\frac{1}{3} (2F_2+F_4-F_8-2F_{10}-3F_{12}-\cdots).
\end{align*}
This gives
\begin{align}\label{equation-3tangle}
&2F_2+F_4 - 6= F_8+2F_{10}+3F_{12}+\cdots.
\end{align}

For each $i$ with $1 \leq i \leq 6$, we denote by $\overline{s_is_{i+1}}$ a proper simple arc in $D^2$ whose end points are $s_i$ and $s_{i+1}$.  Note that when $G$ has no $2$-gons in ${\rm Int}D^2$,  any $2$-gon in $G^\ast$, if there exists, is   $s_i s_{i+1} \cup \overline{s_is_{i+1}}$  for some $i$.  

\begin{lemma}\label{7-move-tangle-lem-1}
Let $G$ be an oriented tangled trivalent graph diagram in $D^2$ with the end points $s_1, \ldots, s_6$ as shown in Figure \ref{fig-3tmd2} such that there are no connected components as diagrams in ${\rm Int}D^2$. Suppose that $G$ has no crossings, $2$-gons and $4$-gons in ${\rm Int}D^2$. Then the tiling $G^\ast$ of $D^2$ is one of the tilings in Figure~\ref{fig-m07at}.  
\end{lemma}

\begin{figure}
\begin{center}
\resizebox{0.7\textwidth}{!}{%
  \includegraphics{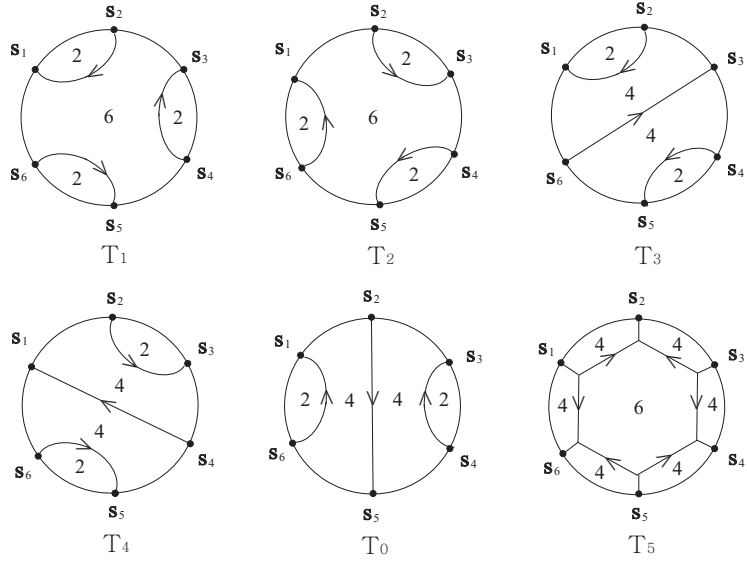}}
\caption{Tilings $G^\ast$  of $D^2$ by $G$ 
}\label{fig-m07at}
\end{center}
\end{figure}

\begin{proof}
It follows from (\ref{equation-3tangle}) that
\begin{equation}\label{contradiction-3tangle}
2F_2+F_4 - 6 \geq 0.
\end{equation}
Since there are no $2$-gons in ${\rm Int}D^2$, we see that one of two edges of any $2$-gon in $G$ lies in the boundary $\partial D^2$. Since $G$ has only trivalent vertices, any two distinct $2$-gones cannot share a vertex $s_i (1\leq i \leq 6)$ in common. This gives that $0\leq F_2 \leq 3$.

{\bf Case I.} Suppose $F_2=3$.  The tiling $G^\ast$ is  $T_1$ or $T_2$ in Figure~\ref{fig-m07at}. 

\smallskip

{\bf Case II.} Suppose $F_2=2$. It follows from (\ref{contradiction-3tangle}) that $F_4 \geq 2$. 

(i) Suppose that the $2$-gons 
are $\{s_1s_2\cup\overline{s_1s_2}, s_3s_4\cup\overline{s_3s_4}\}$ or they are in a position obtained by rotating $\{s_1s_2\cup\overline{s_1s_2}, s_3s_4\cup\overline{s_3s_4}\}$.   
Consider  the case of $\{s_1s_2\cup\overline{s_1s_2}, s_3s_4\cup\overline{s_3s_4}\}$.  Then edges 
$s_6s_1, \overline{s_1s_2}, s_2s_3, \overline{s_3s_4}, s_4s_5$ are edges of the same $n$-gon for some $n >6$. 
Since any $4$-gon of $G^\ast$ must have $s_5s_6$ as an edge, we have $F_4 \leq 1$. This yields a contradiction. Similarly, when the $2$-gons are in a position that is obtained by rotating $\{s_1s_2\cup\overline{s_1s_2}, s_3s_4\cup\overline{s_3s_4}\}$, we have a contradiction.  Thus case (i) does not occur.  

(ii) Suppose that the $2$-gons are 
$\{s_1s_2\cup\overline{s_1s_2}, s_4s_5\cup\overline{s_4s_5}\}$ or they are in a position obtained by rotating 
$\{s_1s_2\cup\overline{s_1s_2}, s_4s_5\cup\overline{s_4s_5}\}$.  
Consider the case of $\{s_1s_2\cup\overline{s_1s_2}, s_4s_5\cup\overline{s_4s_5}\}$.  
Let $A$ be a $4$-gon in this tiling.  
Since there are no $4$-gons in {\rm Int}$D^2$, one of the four edges of $A$ have to be $s_2s_3$, $s_3s_4$, $s_5s_6$ or $s_1s_6$. If $s_2s_3$ or $s_6s_1$ is an edge of $A$, then $A = s_6s_1 \cup \overline{s_1s_2} \cup 
s_2s_3 \cup \overline{s_3s_6}$.  
Then the tiling is  $T_3$ in Figure~\ref{fig-m07at}.  If $s_3s_4$ or $s_5s_6$  is an edge of  $A$, then by the same reason, the tiling is  $T_3$.  Thus we have $T_3$ in case this case.  For the other cases, the tilings are obtained by rotating $T_3$, which are $T_4$ and $T_0$.   

\smallskip

{\bf Case III.} 
Suppose $F_2=1$. It follows from (\ref{contradiction-3tangle}) that $F_4 \geq 4$. 

Consider a case where the $2$-gon is $s_1s_{2} \cup \overline{s_1s_2}$.  

(i) If $s_2s_3$ or $s_6s_1$ is an edge of a $4$-gon $A$ in $G^\ast$, then  $A= s_6s_1\cup \overline{s_1s_2} \cup s_2s_3 \cup \overline{s_3s_6}$.  Since $F_4 \geq 4$, there are three  $4$-gons besides $A$. 
Each of them has edge $s_3, s_4, s_4s_5, s_5s_6$. However, a $4$-gon having $s_3s_4$ or $s_5s_6$ is a $4$-gon $s_5s_6 \cup \overline{s_6s_3} \cup s_3s_4 \cup \overline{s_4s_5}$.  Then the tiling is $T_3$, and this contradicts to the assumption $F_2=1$.  Thus this is not the case.  

(ii) If neither $s_2s_3$ nor $s_6s_1$ is an edge of a $4$-gon in $G^\ast$, then $F_4 \leq 3$. This contradicts to 
$F_4 \geq 4$. Thus this is not the case.   

\smallskip
{\bf Case IV.} Suppose $F_2=0$. It follows from (\ref{contradiction-3tangle}) that $F_4 \geq 6$. Since each $4$-gon has an edge  in $\partial D^2$, we have $T_5$.  
\end{proof}

\begin{figure}
\begin{center}
\resizebox{0.50\textwidth}{!}{%
  \includegraphics{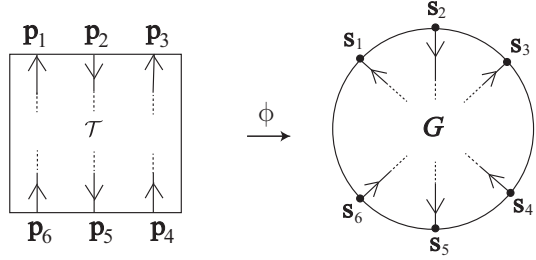}}
\caption{A homeomorphism $\phi:I^2 \rightarrow D^2$}\label{fig-3tm07}
\end{center}
\end{figure}

\begin{proof}[\bf Proof of Lemma \ref{3-tangle basis}]
Let $\mathcal T$ be a $3$-tangle diagram with the boundary (a) in Figure~\ref{fig-bdary-t} such that there are no crossings, $2$-gons and $4$-gons and there are no connected components as diagrams in ${\rm Int}D^2$. Let $p_1, \ldots, p_6$ be the end points of $\mathcal T$ and let $\phi:I^2 \rightarrow D^2$ be a homeomorphism from $I^2$ onto a $2$-disk $D^2$ with $\phi(p_i)=s_i$ $(1\leq i\leq 6)$. See Figure~\ref{fig-3tm07}. 
Putting $G= \phi(\mathcal T)$, we obtain the result from Lemma~\ref{7-move-tangle-lem-1}. 
\end{proof}


\section{Proof of Lemma~\ref{4-tangle basis}}
\label{sect-pf-prop-e4t}

Let $G$ be an oriented tangled trivalent graph diagram in a $2$-disk $D^2$ with the end points $t_1, \ldots, t_8$ as shown in Figure \ref{fig-4tm08-1} such that there are no connected components as diagrams in ${\rm Int}D^2$. If $G$ has no crossings, then $G$ forms a tiling, say $G^\ast$,  of $D^2$, where we regard $\{t_1, \dots, t_8 \}$ and $\{t_i t_{i+1} \mid 1 \leq i \leq 8\}$ as trivalent vertices and edges of the tiling $G$, respectively. Here $t_it_{i+1}$ means an edge of $\partial D^2$ whose end points are $t_i$ and $t_{i+1}$, and we assume $t_9=t_1$. 

\begin{figure}[ht]
\begin{center}
\resizebox{0.20\textwidth}{!}{%
  \includegraphics{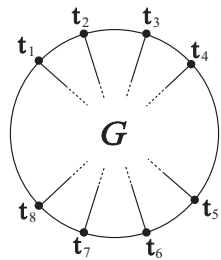}}
\caption{A 4-tangle of trivalent graph diagram} \label{fig-4tm08-1}
\end{center}
\end{figure}

Let $V, E$ and $F$ be the numbers of the vertices, edges and faces of the tiling $G^\ast$, respectively. Then 
\begin{align*}
&F=F_2+F_4+F_6+F_8+F_{10}+F_{12}+F_{14}+\cdots,\\
&E=\frac{1}{2}(8+2F_2+4F_4+6F_6+8F_8+10F_{10}+12F_{12}+14F_{14}+\cdots)=4+T,\\
&V=\frac{2}{3}E =\frac{2}{3}(4+T) = \frac{8}{3}+\frac{2}{3}T,~\text{where}~T=\sum^{\infty}_{i=1}i F_{2i}.
\end{align*}
Considering the Euler characteristic of the tiling $G^\ast$ of $D^2$, we have
\begin{equation*}
1=V-E+F=\frac{8}{3}+\frac{2}{3}T -(4+T)+F = -\frac{4}{3}-\frac{1}{3}T+F.
\end{equation*}
This gives
\begin{equation}\label{equation-4tangle}
2F_2+F_4 - 7 = F_8 + 2F_{10}  + 3F_{12} + 4F_{14}+\cdots.  
\end{equation}

For each $i$ with $1 \leq i \leq 8$, we denote by $\overline{t_it_{i+1}}$ a proper simple arc in $D^2$ whose end points are $t_i$ and $t_{i+1}$.  Note that when $G$ has no $2$-gons in ${\rm Int}D^2$,  any $2$-gon in $G^\ast$, if there exists, is   $t_i t_{i+1} \cup \overline{t_i t_{i+1}}$  for some $i$.

\begin{lemma}\label{8-move-tangle-lem-1}
Let $G$ be an oriented tangled trivalent graph diagram in $D^2$ with the end points $t_1, \ldots, t_8$ as shown in Figure \ref{fig-4tm08-1} such that there are no connected components as diagrams in ${\rm Int}D^2$. Suppose that $G$ has no crossings, $2$-gons and $4$-gons in ${\rm Int}D^2$. Then the tiling $G^\ast$ is one of the tilings in Figures~\ref{fig-m08at-1}, \ref{fig-m08at-2}, \ref{fig-m08at-3}, \ref{fig-m08at-4} and \ref{fig-m08at-5}. 
\end{lemma}

\begin{proof}
Since there are no $2$-gons in ${\rm Int}D^2$, one of two edges of any $2$-gon in $G$ lies in  $\partial D^2$. Since $G^\ast$ has only trivalent vertices, any two distinct $2$-gons cannot share a vertex $t_i$ $(1\leq i \leq 8)$ in common. This gives that $0\leq F_2 \leq 4$. 

{\bf Case I.} Suppose $F_2=4$. 
Then $G^\ast$ is one of the tilings shown in Figure~\ref{fig-m08at-1}. 

\begin{figure}[ht]
\begin{center}
\resizebox{0.40
\textwidth}{!}{%
  \includegraphics{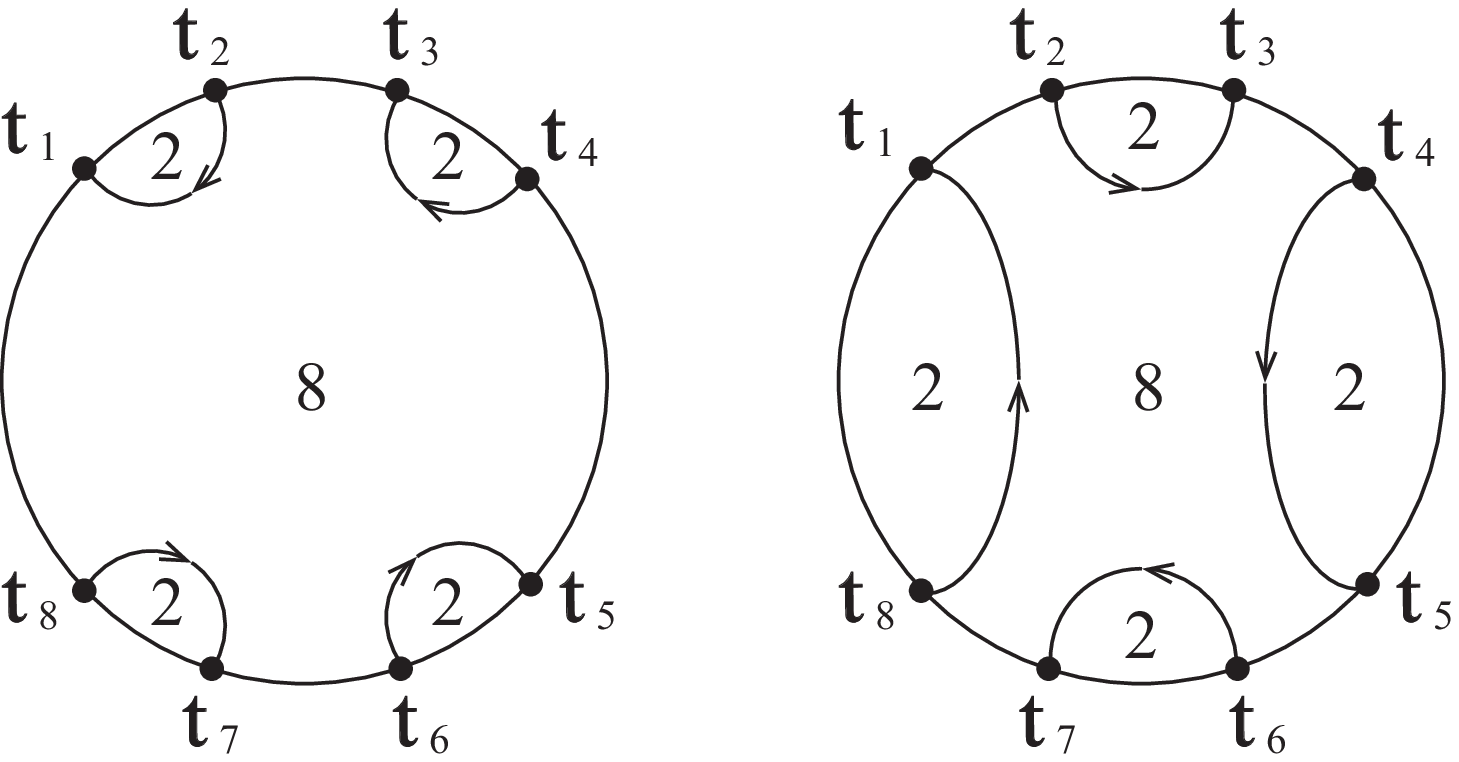}}
\caption{$F_2=4$} \label{fig-m08at-1}
\end{center}
\end{figure}

{\bf Case II.}  Suppose $F_2=3$. It follows from (\ref{equation-4tangle}) that $0\leq F_8 \leq F_4 -1$. So $F_4 \geq 1$. 

(i) 
Suppose that the three $2$-gons are 
$\{ t_1 t_2 \cup \overline{t_1 t_2},  t_3 t_4 \cup \overline{t_3 t_4},   t_5 t_6 \cup \overline{t_5 t_6} \}$ 
or they are in a position obtained by rotating 
$\{ t_1 t_2 \cup \overline{t_1 t_2},  t_3 t_4 \cup \overline{t_3 t_4},   t_5 t_6 \cup \overline{t_5 t_6} \}$.  

Consider the case of $\{ t_1 t_2 \cup \overline{t_1 t_2},  t_3 t_4 \cup \overline{t_3 t_4},   t_5 t_6 \cup \overline{t_5 t_6} \}$.  Edges $t_8 t_1, \overline{t_1 t_2}, t_2 t_3, \overline{t_3 t_4},$ $t_4 t_5, \overline{t_5 t_6}, t_6 t_7$ are edges of the same $n$-gon for some $n \geq 8$. Thus by (\ref{equation-4tangle}), we have $F_4 \geq 2$. 
On the other hand, 
any  $4$-gon of $G^\ast$ has $t_7 t_8$ as an edge.  Thus $F_4 \leq 1$.  This is a contradiction. 
In the other cases of (i), we have a contradiction. Thus the case (i) does not occur.  

(ii) 
Suppose that the three $2$-gons are 
$\{ t_1 t_2 \cup \overline{t_1 t_2},  t_3 t_4 \cup \overline{t_3 t_4},   t_6 t_7 \cup \overline{t_6 t_7} \}$ 
or they are in a position obtained by rotating 
$\{ t_1 t_2 \cup \overline{t_1 t_2},  t_3 t_4 \cup \overline{t_3 t_4},   t_6 t_7 \cup \overline{t_6 t_7} \}$.  

Consider the case of $\{ t_1 t_2 \cup \overline{t_1 t_2},  t_3 t_4 \cup \overline{t_3 t_4},   t_6 t_7 \cup \overline{t_6 t_7} \}$.  Edges $t_8 t_1, \overline{t_1 t_2}, t_2 t_3, \overline{t_3 t_4},$ $t_4 t_5$ are edges of the same $n$-gon for some $n \geq 6$. Since $F_4 \geq 1$, there is a $4$-gon $A$ in $G^\ast$.  Since $A$ has $t_5 t_6$ or $t_7 t_8$ as an edge, we have $A= t_5 t_6 \cup  \overline{t_6 t_7} \cup t_7 t_8 \cup  \overline{t_8 t_5}$.  Then $G^\ast$ is one of the tilings in 
Figure~\ref{fig-m08at-2}.  In the other cases of (ii), we have the other tilings in Figure~\ref{fig-m08at-2}. 

\begin{figure}[ht]
\begin{center}
\resizebox{0.76
\textwidth}{!}{%
  \includegraphics{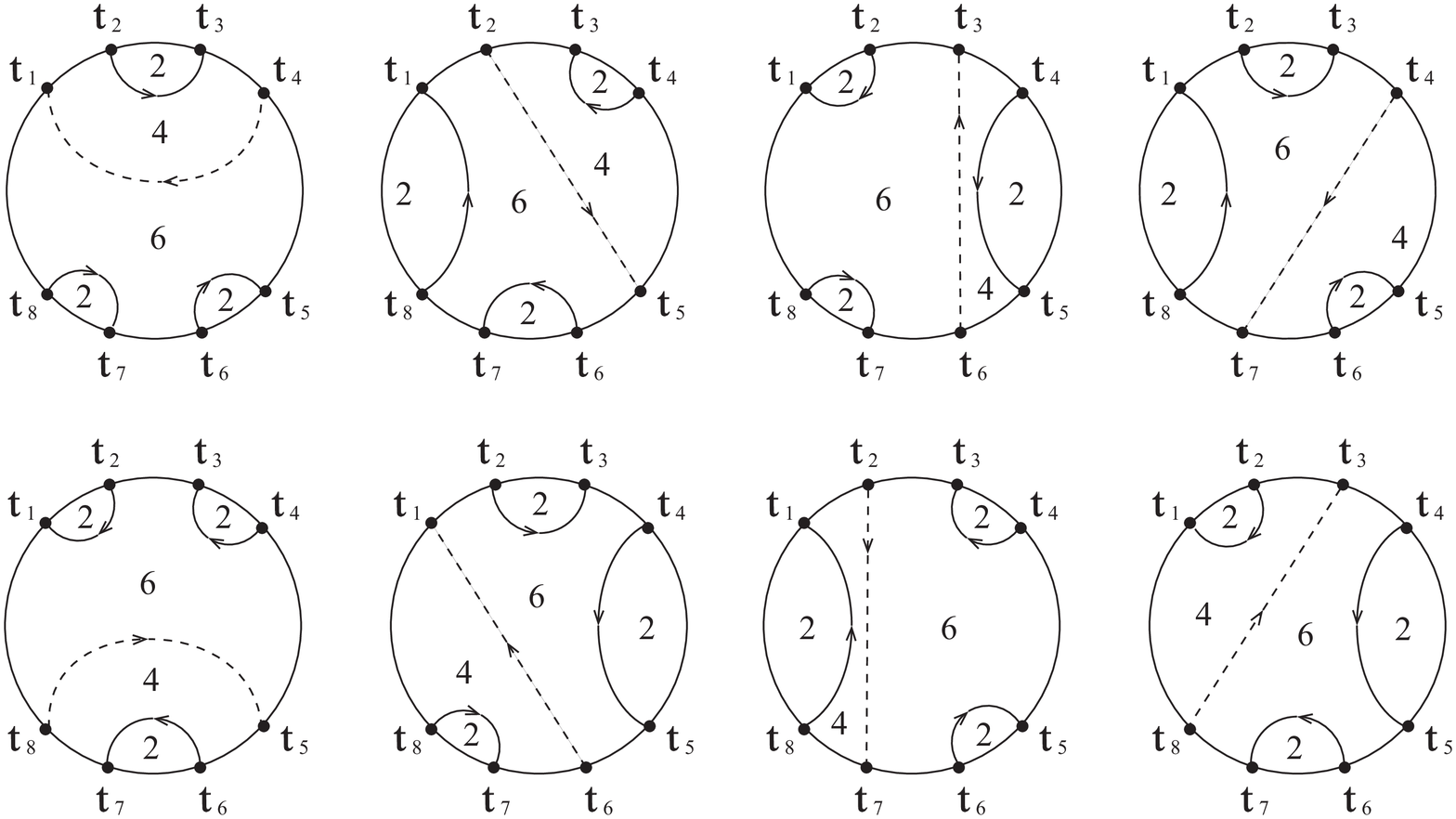}}
\caption{$F_2=3$} \label{fig-m08at-2}
\end{center}
\end{figure}

{\bf Case III.} $F_2=2$. It follows from (\ref{equation-4tangle}) that $0\leq F_8 \leq  F_4-3$ and so $F_4 \geq  3$.  

(i) Suppose that the $2$-gons are 
$\{ t_1 t_2 \cup \overline{t_1 t_2}, t_3 t_4 \cup \overline{t_3 t_4} \}$ or in a position obtained by rotating them.  
Consider the case of $\{ t_1 t_2 \cup \overline{t_1 t_2}, t_3 t_4 \cup \overline{t_3 t_4}\}$.  
Edges $t_8 t_1, \overline{t_1 t_2}, t_2 t_3, \overline{t_3 t_4}, t_4 t_5$ are edges of the same $n$-gon for some $n \geq 6$. It implies $F_4 \leq 1$ as before. This contradicts to $F_4 \geq  3$.  We have a contradiction in other cases of (i). Thus the case (i) does not occur.  

(ii) Suppose that the $2$-gons are 
$\{ t_1 t_2 \cup \overline{t_1 t_2}, t_4 t_5 \cup \overline{t_4 t_5} \}$ or in a position obtained by rotating them.  
Consider the case of $\{ t_1 t_2 \cup \overline{t_1 t_2}, t_4 t_5 \cup \overline{t_4 t_5}\}$.  

(a) If $t_2 t_3$ or $t_8 t_1$ is an edge of a $4$-gon, then the $4$-gon is $t_8 t_1 \cup \overline{t_1 t_2} \cup t_2 t_3 \cup  \overline{t_3 t_8}$.  Edges $t_7 t_8, \overline{t_8 t_3}, t_3 t_4, \overline{t_4 t_5}, t_5 t_6$ are edges of the same $n$-gon for some $n \geq 6$.  Then another $4$-gon, if there exists, must have $t_6 t_7$ as an edge. Thus $F_4 \leq 2$, which contradicts to $F_4 \geq  3$.  

(b) If $t_3 t_4$ or $t_5 t_6$ is an edge of a $4$-gon, then by the same argument with (a), we have a contradiction.  

By (a) and (b), if there is a $4$-gon then it has $t_6 t_7$ or $t_7 t_8$. Thus $F_4 \leq 2$, which contradicts to $F_4 \geq  3$.  
Therefore we see that (ii) does not occur.  

(iii) Suppose that the $2$-gons are 
$\{ t_1 t_2 \cup \overline{t_1 t_2}, t_5 t_6 \cup \overline{t_5 t_6} \}$ or in a position obtained by rotating them.  
Consider the case of $\{ t_1 t_2 \cup \overline{t_1 t_2}, t_5 t_6 \cup \overline{t_5 t_6}\}$.  

If none of $t_2 t_3, t_4 t_5, t_6 t_7, t_8 t_1$ is an edge of a $4$-gon, then $F_4 \leq 2$,  which contradicts to $F_4 \geq  3$.  Thus, at least one of them is an edge of a $4$-gon.  Assume that $t_2 t_3$ is so.  Then $t_8 t_1 \cup \overline{t_1 t_2} \cup t_2 t_3 \cup \overline{t_3 t_8}$ is a $4$-gon.  Beside of this $4$-gon,  there are at least two $4$-gons.  Since one of them has $t_4 t_5$ or $t_6 t_7$ as an edge, it is the $4$-gon 
 $t_4 t_5 \cup \overline{t_5 t_6} \cup t_6 t_7 \cup \overline{t_7 t_4}$.  Then we have a tiling in 
Figure~\ref{fig-m08at-3}.  When one of $t_4 t_5, t_6 t_7, t_8 t_1$ is an edge of a $4$-gon, we have the same tiling.  For the other cases of (iii), we have a tiling by rotation. Thus the possible tilings are shown in the figure.   

\begin{figure}[ht]
\begin{center}
\resizebox{0.76
\textwidth}{!}{%
  \includegraphics{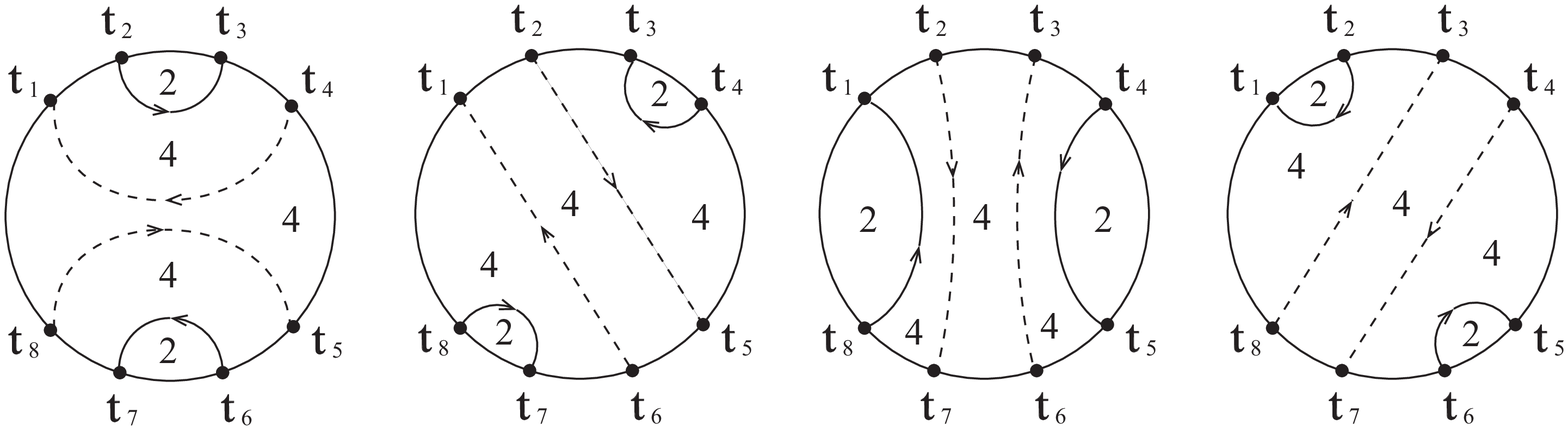}}
\caption{$F_2=2$} \label{fig-m08at-3}
\end{center}
\end{figure}

{\bf Case IV.} Suppose $F_2=1$. It follows from (\ref{equation-4tangle}) that $F_4 \geq 5$. Consider the case that the $2$-gon is $t_1 t_2 \cup \overline{t_1 t_2}$.  Edges $t_8 t_1, \overline{t_1 t_2}, t_2 t_3$ are edges of the same $n$-gon, say $A$, for some $n \geq 4$. 

(i) Suppose that $A$ is a $4$-gon.  Namely, 
$A=t_8 t_1 \cup \overline{t_1 t_2} \cup t_2 t_3 \cup \overline{t_3 t_8}$ is a $4$-gon in $G^\ast$.  
Besides $A$, there must be at  least four $4$-gons in $G^\ast$. Since each of which has 
one of $t_3 t_4, t_4 t_5, t_5 t_6, t_6 t_7, t_7 t_8$ as an edge.  Thus at least one of $t_3 t_4, t_7 t_8$ is an edge of a $4$-gon. Then $A':= t_3 t_4 \cup \overline{t_4 t_7} \cup t_7 t_8 \cup \overline{t_8 t_3}$ is a $4$-gon in $G^\ast$. Besides $A, A'$, there must be at  least four $3$-gons in $G^\ast$. Since each of which has 
one of $t_4 t_5, t_5 t_6, t_6 t_7$ as an edge. Then  $A'':= t_4 t_5 \cup \overline{t_5 t_6} \cup t_6 t_7 \cup \overline{t_7 t_4}$ is a $4$-gon in $G^\ast$.  Then $t_5 t_6$ is an edge of a $2$-gon and we have $F_4= 3$. This contradicts to $F_4 \geq 5$.  

(ii) Suppose that $A$ is a $6$-gon.  Since $F_4 \geq 5$ and since each $4$-gon has one of $t_3 t_4, t_4 t_5, t_5 t_6, t_6 t_7, t_7 t_8$ as an edge,  the tiling $G^\ast$ must be one of the tilings in 
Figure~\ref{fig-m08at-4}.

\begin{figure}[ht]
\begin{center}
\resizebox{0.80
\textwidth}{!}{%
  \includegraphics{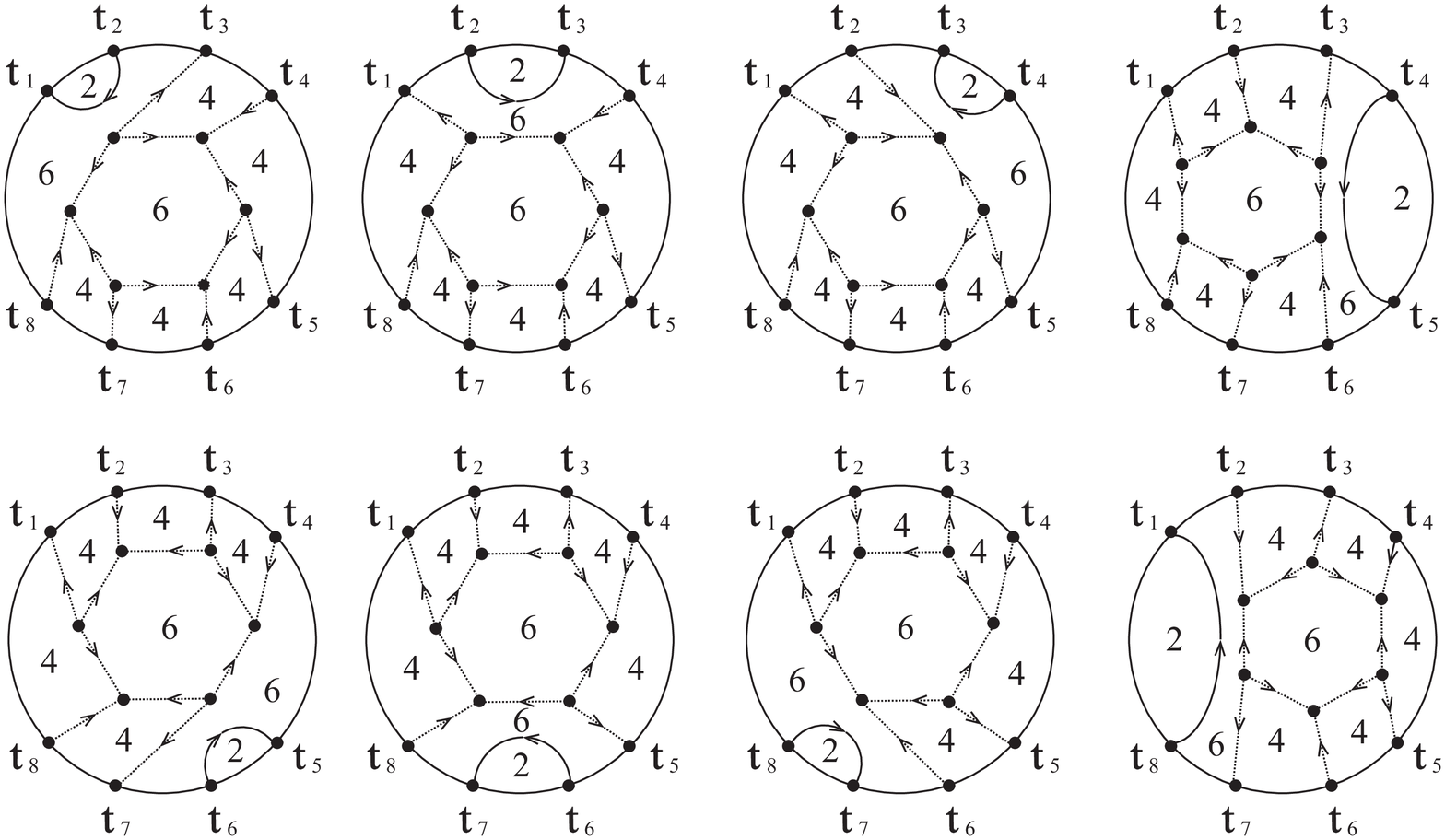}}
\caption{$F_2=1$} \label{fig-m08at-4}
\end{center}
\end{figure}

(iii) Suppose that $A$ is an $n$-gon with $n \geq 8$.  Then by (\ref{equation-4tangle}) we have $F_4 \geq 6$. 
On the other hand, since each $4$-gon has one of $t_3 t_4, t_4 t_5, t_5 t_6, t_6 t_7, t_7 t_8$ as an edge, $F_4 \leq 5$.  This is a contradiction.  Thus (iii) does not occur.  

{\bf Case V.} Suppose $F_2=0$.

(i) Suppose $F_k=0$ for all $k \geq 8$.   It follows from (\ref{equation-4tangle}) that $F_4 = 7$. 
The seven $4$-gons consecutively appear along $\partial D^2$.  There are seven edges that are edges of the seven $4$-gons and they are disjoint from  $\partial D^2$.  The seven edges are edges of the same $n$-gon for some $n \geq 8$. This contradicts the hypothesis.  

(ii) Suppose $F_k \neq 0$ for some $k \geq 8$.  It follows from (\ref{equation-4tangle}) that $F_4 \geq 8$. 
Since each $4$-gon has an edge in $\partial D^2$, the tiling $G$ is as in Figure~\ref{fig-m08at-5}. 
\end{proof}

\begin{figure}[ht]
\begin{center}
\resizebox{0.22\textwidth}{!}{%
  \includegraphics{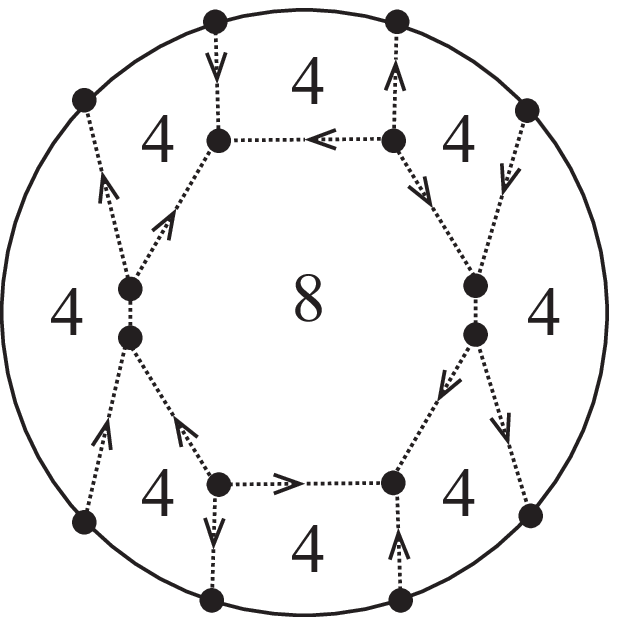}}
\caption{$F_2=0$} \label{fig-m08at-5}
\end{center}
\end{figure}

Now we are in a position to complete the proof of Lemma~\ref{4-tangle basis}.

\begin{proof}[\bf Proof of Lemma \ref{4-tangle basis}]
Let $\mathcal T$ be a $4$-tangle diagram with the boundary (a) in Figure~\ref{fig-bdary-t} such that there are no crossings, $2$-gons and $4$-gons cut that there are no connected components as diagrams in ${\rm Int}D^2$. Let $q_1, \ldots, q_8$ be the end points of $\mathcal T$ and let $\phi:I^2 \rightarrow D^2$ be a homeomorphism from $I^2$ onto a $2$-disk $D^2$ with $\phi(q_i)=t_i (1\leq i\leq 8)$. See Figure~\ref{fig-4tm08}. 
Putting $G= \phi(\mathcal T)$, we obtain the result from Lemma~\ref{8-move-tangle-lem-1}. 
\end{proof}

\begin{figure}[ht]
\begin{center}
\resizebox{0.50\textwidth}{!}{%
  \includegraphics{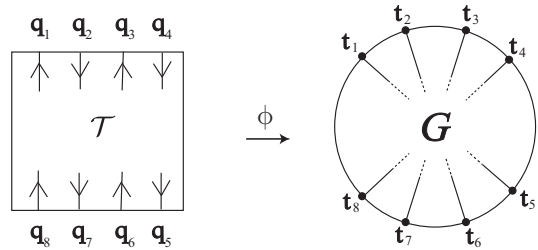}}
\caption{A homeomorphism $\phi:I^2 \rightarrow D^2$} \label{fig-4tm08}
\end{center}
\end{figure}


\section*{Acknowlegements}
The second and the third authors were supported by JSPS KAKENHI Grant Numbers 24244005 and 26287013. The fourth author was supported by Basic Science Research Program through the National Research Foundation of Korea (NRF) funded by the Ministry of Education, Science and Technology (2013R1A1A2012446).


\end{document}